\documentclass[11pt]{amsart}
\usepackage[margin=1.2in]{geometry}
\usepackage{amssymb}
\usepackage{amsmath}
\usepackage{mathtools}
\usepackage{color}
\usepackage{amsthm}
 \usepackage{lmodern}
\usepackage{tikz}
\usetikzlibrary{matrix}
  \usepackage{caption} 
 \input xy
\xyoption{all}
\usepackage{graphics}
\usepackage{amsfonts}
\usepackage{amssymb}
\usepackage{amscd}
\numberwithin{equation}{section}
\newtheorem{theorem}{Theorem}[section]
\newtheorem{cor}[theorem]{Corollary}
\newtheorem{proposition}[theorem]{Proposition}
\newtheorem{lemma}[theorem]{Lemma}
\newtheorem{prop}[theorem]{Proposition}
\theoremstyle{definition}
\newtheorem{definition}[theorem]{Definition}
\newtheorem{remark}[theorem]{Remark}
\newtheorem{example}[theorem]{Example}
\newtheorem{Conjecture}[theorem]{Conjecture}

\renewcommand\j{{\bf j}}

\newcommand{\bmu}{\bar{\mu}}

\newcommand\half{\tfrac{1}{2}}

\DeclareMathOperator{\ad}{ad}

\newcommand\rhat{\widehat\rho}

\newcommand\be{\beta}

\newcommand\w{\wedge}
\newcommand\g{\mathfrak g}

\newcommand\ga{\widehat{\mathfrak g}}
\newcommand\h{\mathfrak h}
\newcommand\ha{\widehat{\mathfrak h}}

\newcommand\n{\mathfrak n}

\newcommand\D{\Delta}
\renewcommand\l{\lambda}
\newcommand\Dp{\Delta^+}

\newcommand\Da{\widehat\Delta}
\newcommand\Pia{\widehat\Pi}
\newcommand\Dap{\widehat\Delta^+}
\newcommand\Wa{\widehat{W}}
\renewcommand\d{\delta}

\renewcommand\a{\alpha}

\newcommand{\Z}{\mathbb Z}
\newcommand\nat{\mathbb N}

\renewcommand\j{\mathfrak j}
\newcommand\s{\sigma}
\renewcommand\L{\Lambda}

\newcommand\e{\epsilon}
\newcommand\C{\mathbb C}
\newcommand\R{\mathbb R}
\newcommand\A{\mathcal  A}
\newcommand\si{\sigma}

\newcommand\What{\widehat W}

\renewcommand\ha{\widehat{\mathfrak h}}

\newcommand{\fg}{\mathfrak{g}}
\newcommand{\fh}{\mathfrak{h}}

\newcommand{\fn}{\mathfrak{n}}

\newcommand{\CC}{\mathbb{C}}

\newcommand{\RR}{\mathbb{R}}
\newcommand{\ZZ}{\mathbb{Z}}

\newcommand{\End}{\mbox{End}}

\newcommand{\Lie}{\mbox{Lie}}
\newcommand{\Res}{\mbox{Res}}
\newcommand{\tw}{\rm{tw}}

\newcommand{\langlerangle}{\langle \, . \, , 
. \, \rangle}
\newcommand{\parenthesis}{(\,\cdot\,|\,\cdot\,)}

\newcommand{\vac}{|0\rangle}
\newcommand{\bea}{\begin{eqnarray}}
\newcommand{\eea}{\end{eqnarray}}
\newcommand{\Ws}{W_k^{\min}(\g)}
\newcommand{\Wu}{W^k_{\min}(\g)}
\newcommand{\Zhu}{Zhu_R}
\begin{document}
\title{Unitarity of minimal $W$--algebras and their representations II: Ramond sector}
\author[Victor~G. Kac, Pierluigi M\"oseneder Frajria,  Paolo  Papi]{Victor~G. Kac\\Pierluigi M\"oseneder Frajria\\Paolo  Papi}

 \begin{abstract}
 In this paper we study unitary Ramond twisted representations of minimal $W$-algebras. We classify all such irreducible highest weight representations with highest weight which is not Ramond extremal (unitarity in the Ramond extremal case, as well as in the untwisted extremal case, remains open). We compute the characters of these representations and deduce from them the denominator identities for all superconformal algebras in the Neveu-Schwarz and Ramond sector. Some of the results rely on conjectures about the properties of the quantum Hamiltonian reduction functor in the Ramond sector.
\end{abstract}

\maketitle

\tableofcontents
\section{Introduction}

The problem of unitarity of highest weight representations of infinite-dimensional Lie algebras and superalgebras has been a hot 
topic in the 1980s and 1990s in Mathematics and Physics (see e.g. \cite{AKMP} for a list of references). With the advent of vertex algebra theory
it was realized that all these Lie (super)algebras arise naturally in the context of a special class of vertex algebras, called the simple quantum  affine
$W$-algebras, obtained by quantum Hamiltonian reduction from  affine vertex algebras of level $k$, starting from the datum $(\g,\mathfrak s)$ where 
$\g$ is a {\it basic} Lie superalgebra and $\mathfrak s$ is an $sl_2$ subalgebra. The most important among these $W$-algebras are the minimal ones, denoted by $\Ws$, corresponding to the ``minimal'' choice of $\mathfrak s$. They include all superconformal algebras, which play an important role in quantum physics.

This paper is a continuation of our paper \cite{KMP1} on classification of unitary minimal (quantum affine) $W$-algebras $\Ws$ and their (non-twisted) representations. 
In that paper we singled out necessary conditions for unitarity and divided the irreducible highest weight representations satisfying these conditions into {\it extremal} and {\it non-extremal ones}.
We proved  unitarity of non-extremal representations and conjectured that extremal representations are unitary, completing thereby the classification of unitary highest weight representations. Support to  our conjecture is given by 
the cases of $N=3$ and $N=4$ superconformal algebra.

In the present paper we study Ramond twisted unitary representations of minimal  $W$-algebras. We introduce (suitably modified)  notions of 
extremal and non extremal modules and we prove unitarity of non-extremal representations, under  (for a certain subclass) the conjectural validity 
of exactness properties for the twisted quantum Hamiltonian reduction functor. Again we conjecture that extremal representations are unitary, and again we have unitarity for extremal representations of the $N=3$ and $N=4$ superconformal algebra. Another indirect support to the properties of this functor that we are assuming  is given by classical combinatorial identites derived from  character formulas by using this functor.

Let us discuss in more  detail the setting and the results of the paper.
Let $\g$ be a simple finite-dimensional Lie superalgebra, over $\C$, with a reductive even part $\g_{\bar 0}$ and invariant
non-degenerate symmetric bilinear form $(\cdot |\cdot)$, with restriction to $\g_{\bar 0}$ non-degenerate.
 Let $\mathfrak s=Span\{e,x,f\}$, where $[e,f]=x, [x,e]=e, [x,f]=-f$ be an $sl_2$ subalgebra of $\g_{\bar 0}$. To the datum
 $(\g,\mathfrak s, k\in\C)$ one associates the universal quantum affine $W$-algebra $W^k(\g,\mathfrak s)$ of level $k$ by the quantum Hamiltonian reduction
 \cite{KRW}, \cite{KW1}. If $k$ is different from the critical level $k_{crit}$, the vertex algebra $W^k(\g,\mathfrak s)$ has a unique maximal ideal , and the quotient by this ideal is a simple  $W$-algebra, denoted by $W_k(\g,\mathfrak s)$.
 
 The {\it minimal} $W$-algebras correspond to the choice of $\mathfrak s$, called minimal, for which the
 $ad\,x$-eigenspace decomposition is of the form 
 \begin{equation}\label{1.1}
 \g=\g_{-1}\oplus\g_{-1/2}\oplus\g_0\oplus\g_{1/2}\oplus\g_1,\quad \text{ where $\g_{-1}=\C f,\,
 \g_{1}=\C e$}.\end{equation}
 We normalize the bilinear form $(\cdot |\cdot)$ by the condition $(x|x)=\half$. Then $k_{crit}=-h^\vee$, where 
 $h^\vee$ is half of the eigenvalue of the Casimir operator on $\g$. The decomposition \eqref{1.1} and the numbers $h^\vee$ are listed in \cite[Tables 1-3]{KW1}.

In order to define unitarity of a $W$-algebra, one needs a conjugate linear involution $\phi$ of $\g$, which fixes the subalgebra $\mathfrak s$ pointwise. Then, provided that $k\in\R$, $\phi$ induces a conjugate linear involution  of the vertex algebra $W^k(\g,\mathfrak s)$, and it descends to $W_k(\g,\mathfrak s)$.

It is proved in \cite[Proposition 7.2]{KMP1} that for minimal $\mathfrak s\subset \g$ and a non-collapsing level $k\in\R$, any conjugate linear involution $\phi$ of the vertex algebra $W^k(\g,\mathfrak s)$ is necessarily induced by a conjugate linear involution $\phi$ of $\g$ fixing $\mathfrak s$. (Recall that $k$ is called a collapsing level if $W_k(\g,\mathfrak s)$ is isomorphic to its affine part.) Moreover, it is proved in \cite[Proposition 8.9]{KMP1} that the vertex algebra $W_k(\g,\mathfrak s)$
is unitary only if the centralizer $\g^\natural$ of $\mathfrak s$ in $\g$ is a semisimple subalgebra of $\g_{\bar0}$, and the conjugate linear  involution $\phi$ is {\sl almost compact}, i.e. it restricts to a compact  involution of $\g^\natural$, and it leaves $\{e,x,f\}$ fixed. We write $\g^\natural=\oplus_i\g^\natural_i$, where $\g^\natural_i$ are simple components of $\g^\natural$.

We prove in  \cite{KMP1} that an almost compact conjugate linear involution of $\g$ exists if and only if $\g$ is from the following lists:
\begin{equation}\label{isunitary}
psl(2|2),\ spo(2|m) \text{ for }m\ge0, D(2,1;a)\text{ for }a\in \R,\ F(4),\ G(3);
\end{equation}
\begin{equation}\label{nonunitary}
sl(2|m)\text{ for }m\ge3,\ osp(4|m)\text{ for }m>2\text{ even},
\end{equation}
and it is essentially unique. Moreover, in these cases $\g$ admits a unque, up to conjugation, minimal $sl_2$-subalgebra $\mathfrak s$. We denote the corresponding minimal unversal $W$-algebra of level $k$ by $\Wu$, and its simple quotient by $\Ws$. 
Recall that, for $k\ne k_{crit}$,  the vertex algebra $\Wu$ is conformal with Virasoro field $L=\sum_{n\in\Z}L_nz^{-n-2}$, and it is strongly and freely generated by the operators $L_n$, $n\in\Z$, and the Fourier coefficients of the primary fields $J^{\{a\}}(z)=\sum_{n\in\Z} J^{\{a\}}_{n}z^{-n-1}$, $a\in\g^\natural$, of conformal weight 1, and $G^{\{u\}}(z)=\sum_{n\in\half+\Z}G^{\{u\}}_nz^{-n-\tfrac{3}{2}}$, $u\in\g_{-1/2}$, of conformal weight $\tfrac{3}{2}$ \cite[Theorems 4.1 and 5.1]{KW1}.

Note that, as in \cite{KMP1}, we exclude the case of $\g=spo(2|m)$, $m=0,1$, and $2$, since in these cases the $W$-algebra $\Wu$ is the universalVirasoro, Neveu-Schwarz, and $N=2$ vertex algebra, respectively, for which unitarity of non-twisted and twisted modules is well understood.

A non-degenerate Hermitian form $H$ on a module $M$ with finite-dimensional $L_0$ eigenspaces over a conformal vertex algebra is called {\sl $\phi$-invariant} if it defines an isomorphism of $M$ with its restricted dual. For $\Wu$ this is equivalent to the following  conditions \cite{DL}, \cite{KMP}:
$$
L_n^*=L_{-n},\ \left(J^{\{a\}}_{n}\right)^*=J^{\{\phi(a)\}}_{-n},\ \left(G^{\{u\}}_{n}\right)^*=G^{\{\phi(u)\}}_{-n}.
$$

We proved in \cite[Proposition 8.19]{KMP1} that, for $k\ne k_{crit}$, the minimal $W$-algebra $\Ws$ is not unitary for $\g$ from the list \eqref{nonunitary}, except when $\g=sl(2|m)$, $m\ge3$, and the level is the collapsing level $k=-1$. Furthermore, we proved in \cite[Corollary 11.2]{KMP1} that, for $\g$ from the list \eqref{isunitary}, the vertex algebra $\Ws$ is non-trivial   unitary for $k\ne k_{crit}$ if and only if $k$ lies in the {\sl unitary range}, given in the  following Table \ref{tabel0}, along with $k=k_{crit}$ and $k=k_0$ for which $\dim\Ws=1$:
\renewcommand{\arraystretch}{1.4}
\begin{center}
\begin{tabular}{c | c| c |c   }
$\g$&
unitary range&
$k_{crit}$& $k_0$\\
\hline
$psl(2|2)$&$-(\nat+1)$&$0$&$-1$\\
\hline
$spo(2|3)$&$-\tfrac{1}{4}(\nat+2)$&$-\half$&$-\half$\\
\hline
$spo(2|m),\,m\ge4$&$-\tfrac{1}{2}(\nat+1)$&$\tfrac{m}{2}-2$&$-\half$\\
\hline
$D(2,1;\tfrac{m}{n})$&$-\tfrac{mn}{m+n}\nat,\ m,n\in\nat\text{ coprime},\, (m,n)\ne(1,1)$&$0$& \text{none}\\
\hline
$F(4)$&$-\tfrac{2}{3}(\nat+1)$&$2$&$-\tfrac{2}{3}$\\
\hline
$G(3)$&$-\tfrac{3}{4}(\nat+1)$&$\tfrac{3}{2}$&$-\tfrac{3}{4}$
\end{tabular}
 \captionof{table}{\label{tabel0}}
\end{center}
 In our paper \cite{KMP1}, we also studied unitarity of irreducible highest weight $\Wu$-modules $L^W(\nu,\ell_0)$, where $\g$ is one of the Lie superalgebras from Table \ref{tabel0} (with the exception of $spo(2|m)$, $m\le 2$), and $k$ lies in the unitary range. These modules are parametrized by pairs $\nu\in (\h_\R^\natural)^*$ and $\ell_0\in\R$. We proved that unitarity of $L^W(\nu,\ell_0)$ holds if and only if the following condition holds: 
 \begin{enumerate}
 \item[(a)] the affine levels $M_i(k)$ for $\g^\natural_i$ are non-negative integers;
 \item[(b)] $\nu\in P^+_k=\{\text{dominant integral weights for }\g^\natural\text{ such that }\nu(\theta_i^\vee)\le M_i(k)\}$ where $\theta_i$ are highest roots of $\g_i^\natural$.
 \item[(c)]$\ell_0\ge A(k,\nu)$, where $A(k,\nu)$ is defined in \cite[formula (8.11)]{KMP1}, and $\ell_0=A(k,\nu)$ if $\nu$ is an extremal weight (i.e. $\nu(\theta_i^\vee)>M_i(k)+\chi_i$ for some $i$, $\chi_i$ being displayed in \cite[Table 2]{KMP1}), except that the unitarity of $L^W(\nu,\ell_0)$ when $\nu$ is an extremal weight and $\ell_0=A(k,\nu)$ is still an open question.
\end{enumerate}

 Actually in \cite{KMP1} we studied unitarity of the $\Wu$-modules; however it has been proved in \cite[Theorem 5.1]{AKMP} that any unitary $\Wu$-module descends to the simple $W$-algebra $\Wu$. 
 
 The study of unitarity of the Ramond twisted irreducible highest weight modules over the vertex algebra $\Wu$, where $k$ is in the unitary range, proceeds along similar lines.
 The main difference is that in the Ramond sector one has to consider separately two cases: when $\half\theta$ is not a root of $\g$, and when it is a root, where $\theta$ is the highest root of $\g$. In both cases the necessary conditions of unitarity are similar to the above conditions (a), (b), (c), except that in condition (c) the constant $A(k,\nu)$ is replaced by the one given by \eqref{Aknufirst}, which we denote here by $A^{\tw}$, and the notion of an extremal weight needs to be replaced by that of a Ramond extremal weight, defined by \eqref{Rextremal}. See Section \ref{necessary} for details.
 
 As in \cite[Section 10]{KMP1}, we find sufficient conditions of unitarity of Ramond twisted irreducible highest weight modules over $\Wu$ by using its free field realization, introduced in \cite[Theorem 5.2]{KW1} and the Ramond twisted version of the Fairlie type modification. As a result, we prove unitarity for $\ell_0$ larger than a certain constant $B$, defined by \eqref{BKnurhoR}, in the cases when $\nu$ is not Ramond extremal (see Section \ref{sufficient}).
 
It turns out that $B=A^{\tw}$ in the cases when $\theta/2$ is a root of $\g$ (see Lemma \ref{Aknusecond} (1)), which completes the proof of unitarity when $\nu$ is not Ramond extremal.

However, in the case when $\theta/2$ is not a root of $\g$, $B=A^{\tw}$ only for some very special weights $\nu$ (see Corollary \ref{thetahalfisaroot}). Generically one has that $B>A^{\tw}$, and we need  to use Proposition \ref{typical} on Euler-Poincar\'e characters, instead of   determinants of $\phi$-invariant Hermitian forms for twisted $\Wu$-modules \cite{KW}, as in \cite[Section 11]{KMP1} for the non-twsited sector. At this point we need to use Conjecture \ref{Arakawa}, which claims that Arakawa's results \cite{Araduke} on properties of the quantum Hamiltonian reduction functor can be extended to the Ramond twisted case. See Section \ref{UbetweenAB} for details.

Note that, due to Corollary  \ref{910}, analogous to that in \cite{AKMP}, all unitary irreducible non-twisted or twisted highest weight modules over $\Wu$ descend to $\Ws$.

In Section \ref{Explicit} conditions for unitarity of Ramond twisted irreducible highest weight modules over $\Wu$ are exhibited in all cases, except for the well known cases of $\g=sl_2$, $spo(2|1)$, and $spo(2|2)$, corresponding to Virasoro, Neveu-Schwarz, and $N=2$ vertex algebras.

In Section 11 we prove unitarity of Ramond extremal modules over the $N=3$ and $N=4$ vertex algebras. The analysis of extremal modules for the big $N=4$ superconformal algebra will appear in a forthcoming publication. For other unitary minimal $W$-algebras the problem of unitarity of Ramond extremal modules remains open.

In Section \ref{massless}, we compute the characters of the Ramond twisted irreducible highest weight modules $\Ws$, when $k$ is in the unitary range. As in the non-twisted case \cite{KMP1}, there are two cases to consider. In the first case, called {\sl massive}, the Ramond twisted $\Ws$-module is obtained by twisted quantum Hamiltonian reduction of typical modules over the corresponding affine Lie algebra, and in the second case, called {\sl massless}, from the maximally atypical ones. The corresponding character formulas are obtained by twisted quantum Hamiltonian reduction, using the properties conjectured in Conjecture \ref{Arakawa}, and they are given by Theorem \ref{characters} and Theorem \ref{charactersmassless} respectively.

In Section \ref{denominator}, using the character formulas for massless representations in the case of level $k_0$ when $\dim\, W_{k_0}^{\min}(\g)=1$, we find the denominator identities for $\Wu$. As a result, we recover the classical identities of Euler, Gauss and Ramanujan,  and find some new identities.

In the Appendix we discuss  a denominator identity for minimal $W$-algebras of Deligne series by exploiting a recent result \cite{BKK} about certain $\ga$-modules of negative integer level. 

Throughout the paper the base field is $\C$, and $\Z_+$ and $\nat$ stand for the set of non-negative and positive integers, respectively.
 
 \section{Twisted modules}\label{Twist}

We will denote by $p$ the parity in a vector superspace, and let $p(a,b)=(-1)^{p(a)p(b)}$.
Let $R$ be a Lie conformal superalgebra over $\C$ with infinitesimal translation operator $T$ and $\l$-bracket
\begin{equation}\label{b}
[a_\l b]=\sum_{j\in\ZZ_+}\tfrac{\l^j}{j!}a_{(j)}b.
\end{equation}
 Let   $\sigma$ be a diagonalizable automorphism of $R$.  We shall
always assume that all eigenvalues of $\sigma$
have modulus~$1$.  We have:
\begin{equation}
  \label{eq:1.6}
  R=\bigoplus_{\bar{\mu}\in \RR /\ZZ} R^{\bar{\mu}}, \hbox{ where }
     R^{\bar{\mu}} =\{ a \in R | \sigma (a) = e^{2\pi i\bar{\mu}}a
        \} \, .
\end{equation}
Here and further $\bar{\mu}$ denotes the coset $\mu + \ZZ$ of
$\mu \in \RR$.  Consider the subspace $\oplus_{\mu\in\R}(R^{\bar\mu}\otimes t^\mu)$  of $R[t^\R]$; it is $T\otimes 1 + 1\otimes\partial_t$-invariant. We associate to the pair $(R ,\sigma)$ the
$\sigma$-\emph{twisted} Lie superalgebra
\begin{equation}
  \label{eq:1.7}
  \Lie (R ,\sigma) =(\bigoplus_{\mu \in \RR} \quad
    (R^{\bmu} \otimes t^{\mu}))\big\slash Image (T \otimes 1 +
        1\otimes \partial_t)\,,
\end{equation}
 endowed with the following (well-defined) bracket, where $a_{(\mu)}$ stands for the image of $a
\otimes t^{\mu}$ in  $\Lie (R ,\sigma)$:
\begin{equation}
  \label{eq:1.2}
  [a_{(\mu)}, b_{(\nu)}]=\sum_{j \in \ZZ_+} \binom{\mu}{j}
     (a_{(j)}b)_{(\mu +\nu-j)},\quad \mu,\nu\in\R.
\end{equation}
A $\Lie(R,\s)$-module $M$ is said to be {\it restricted} if, for each $m\in M$, 
$$
a_{(\mu)}m=0\text{ for } \mu\gg0.
$$
Let $V(R)$ be the universal enveloping vertex algebra of $R$. By the universality property of $V(R)$, $\s$ extends to define an automorphism of $V(R)$. Since $V(R)$ is generated by $R$, it is clear that $\s$ is diagonalizable on $V(R)$ with modulus one eigenvalues.

If $V$ is a vertex algebra and $\s$ is a diagonalizable automorphism of $V$ with modulus one eigenvalues, then we write
$$
V=\oplus_{\bar\mu \in \RR /\ZZ}
V^{\bar\mu}
$$ to be its eigenspace decomposition. 
Recall that a $\sigma$-twisted
module $M$ over $V$ is a linear map $a \to Y^M(a,z) = \sum_{\mu
\in \bar\mu} a^M_{(\mu)} z^{-\mu-1}$ $(a \in V^{\bar\mu})$  where $a^M_{{(\mu)}}
\in \End M$ and for any $v \in M$, $a^M_{(\mu)} v =0$ if $\mu \gg
0$, satisfying
\begin{gather}
  \vac _{(\mu)}^M = \delta_{\mu,-1} I_M \, ,\label{vacuumaxiom}\\
    \sum_{j \in \Z_+} \binom{\mu}{j} (a_{(n+j)}b)^M_{(\mu+\nu-j)}v\label{Borcherds}\\
    = \sum_{j \in \Z_+} (-1)^j \binom{n}{j} (a^M_{(\mu+n-j)}
       b^M_{(\nu+j)} -p(a,b)(-1)^n b^M_{(\nu+n-j)}a^M_{(\mu+j)}) v\, ,\notag
      \\
         \label{eq:1.5}
 ( Ta)^M_{(\mu)} =-\mu a^M_{(\mu-1)}.
  \end{gather}
where $a \in V^{\bar\mu}$, $b \in V^{\bar\nu}$,  $n \in \Z$.
Specializing \eqref{Borcherds} to $n=0$ one obtains \eqref{eq:1.2}. It follows from \eqref{eq:1.2}  that a $\s$-twisted $V(R)$-module that satisfies \eqref{eq:1.5} is naturally a restricted $\Lie(R,\s)$-module. The construction given in \cite[\S 3]{Litwist} shows that the  converse also holds: if $M$ is a $\Lie(R,\s)$-module, let,  for $a\in R^{\bar\mu}$, $a_{(\mu)}^M$ to be the operator on $M$ given by the action of  $a_{(\mu)}$.  For completeness we give a detailed proof of Li's result  in a slightly more general setting  (since we consider also infinite order automorphisms). Define the quantum fields 
\begin{equation}\label{YM}
a^M(z)\equiv Y^M(a,z)=\sum_{\mu\in\bar\mu}a^M_{(\mu)}z^{-\mu-1},\quad a\in R^{\bar\mu}.
\end{equation}

\begin{proposition}\label{lii}
The assignment $a\mapsto Y^M(a,z)$ given by \eqref{YM}, extends to define the structure of a $\s$--twisted $V(R)$--module on $M$ such that \eqref{eq:1.5} holds for all $a\in V(R)$.
\end{proposition}
\begin{proof} Fix $\bar\mu\in \R/\ZZ$ and let 
$$F(M,\bar\mu)=\left\{\sum_{\mu\in\bar\mu}a^M_{(\mu)}z^{-\mu-1}\mid  a^M_{(\mu)}\in End(M),\ a^M_{(\mu)}m=0\text{ for $m\in M,\,\mu\gg 0$}\right\}.$$
If $a\in R^{\bar\mu}$, then $a^M(z)$ lies in $F(M,\bar\mu)$. 
For $a(z)\in F(M,\bar\mu)$ choose $\mu\in\bar\mu$ and define an $n$-product by setting
\begin{equation}
  \label{g-nprod}
  a(z)_{(n)} b(z) =\Res_{z_1} Res_{z_0}i_{z_1,z_0}\left(\frac{z_1-z_0}{z}\right)^{\mu}z_0^n X,\end{equation}
  where 
  
  $$X=\sum_{n\in\ZZ}z_0^{-n-1}\left(i_{z_1,z}(z_1-z)^na(z_1)b(z)-(-1)^{n}p(a,b)i_{z,z_1}(z-z_1)^nb(z)a(z_1)\right).$$
  As usual , $i_{z_0,z_1}$ stands for the expansion in the domain $|z_0|>|z_1|$.
  Note that $a(z)_{(n)} b(z) \in F(M, \bar\mu+\bar\nu)$. Similarly to \cite[Remark 3.8]{Litwist}, one shows that this definition does not depend on the choice of $\mu$. Recall from \cite[Definition 3.2]{Litwist} the definition of locality for twisted quantum fields and remark that 
Dong's Lemma holds in this setting (cf.  \cite[Proposition 3.9]{Litwist}). Hence the maximal local family $A$ containing $a^M(z), a\in R^{\bar\mu},\, \bar\mu\in\R/\ZZ$, is a vertex algebra and the map on $A$ defined by setting $\si(a(z))=e^{2\pi\sqrt{-1}\bar\mu}a(z)$ for $a(z)\in F(M,\bar\mu)\cap A$ is an automorphism of $A$. It is then clear that $M$ is a $\si$-twisted representation of $A$. 

It is therefore enough to show that $A$ is a quotient of $V(R)$:
  consider the special case $a(z_1)=Y^M(a,z_1), b(z)=Y^M(b,z)$, $a\in R^{\bar\mu}$, $b\in R$. Since $M$ is a twisted representation of $A$, \eqref{Borcherds} holds. We rewrite it as in \cite[(2.43)]{DK} using the twisted delta functions 
  $$\d_{\bar\mu}(z-w):=z^{-1}\sum_{\mu\in\bar\mu}(\tfrac{w}{z})^\mu.
  $$ Using standard properties of formal calculus, we obtain  
\begin{align}\notag X&=\sum_{n\in\ZZ}z_0^{-n-1}\left(i_{z_1,z}(z_1-z)^na(z_1)b(z)-p(a,b)i_{z,z_1}(z_1-z)^nb(z)a(z_1)\right)\\
\notag&=\sum_{n\in\ZZ}z_0^{-n-1}\left(\sum_{j \in \Z_+} Y^M (a_{(n+j)} b,z) \partial^j_{z}
\delta_{\bar\mu}(z_1-z)/j!\right)\\\notag
&=\sum_{n\in\ZZ}z_0^{-n-1}\left(\sum_{j \in \Z_+}  \sum_{m \in \bar{\gamma}}  Y^M (a_{(n+j)} b,z) \partial^j_{z}
z^m z_1^{-1-m}/j!\right)\\\notag
&=\left(\sum_{j \in \Z_+}  \sum_{m \in \bar{\gamma}}  Y^M (\sum_{n\in\ZZ}z_0^{-n-1} a_{(n+j)} b,z) \partial^j_{z}
z^m z_1^{-1-m}/j!\right)\\\notag
&=\left(\sum_{j \in \Z_+} z_0^j \sum_{m \in \ZZ}  Y^M (Y(a,z_0)b,z) \partial^j_{z}
z^{m+\gamma_a} z_1^{-1-m-\gamma_a}/j!\right)
\\&=Y^M (Y(a,z_0)b,z)\sum_{m\in\ZZ} \sum_{j \in \Z_+} \tfrac{ 1}{j!}z_1^{-1-m-\gamma_a}\partial^j_{z_0}\left(z+z_0\right)^{m+\gamma_a} _{|z_0=0}z_0^j
\notag\\&=Y^M (Y(a,z_0)b,z)\sum_{m\in\ZZ}z_1^{-1} i_{z,z_0}\left(\frac{z+z_0}{z_1}\right)^mi_{z,z_0}\left(\frac{z+z_0}{z_1}\right)^{\gamma_a}
\notag\\&= Y^M (Y(a,z_0)b,z) z^{-1}\sum_{m\in\ZZ}i_{z_1,z_0}\left(\frac{z_1-z_0}{z}\right)^mi_{z_1,z_0}\left(\frac{z_1-z_0}{z}\right)^{-\gamma_a}.\label{last}\end{align}  
 Using \eqref{last} we have 
\begin{align*}
{Y^M(a,z)}_{(j)}{Y^M(b,z)} &= \Res_{z_1} Res_{z_0}i_{z_1,z_0}\left(\frac{z_1-z_0}{z}\right)^{\gamma_a}z_0^j X\\
 &=\Res_{z_1} \Res_{z_0}z_0^j Y^M (Y(a,z_0)b,z) z^{-1}\sum_{m\in\ZZ}i_{z_1,z_0}\left(\frac{z_1-z_0}{z}\right)^m\\
  &= \Res_{z_1} \Res_{z_0}z_0^j Y^M (Y(a,z_0)b,z) z_1^{-1}\sum_{m\in\ZZ}i_{z,z_0}\left(\frac{z+z_0}{z_1}\right)^m\\
&=\Res_{z_0}z_0^j Y^M (Y(a,z_0)b,z) =Y(a_{(j)}b,z).\end{align*}
  We have proven that the map 
$a\mapsto Y^M(a,z)$ is a Lie conformal superalgebra homomorphism $R\to A$. By the universality property of $V(R)$ there is a vertex  algebra homomorphism $V(R)\to A$ extending $a\mapsto Y^M(a,z),\,a\in R$. 
 \end{proof}

\begin{example}\label{12} Let $A$ be a superspace with a non-degenerate skew-supersymmetric bilinear form $\langlerangle$. Let $R=(\C[T]\otimes A)\oplus\C K$ be the Lie conformal superalgebra with $\l$-bracket
$$
[a_\l b]=\langle a,b\rangle K,\ a,b\in A.
$$
Let $V(R)$ be the universal vertex algebra of $R$ and set $F(A)=V(R)/(K-\vac)$.
Let $\s$ be a linear diagonalizable map on $A$ with modulus one eigenvalues and such that $\langle \s(a),\s(b)\rangle=\langle a,b\rangle$. Write $A=\oplus _{\bar\mu\in \R/\ZZ}A^{\bar\mu}$ for the corresponding eigenspace decomposition. In this case 

\begin{align*}\Lie(R,\si) &=\left(\bigoplus_{\mu\in\R}(\C[T]\otimes A^{\bar\mu})t^\mu\oplus\C K[t^{\pm 1}]\right)/Image (T \otimes 1 +
        1\otimes \partial_t)\\
        &=\left(\bigoplus_{\mu\in\R}(\C[T]\otimes A^{\bar\mu})t^\mu\right)/Image (T \otimes 1 +
        1\otimes \partial_t)\oplus\C K.
        \end{align*}
        From now on we shall drop $\otimes$ sign.
        Consider the superspace
$$\widetilde A=\sum_{\mu\in \R} A^{\bar\mu}t^\mu,$$
and Lie superalgebra
$$
\widehat A^{\tw}=\widetilde A	\oplus \C K.
$$
with bracket
$$[ at^\mu+\a K,  bt^\nu+\beta K]= \d_{\mu+\nu,-1}\langle a,b\rangle K,\quad a\in A^{\bar \mu}, b\in A^{\bar \nu},\,\a,\beta\in \C.
$$
Then the map 
\begin{equation}\label{identif}(\frac{T^r}{r!}\otimes a)t^\mu\mapsto (-1)^r \binom{\mu}{r} a t^{\mu-r},\quad K\mapsto K
\end{equation} 
extends to  a Lie superalgebra  isomorphism $\Lie(R,\si)\cong  \widehat A^{\tw}$.

Extend $\langlerangle$ to $\sum_{\mu\in \R} A^{\bar\mu}t^\mu$ by
$$
\langle  at^\nu, bt^\mu\rangle=\delta_{\nu+\mu,-1}\langle a,b\rangle,
$$
and consider the corresponding Clifford algebra $Cl(\widetilde A,\langlerangle)$. Choose a maximal isotropic subspace $U$ of $A^{-1/2}$ and set
$$
\widetilde A^+=\left(Ut^{-1/2}\right)\oplus \sum_{\mu>-\tfrac{1}{2}} A^{\bar\mu}t^\mu.
$$
Let 
$$F(A,\s)=Cl(\widetilde A,\langlerangle)\big/(Cl(\widetilde A,\langlerangle)\widetilde A^+).
$$
We can extend the natural action of $\widetilde A$ on $F(A,\s)$ to $\widehat A^{\tw}$ by letting $K$ act by $I_{F(A,\si)}$.
Under  the identification \eqref{identif}, Proposition \ref{lii} gives a $\si$-twisted representation of $V(R)$ on $F(A,\si)$ with $K$ acting by  $I_{F(A,\si)}$, hence a $\si$-twisted  representation of $F(A)$.
\end{example}

\begin{example}\label{13}
Let $\g$ be either a simple  Lie superalgebra or an (even) abelian Lie algebra. Assume furthermore that $\g$  is equipped with an even  supersymmetric non-degenerate invariant bilinear  form $\parenthesis$. Fix $k\in\C$ and let $R=\C[T]\otimes \g\oplus \C K$ be the  Lie conformal algebra with $\l$-bracket
$$
[a_\l b]=[a,b]+\l k (a|b)K.
$$
The vertex algebra $V(R)/(K-k\vac)$ is called the universal affine vertex algebra of level $k$ associated to $\g$ and it is denoted by 
$V^k(\g)$. 
Let $\s$ be an automorphism of $\g$ with modulus one eigenvalues such that $(\s(a)|\s(b))=(a|b)$, and
let 
$$
\g=\oplus_{\bar\mu\in\R/\Z} \g^{\bar\mu},\ \text{where }\g^{\bar\mu}=\{a\in\g\mid \si(a)=e^{2\pi i\bar\mu}a\}
$$
be its eigenspace decomposition. Let
$$
\widetilde\g^{\tw}=\sum_{\mu\in \R}\g^{\bar\mu} t^\mu
$$
be the corresponding twisted loop algebra, and consider the  central extension
$${\widehat{\g}}^{\tw'}=\widetilde\g^{\tw}\oplus\C  K$$of $
\widetilde\g$ with bracket
\begin{equation}\label{braffine}
[at^\mu,bt^\nu]=[a,b]t^{\mu+\nu}+\mu\d_{\mu,-\nu}(a|b) K,\ a\in\g^{\bar\mu},\,b\in\g^{\bar\nu}.
\end{equation}
As in Example \ref{12}, we can identify $\Lie(R,\si)$ with $\widehat\g^{\tw'}$ via \eqref{identif}.

Let $(M,\pi_M)$ be a restricted ${\widehat{\g}}^{\tw'}$--module such that $K$ acts by $kI_M$ and define  $\s$--twisted fields
$$
Y^M(a,z)=\sum_{\mu\in\R}\pi_M(at^\mu)z^{-\mu-1},\text{ where $a\in \g^{\bar\mu}$.}
$$
Since \eqref{eq:1.2} in this case is \eqref{braffine},
we see, by Proposition \ref{lii}, that $M$ is a $\s$--twisted module over the vertex algebra $V^k(\g)$.

\end{example}

\section{Minimal $W$-algebra setup}\label{Setup}Let $\g$ be a basic simple finite-dimensional Lie superalgebra such that 
 \begin{equation}\label{goss}
\g_{\bar 0}=\mathfrak s\oplus \g^\natural.
\end{equation}
where  $\mathfrak s\cong sl_2$ and $\g^\natural$ is the centralizer of $\mathfrak s $ in $\g$. This corresponds to consider $\g$  as in  Table 2 of \cite{KW1}. 
Let $\{e,x,f\}$ be an $sl_2$-triple for $\mathfrak s$, i.e.  $\mathfrak s=span(e,x,f)$, and $[x,e]=e, [x,f]=-f, [e,f]=x$. 
Let
\begin{equation}\label{grad}
\g=\bigoplus_{j\in\half\ZZ}\g_j
\end{equation}
be the $ad\,x$-eigenspace decomposition of $\g$. Thus
\begin{displaymath}
  \fg =\CC f + \fg_{-1/2}+\fg_0
     +\fg_{1/2} +\CC e,
\end{displaymath}
with 
$$
\g_0=\C x\oplus \g^\natural.
$$
Note that our assumptions imply that $\g_{\pm 1/2}$ are purely odd. \par
We will also assume, as in \cite{KMP1}, that $\g^\natural$ is not abelian; this condition rules out   $\g=spo(2|m),\ m=0,1,2$. 
Since we are interested in unitary $W$-algebras, according to \cite{KMP1}, we may exclude $\g=sl(2|m)$ and $osp(4|m)$ 
with $m>2$ from consideration. Then 
\begin{equation}\label{gnatural}\g^\natural=\bigoplus_{i=1}^s\g^\natural_i\end{equation} is  the decomposition of $\g^\natural$ into the direct sum of  simple ideals (where $s\le 2$).

Recall that $\g$ carries an even invariant non-degenerate supersymmetric bilinear form $(\, . \, | \,. \, )$ that we normalize by requiring that $(x|x)=\half$. An important role is played by the following bilinear forms
$\langle \, . \, , \, . \, \rangle_{\text{ne}}$ on $\fg_{1/2}$ and $\langle \, . \, , \, . \, \rangle$  on $\fg_{-1/2}$:
\begin{equation}
  \label{eq:2.2}
  \langle a,b \rangle_{\text{ne}}=(f|[a,b]),\  \langle a,b \rangle=(e|[a,b]),
\end{equation}
which are symmetric and non-degenerate. Denote by $A_{ne}$ the vector superspace $\g_{-1/2}$ with the bilinear form $\langle \cdot,\cdot\rangle$ and by $A_{ch}$ the vector superspace $\Pi(\g_{<0}+\g^*_{<0})$ with the skewsupersymmetric bilinear form
given by pairing ($\Pi$ is the  parity reversing functor). We use 

 Let $h^\vee$ be the dual Coxeter number of $\g$.
Let $\mathcal C(\g)$ be the vertex algebra  $V^k(\g)\otimes F(A_{ch})\otimes F(A_{ne})$. Let $d\in \mathcal C(\g)$ be as in \cite[Section 1]{KW1}. Since $[d_\l d]=0$, $d_{(0)}^2=0$ on $\mathcal C(\g)$. The homology $(H_\bullet(\mathcal C(\g),d_{(0)})$ is the vertex algebra $\Wu$, called the {\it universal minimal  $W$-algebra  of level $k$}, associated to the pair $(\g,\mathfrak s)$ (cf \cite{KW1}).  If $k\ne - h^\vee$, this vertex algebra has a unique simple quotient $\Ws$, since $\Wu$, when $k\ne -h^\vee$, is a conformal vertex algebra with conformal vector $L$, given in \cite[(2.2)]{KW1}. Furthermore, this vertex algebra is strongly and freely generated by $L$, primary elements $J^{\{a\}}$, $a\in\g^\natural$, of conformal weight $1$, and primary odd elements $G^{\{u\}}$, $u\in\g_{-1/2},$ of conformal weight $3/2$ \cite[Theorem 5.1]{KW1}.

Fix an automorphism $\sigma$ of $\fg$ with the following three
properties:
\begin{align}
&\sigma (x) =x,\ \sigma (f) =f;\label{sigma1}\\
&(\si(a)|\si(b))=(a|b),\ a,b\in\g;\label{sigma2}\\\
&\sigma\text{ is diagonalizable and all its eigenvalues have modulus~$1$}.\label{sigma3}
\end{align}
 Then
$\si$ defines automorphisms of vertex algebras $V^k(\g)$, $F(A_{ch})$, $F(A_{ne})$, hence an automorphism  of $\mathcal C(\g)$. Since $\sigma(d)=d$, $\si$ induces an automorphism of $\Wu$ (and also of $\Ws$).

We want to apply the construction of  Example \ref{12} to $A_{ch}$ and $A_{ne}$. For this we
introduce the following $\tfrac{1}{2}\ZZ$-graded subalgebra of
$\fg$:
\begin{equation}
  \label{eq:2}
  \fg (\sigma) =\bigoplus_{j \in \frac{1}{2}\ZZ} \fg_{j}(\sigma)\, ,
  \hbox{ where } \fg_j (\sigma) = \{ a \in \fg_j |\sigma (a)
     =(-1)^{2j}a \} \, .
\end{equation}
The  $\frac{1}{2} \ZZ $-gradation 
(\ref{eq:2}) looks as follows:
\begin{displaymath}
  \fg (\sigma)=\CC f + \fg^{-\sigma}_{-1/2}+\fg^{\sigma}_0
     +\fg^{-\sigma}_{1/2} +\CC e.
\end{displaymath}

Fix a $\s$-stable Cartan subalgebra $\h^\natural$ of $\g^\natural$. Then $\h=\h^\natural\oplus\C x$ is a $\s$-stable Cartan subalgebra of $\g$. Define $\theta\in\h^*$ by $\theta(\h^\natural)=0$ and $\theta(x)=1$. Observe that $\theta$ is the weight of $e$ and  $(\theta|\theta)=2$.

Since
$\fg^{\sigma}_0=(\fg^\natural)^{\sigma} + \CC x$, it follows that
there exists an element $h_0 \in (\fh^\natural)^\s$ such that 
\begin{itemize}
\item the eigenvalues of $\ad h_0$
are real, 
\item $h_0$ is a regular element of $(\fg^\natural)^{\sigma}$, 
\item  the
$0$-{th} eigenspace of $\ad h_0$ on $\fg^{-\sigma}_{1/2}$
(resp. $\fg^{-\sigma}_{-1/2}$) is either $0$ or $\CC e_{\theta /2}$ (resp. $\CC
e_{-\theta /2}$). Here $e_{\theta /2}$ is a root vector of $\fg
(\sigma)$ and $\theta /2$ stands for the restriction of $\theta /2$ to $\fh^{\sigma}$. 
\end{itemize}
 Let $\fn(\sigma)_+ $ (resp. $\fn (\sigma)_-$) be
the span of all eigenvectors of $\ad h_0$ with positive
(resp. negative) eigenvalues and the vectors $f=e_{-\theta}$ and
$e_{-\theta/2}$ (resp. $e=e_{\theta}$ and $e_{\theta /2}$). Then 
\begin{equation}
  \label{eq:2.4}
  \fg (\sigma) = \fn (\sigma)_- \oplus \fh^{\sigma}
     \oplus \fn (\sigma)_+ .
\end{equation}
Set $\n_j(\s)_\pm=\n(\s)_\pm\cap \g_j(\s)$. Then  the following properties hold:

\begin{enumerate}
\item 
$\fn (\sigma)_{\pm}$ are isotropic with respect to $(\, . \, | \,. \, )$, and are  nilpotent subalgebras normalized by $\fh^{\sigma}$,

\item 
$f \in \fn (\sigma)_+$,

\item 
$\fn_{1/2}(\sigma)_+ $ is a
maximal isotropic subspace of $\fg_{1/2}(\sigma)$ with respect to
$\langle \, . \, , \, . \, \rangle_{\text{ne}}$ and $\fn_{1/2}(\sigma)_- $ is a
maximal isotropic subspace of $\fg_{1/2}(\sigma)$ with respect to $\langle \, . \, , \, . \, \rangle$.

\item 
$\fn_{1/2} (\sigma)_-$ is a direct sum of a maximal isotropic
subspace $\fn_{1/2}(\sigma)_-^{'}$ of  $\fg_{1/2}(\sigma)$ with respect to $\langle \,
. \, , \, . \, \rangle_{\text{ne}}$ and of a  subspace
$\fg^0_{1/2}(\sigma)$ (at most $1$-dimensional), normalized by $\fh^{\sigma}$.
\item $\fn_{-1/2} (\sigma)_+$ is a direct sum of a maximal isotropic
subspace $\fn_{-1/2}(\sigma)_+^{'}$ of  $\fg_{-1/2}(\sigma)$ with respect to $\langle \,
. \, , \, . \, \rangle$ and of a  subspace
$\fg^0_{-1/2}(\sigma)$ (at most $1$-dimensional), normalized by $\fh^{\sigma}$.
\end{enumerate}

We thus have the following decompositions:
\begin{equation}
\label{eq:2.5}
\fg_{1/2} (\sigma)= \fn_{1/2} (\sigma)_+ + \fg^0_{1/2}(\sigma)
    + \fn_{1/2}(\sigma)_-^{'} \, ,
\end{equation}
and
\begin{equation}
\label{eq:2.5bis}
\fg_{-1/2} (\sigma)= \fn_{-1/2} (\sigma)'_+ + \fg^0_{-1/2}(\sigma)
    + \fn_{-1/2}(\sigma)_- \, .
\end{equation}

Set \begin{equation}\label{eps} \epsilon (\sigma) =\dim \fg_{1/2}^0(\sigma)=\dim \fg_{-1/2}^0(\sigma).\end{equation} Then $\epsilon(\si)=0$ or $1$ and
 $\epsilon (\sigma ) \neq 0$ iff $\dim \fg_{1/2}
(\sigma)$ ($=\dim \fg_{-1/2}
(\sigma)$)  is odd.

Note also that in the decomposition (\ref{eq:2.5}),
$\fn_{1/2}(\sigma)_+$ (resp. $\fn_{1/2}(\sigma)'_-$) is the span of
all eigenvectors of $\ad h_0$ with positive (resp. negative)
eigenvalues and $\epsilon (\sigma) \neq 0$ iff $\theta /2$ is a
root of $\fg$  with respect to $\fh$ and $\sigma (e_{\theta /2})=-e_{\theta /2}$.

Following \cite{KW}, we let $F(A_{ne},\si)$ be the $\si$-twisted $F(A_{ne})$-module constructed as in Example \ref{12} with $U=\fn_{1/2}(\sigma)_+$ as maximal isotropic subspace of $A_{ne}^{-1/2}=\g_{1/2}^{-\s}$.
Similarly, we let $F(A_{ch},\si)$ be the $\si$-twisted $F(A_{ch})$-module constructed using $U=\fn_{1/2}(\sigma)_+\oplus (\fn_{1/2}(\sigma)_-)^*$ as maximal isotropic subspace of $A_{ch}^{-1/2}$.

Given a $\si$-twisted module $M$ of $V^k(\g)$, then 
$$
\mathcal C(M)=M\otimes F(A_{ch},\si)\otimes F(A_{ne},\si)
$$
is a  $\si$-twisted module over the vertex algebra  $\mathcal C(\g)$. It has the {\sl charge decomposition} 
$$
\mathcal C(M)=\oplus_{j\in\ZZ}\mathcal C_j(M),
$$
defined by 
$$
\text{charge\,}M=\text{charge\,}F(A_{ne},\s)=0
$$
and
$$
\text{charge\,}(\fn_{1/2}(\sigma)_+\oplus \fn_{1/2}(\sigma)_-^*)=-\text{charge\,}(\fn_{1/2}(\sigma)_+^*\oplus \fn_{1/2}(\sigma)_-)=1.
$$

Let 
$$
C(\g)=\oplus_{\bar\mu\in\R/\Z}C(\g)^{\bar\mu}
$$
be the eigenspace decompositions for $\s$, and 
$$
\Wu=\oplus_{\bar\mu\in\R/\Z}\Wu^{\bar\mu}
$$
the corresponding decomposition of its homology.

Since $[d_\l d]=0$, it follows from \eqref{eq:1.2} that $(d^{\tw}_{(0)})^2=0$.
Let
$
H(M)=\sum_{j\in\ZZ}H_j(M)
$
be the homology of the complex $(\mathcal C(M),d^{\tw}_{(0)})$, with the $\ZZ$-grading induced by the charge decompostion.

If $a\in\mathcal C(\g)$ then
$$[d^{\tw}_{(0)},a^{\tw}_{(\mu)}]=(d_{(0)}a)^{\tw}_{(\mu)}.
$$
In particular, if $d_{(0)}a=0$ and $d^{\tw}_{(0)}m=0$, then $d^{\tw}_{(0)}(a^{\tw}_{(\mu)}m)=0$ and, if $m=d^{\tw}_{(0)}m'$, then $a^{\tw}_{(\mu)}m=(-1)^{p(a)}d^{\tw}_{(0)}(a^{\tw}_{(\mu)}m')$. Moreover, if $a=d_{(0)}a'$, then
$a^{\tw}_{(\mu)}m=d^{\tw}_{(0)}((a')^{\tw}_{(\mu)}m)$. Therefore the quantum fields
$$
Y^{H(M)}(a,z)=\sum_{\mu\in\bar \mu}a^{\tw}_{(\mu)}z^{-\mu-1},\ a\in \Wu^{\bar\mu},
$$
are well defined and define the structure of a $\si$-twisted $\Wu$-module on $H(M)$. This module is called the quantum Hamitonian reduction of the $\s$-twisted $V^k(\g)$-module $M$.



  If $a\in \Wu^{\bar\mu}$ with conformal weight $\D_a$ and $M$ is a 
$\si$-twisted module, we write the field $Y^M(a,z)$ as
$$
Y^M(a,z)=\sum_{n\in {\bar\mu}-\D_a}a^M_nz^{-n-\D_a}.
$$
With this notation, \eqref{Borcherds} can be rewritten in its graded version:
\begin{align}
  \label{eq:2.36}
  \sum_{j \in \Z_+} &\binom{m+\Delta_a -1}{j}
     (a_{(n+j)} b)^M_{m+k}\\
\nonumber
&= \sum_{j \in \Z_+} (-1)^j \binom{n}{j}
(a^M_{m+n-j} b^M_{k+j-n}- p(a,b) (-1)^n b^M_{k-j} a^M_{m+j})\, ,
\end{align}
where  $a\in\Wu^{\bar\mu}, b\in\Wu^{\bar\nu}$, $m \in {\bar\mu}-\D_a$, $n \in \Z$, $k \in {\bar\nu}-\D_b$. Note that, putting $n=0$, \eqref{eq:2.36} becomes the (twisted) commutator formula
\begin{align}
  \label{eqtcf}
[a^M_{m}, b^M_{k}]=\sum_{j \in \Z_+} &\binom{m+\Delta_a -1}{j}
     (a_{(j)} b)^M_{m+k}.\end{align}

\section{Twisted highest weight modules over minimal $W$-algebras}
Recall from Example \ref{13}  the Lie superalgebras $\widetilde\g$ and $\widehat\g^{\tw'}$. Let $D=-L^{\fg ,\tw}_0$.  Recall \cite{KW} that we have $(a \in
\fg^{\bar{\mu}})$:
\begin{displaymath}
  [D,at^{\mu}] =\mu (at^{\mu}) \, , \, [D,K]=0\, .
\end{displaymath}
As usual, we shall consider the extension
\begin{displaymath}
\widehat\g^{\tw}=\widehat\g^{\tw'}\rtimes \CC D 
\end{displaymath}
 of $\widehat\g^{\tw'}$.
 The decomposition \eqref{eq:2.4} induces a triangular
decomposition of the Lie superalgebra $\widehat\g^{\tw}$:
\begin{equation}
  \label{eq:2.6}
\widehat\g^{\tw}= \widehat{\fn}_- \oplus \widehat{\fh}\oplus
      \widehat {\fn}_+ \, , \\
\end{equation}
where
\begin{eqnarray}
   \label{eq:hhat}
      \widehat{\fh} &=&\h^\si\oplus \C K\oplus \C D\, , \\
   \label{eq:2.8}
      \widehat{\fn}_+ &=& \sum_{j \in \frac{1}{2}\ZZ}
          (\fn_j (\sigma)_+ t^{-j} +
          \sum_{\substack{\mu \in \RR\\j+\mu >0}} \fg^{\bmu}_j
           t^{\mu})\, , \\
   \label{eq:2.9}
      \widehat{\fn}_- &=& \sum_{j \in \frac{1}{2} \ZZ}(\fn_j(\sigma)_-
          t^{-j} + \sum_{\substack{\mu \in \RR \\
             j+\mu<0}}
            \fg^{\bmu}_j  t^{\mu}) \, .
\end{eqnarray}

 Recall that, given a triangular decomposition
(\ref{eq:2.6}), a \emph{highest weight module} over the Lie
superalgebra $\widehat\g^{\tw}$ with highest
weight $\widehat\Lambda \in \widehat\fh^*$ is a $\widehat\g^{\tw}$-module $M$
which admits a non-zero vector $v_{\widehat{\Lambda}}$ with  the properties:

\begin{enumerate}
\item 
 $hv_{\widehat{\Lambda}} =\widehat{\Lambda} (h) v_{\widehat{\Lambda}}$, $h \in \widehat{\fh}$,

\item 
$\widehat{\fn}_+ v_{\widehat{\Lambda}}=0$,

\item 
$U(\widehat{\fn}_-)v_{\widehat{\Lambda}}=M$.
\end{enumerate}
Note that a highest weight module $M$ is graded by the eigenspace decomposition corresponding to the action of $D$. Since, by \eqref{eq:2.9}, the eigenvalues of the action of $D$ have real parts bounded above, it is clear that a highest weight module $M$ is restricted. Moreover $M$ has level $k$ if and only if 
\begin{equation}\label{levelk}
\widehat \L(K)=k.
\end{equation} 
In particular the highest weight modules of highest weight $\widehat\L$  such that \eqref{levelk} holds are $\si$-twisted $V^k(\g)$-modules.

Let $\{u_i\}_{i\in S}$ be a basis of $\g$, compatible with the decomposition \eqref{grad}, where $S$ is the index set. For $j\ne0$ let $S_j$ denote the subset of indices of $S$ which corresponds to a basis of $\g_j$, and denote by $S'\subset S$ the subset of indices of the part of the basis  
$\{ u_i \}_{i \in S}$ of $\fg$, which is a basis of $\fg \mod \fh^{\sigma}$.
Let
\begin{equation}
  \label{eq:3.1}
  s_{u_i} =\min \{ n |\, u_i  t^n \hbox{ is non-zero and lies in }
      \hat{\fn}_+ \}\,\hbox{for}\,i\in S' \, , s_{h}=1\, \hbox{for}\, h\in \fh^{\sigma}.
\end{equation}
Since, by \eqref{sigma1}, each summand $\fg_j$ of the gradation \eqref{grad}  is
$\sigma$-invariant, we have its $\sigma$-eigenspace
decomposition:
\begin{displaymath}
   \fg_j = \oplus_{\bmu \in \RR /\ZZ}
\fg^{\bmu}_j, \hbox{where}\, \fg^{\bmu}_j =\{ a \in \fg_j| \sigma (a)
=e^{2\pi i \bmu}a \}.
\end{displaymath}
Hence for a basis element $u_i \in
\fg^{\bmu_i}_{m_i}$ we can rewrite formula (\ref{eq:3.1}) for
$s_i =s_{u_i}$ ($i \in S'$) as follows:
\begin{eqnarray}
  \label{eq:3.2}
  s_i = \left\{
    \begin{array}{ll}
      \min \{ n \in \bmu_i |\,  n>-m_i \} \hbox{ if }
          & u_i \not\in \fn (\sigma)_+ \, , \\
       -m_i \hbox{ if } u_i \in \fn (\sigma)_+ \, .
    \end{array}\right.
\end{eqnarray}
It is easy to see that for a dual basis element $u^i \in
\fg^{-\bmu}_{-m_i}$ we have for $s^i =s_{u^i}$:
\begin{equation}
  \label{eq:3.3}
  s^i = 1-s_i \hbox{ for all }i \in S' \, .
\end{equation}

We extend this definition to
$F(A_{ne},\si)$ and $F(A_{ch},\si)$ as follows. Let us relabel the basis $\{u_i\}_{i\in S_{-1/2}}$ of $A_{ne}=\g_{-1/2}$ as $\{\Phi_i\}_{i\in S_{-1/2}}$ and the pair of dual bases
$(\{u_i\}_{i\in S_{<0}}, \{u_i^*\}_{i\in S_{<0}}),$ of $A_{ch}=\g_{<0}\oplus \g^*_{<0}$ as $\{\varphi_i\},\{\varphi_i^*\}$. Then
\begin{equation}
  \label{eq:3.4}
  s_{\Phi_i} = s_i  \,(i \in S_{1/2}) \, , \, s_{\varphi_i}=s_i \, , \,
     s_{\varphi^*_i}=1-s_i \,\, (i \in S_+)\, .
\end{equation}
It is easy to see that we have
\begin{equation}
  \label{eq:3.5}
  s_{\Phi_i}=\mp 1/2  \hbox{ if }\Phi_i \in \fn_{1/2}
     (\sigma)_{\pm}\, , \, |s_{\Phi_i}|<1/2\hbox{ otherwise.}
\end{equation}
\begin{equation}
  \label{eq:3.6}
  s_{\Phi_i} + s_{\Phi^i} = \delta_{i,i_0}\, , \, \hbox{where}\,\, \langle \Phi_{i_0},\Phi_{i_0}
 \rangle_{ne}  \neq 0.
\end{equation}

For a $\si$-twisted $\Wu$-module $M$, a vector $m\in M$ is called cyclic if  polynomials in the operators $J^{\{ a \} ,\tw}_n$, with $a\in \g^\natural,\, n\in\ZZ,$ $G^{\{ v \} ,\tw}_n,$ with $v\in\g_{-1/2}, n\in \half+ \ZZ,$ and $L_n^{\tw}$ with $n\in\ZZ$
applied to  $m$ span $M$. A vector $m\in M$ such that there are $\ell_0\in\C$ and $\l\in((\h^\natural)^\si)^*$ for which
\begin{align}
\label{m1}
&L_0^{\tw}(v_{\lambda,\ell_0})=\ell_0 v_{\lambda,\ell_0},
\\& J^{\{ a\} ,\tw}_0 v_{\lambda,\ell_0} =
     \lambda (a) v_{\lambda,\ell_0} \hbox{ if } a \in (\fh^\natural)^{\sigma},\label{m2}
\end{align}
is called a weight vector and  the pair $(\l,\ell_0)$ is called the weight of $m$.

We define a highest weight module for $\Wu$ as follows: 

\begin{definition}
 A $\s$-twisted $\Wu$-module $M$ is called a
\emph{highest weight module} of highest weight  $(\lambda,\ell_0)$ if there exists  a cyclic weight  vector $v_{\lambda,\ell_0
} \in M$ of weight  $(\lambda,\ell_0)$ such that 
\begin{align}
  &J^{\{ a \} ,\tw}_m v_{\lambda,\ell_0}=G^{\{ v \} ,\tw}_m v_{\lambda,\ell_0} =L^{\tw}_m v_{\lambda,\ell_0} =
  0 \text{ if $m>0$,}\label{m3}\\
 &J^{\{ a \} ,\tw}_0 v_{\lambda,\ell_0} =
  0 \text{ if 
     $a \in \fn_0 (\sigma)_+$,}\label{m4}\\
     &G^{\{ u \},\tw}_0 v_{\lambda,\ell_0} =
  0 \text{ if 
     $u \in \fn_{-1/2} (\sigma)'_+$}.\label{m5}
\end{align}
The vector $v_{\lambda,\ell_0}$ is called a highest weight vector.
\end{definition}


Let  $V$ be a conformal vertex algebra strongly generated by elements
$\{ J^{\{ i \}}\}_{i \in I},$ where $J^{\{ i \}}$ has conformal
weight $\Delta (i) \in \RR$. Let $M$ be a $\si$-twisted positive energy module (i.e. the eigenvalues of $L_0^{\tw}$ are bounded below). Write, for shortness,  $J^{\{ i \}}_m$ instead of  $(J^{\{ i \}})^M_m$. Then, by  the commutator formula 	\eqref{eqtcf}, we have:
\begin{equation}
  \label{eq:6.3}
  [J^{\{ i \}}_m \, , \, J^{\{ j \}}_n]=
  \sum_{\vec{s} , \vec{t}} c^{ij}_{m,n} (\vec{s} ,\vec{t})
  (J^{\{ s_1 \}}_{t_1} \,  J^{\{ s_2 \}}_{t_2} \ldots \,),
\end{equation}
where $c^{ij}_{m,n}\in\C$ and for each term of this sum we have:
\begin{equation}
  \label{eq:6.4}
  t_r \in \ZZ -\Delta (s_r) \, , \, \sum_r t_r =m+n
     \hbox{  and  } t_1 \leq t_2 \leq \ldots \, .
\end{equation}

Denote by $\A$ the unital associative superalgebra generated by 
$J^{\{ i \}}_m$ $(i \in I , m \in \ZZ-\Delta (i))$ inside $End(M)$. 
Note that, though the sum in the R.H.S. of \eqref{eq:6.3} may not be finite, by \eqref{eq:6.4} and the fact that $M$ is a positive energy $V$-module, 
the R.H.S. of \eqref{eq:6.3} makes sense as an element of $End(M)$.

  Let $\tilde{\A}_-$, $\tilde{\A}_+$ and
$\tilde{\A}_0$ be the subalgebras of $\A$ generated by the $J^{\{
  i \}}_m$ with $m<0$, $m>0$ and $m=0$, respectively.   It follows
from (\ref{eq:6.3}) and (\ref{eq:6.4}) that
\begin{equation}
  \label{eq:6.5}
  \A = \tilde{\A}_- \tilde{\A}_0 \tilde{\A}_+ \, .
\end{equation}
%

The previous discussion proves the following result.
\begin{lemma}\label{basisWu}
Let  $M$ be a $\s$-twisted highest weight $\Wu$-module  and let  $v$ be a highest weight vector for $M$.
Then $M$ is spanned  by 
\begin{equation}\label{basisW}
\{P_{j_1}P_{j_2}\cdots P_{j_t}v\mid j_1<j_2\cdots<j_t\le 0\}\text{ with $P_j\in\mathcal M_j$, where}
\end{equation}
\begin{align*}
&\mathcal M_j=\{(J^{\{a_{i_1}\}}_{j})^{m_1}\cdots (J^{\{a_{i_t}\}}_{j})^{m_t}G^{\{u_{i_1}\}}_{j}\cdots G^{\{u_{i_t}\}}_{j}L_{j}^k\},\ \text{if $j<0$,}\\
&\mathcal M_0=\{(J^{\{a_{i_1}\}}_{0})^{m_1}\cdots (J^{\{a_{i_t}\}}_{0})^{m_t}G^{\{u_{i_1}\}}_{0}\!\!\cdots G^{\{u_{i_t}\}}_{0}\!\!: a_{i_s}\in \n_0(\s)_-,
u_{i_r}\in(\n_{-1/2}(\si)_-\oplus \g^0_{-1/2}(\si))\}.
\end{align*}
\end{lemma}
We say that a highest weight $\Wu$-module $M$ of highest weight $(\l,\ell_0)$ is a {\it Verma module} (denoted by $M^W(\nu,\ell_0)$) if the elements in \eqref{basisW} form a basis of $M$. 
The following proposition summarizes various results  proven in  Proposition 3.1, Theorem 3.1,  and Proposition 4.1 of \cite{KW}.

\begin{proposition}\label{Hhw}
Let $M$ be a $\si$-twisted  highest weight module over $V^k(\g)$ with highest weight $\widehat{\Lambda}$  and highest weight vector $v_{\widehat\L}$. Set $\L=\widehat \L_{|\h^\si}$ and
\begin{equation}\label{g12}
  \gamma'=\frac{1}{2}\sum_{\alpha \in S'} (-1)^{p(\a)}
      s_{\alpha} \alpha \, , \quad
  \gamma_{1/2} =\frac{1}{2}\sum_{\alpha \in S_{1/2}}
     (-1)^{p(\a)}s_{\alpha} \alpha \, .
\end{equation}
Then
\begin{enumerate}
\item $d_0^{\rm tw}(v_{\widehat\L}\otimes 1\otimes 1)=0$,
\item if the homology class $[v_{\widehat\L}\otimes 1\otimes 1]$ is non-zero, then it  is a highest weight vector of the highest weight module $\Wu [v_{\widehat\L}\otimes 1\otimes 1]$ whose weight is $(\l,\ell)$, where

\begin{equation}\label{ellgen}
 \ell = \frac{1}{2(k+h^\vee)} ((\L | \L)
  -2(\L|\gamma'))
  -\L(x) + s_{fg} + s_{gh}\, ,
\end{equation}
with
\begin{displaymath}
  s_{fg}=-\frac{k}{4(k+h^\vee)} \sum_{\alpha \in S'}
     (-1)^{p(\a)} s_{\alpha} (s_{\alpha}-1)\, , \quad
     s_{gh}=\frac{1}{4}
     \sum_{\alpha \in S_{1/2}} (-1)^{p(\a)}s^2_{\alpha},
\end{displaymath}
and 
\begin{equation}\label{lambdagen}
 \l=(\L- \gamma_{1/2})_{|\h^\natural}.
\end{equation}
\item If $M(\widehat \L)$ is a $\s$-twisted Verma module over  $V^k(\g)$, then $H_j(M(\widehat \L))=0$ for $j\ne0$ and $H_0(M(\widehat \L))$ is  the $\s$-twisted Verma module over $\Wu$ of highest weight $(\l,\ell)$ given by \eqref{ellgen}, \eqref{lambdagen}.
\end{enumerate}
\end{proposition}
%

%



\section{The Zhu algebra in the Ramond sector}\label{SectionZhu}
Let $\sigma_R$ be the automorphism of $\g$ given by $\sigma_R(a)=(-1)^{p(a)}a$, which clearly satisfies the properties \eqref{sigma1}, \eqref{sigma2}, \eqref{sigma3}.  In this case
$$\g(\sigma_R)=\g,\quad \g^{\si_R}=\g_{\bar 0}=\g^\natural \oplus \C x,\quad \g_{\pm 1/2}^{-\si_R}=\g_{\pm 1/2},\quad \h^{\si_R}=\h=\h^\natural \oplus \C x.
$$

The main purpose of this section is to compute  the Zhu algebra $Zhu_{\si_R}(\Wu)$, which will be denoted $\Zhu$ for short, define its Verma and irreducible highest weight modules, and check the existence of  an invariant even Hermitian form on Verma modules.

Note  that the grading induced by $\si_R$ is the same as the $\half \ZZ/\ZZ$-grading induced by $L_0$ as described in Example 2.12 of \cite{DK}. It follows that  the Zhu algebra $\Zhu$ is the algebra $Zhu_{L_0}(\Wu)$ described in \cite[Section 7]{KMP} which we now recall. Let $\pi_Z: \Wu\to \Zhu$ be the canonical projection; then the map  $\g^\natural\oplus \g_{-1/2}\oplus \C L\to \Zhu $ defined by $a\mapsto \pi_Z(J^{\{a\}})$, $a\in \g^\natural$, $v\mapsto \pi_Z(G^{\{v\}})$, $v\in \g_{-1/2}$, $L\mapsto \pi_Z(L)$ is a linear isomorphism onto a set of generators.

Moreover the commutation relations among the generators are as follows (here $[\cdot,\cdot]_\g$ denotes the bracket in $\g$, while $[\cdot,\cdot]$ is the bracket in $\Zhu$.
\begin{enumerate}
\item $L$ is a central element,
\item $[a,b]=[a,b]_\g$ if $a,b\in\g^\natural$,
\item $[a,v]=[a,v]_\g$ if $a\in \g^\natural$ and $v\in\g_{-1/2}$,
\item If $u,v\in\g_{-1/2}$ then  \begin{align}\label{Zhu4}
[u,v]=&\langle u,v\rangle\left(\sum_{\a=1}^{\dim\g^\natural} a^\a*a_\a-2(k+h^\vee)L-\half p(k)\right)\\
&+\sum_{\a,\be=1}^{\dim\g^\natural}\langle[a_\a,u]_\g,[v,a^\be]_\g\rangle (a^\a*a_\be+a_\be*a^\a), \notag
\end{align}
\end{enumerate}
where $p(k)$ is a monic quadratic polynomial defined in \cite{AKMPP} and $\{a_\a\},\,\{a^\a\}$ are dual bases of $\g^\natural$. By (2) and (3) above, we can drop the subscript $\g$ from the bracket $[\cdot,\cdot]_\g$.
Let $L'=2(k+h^\vee)L+\half p(k)$. Then $\Zhu$ is independent of $k$ if $k\ne-h^\vee$. 

Let 
\begin{align*}&m_+=\dim\n_{-1/2}(\si_R)'_+,\,m_-=\dim\n_{-1/2}(\si_R)_-+\e(\sigma_R),\\&n= \dim\n_{0}(\si_R)_-=\dim\n_{0}(\si_R)_+, r=\dim \h^\natural.\end{align*}
Choose  $\{v^+_i\mid i=1,\ldots,m_+\}$  (resp. $\{v^-_i\mid i=1,\ldots,m_-\}$) as a basis of $\fn_{-1/2}(\sigma_R)'_+$ (resp- $\fn_{-1/2}(\sigma_R)_-\oplus
\g^0_{-1/2}(\sigma_R)$). Also assume  that $ v^-_{m_-}\in \g^0_{-1/2}(\si_R)$ when $\epsilon(\s_R)=1$   (see \eqref{eq:2.5}). Similarly  choose 
$\{a^\pm_i\mid i=1,\ldots n\}$ bases of $\n_0(\si_R)_\pm$ and $\{h_i\mid i=1,\ldots r\}$ a basis of $\h^\natural$. 
If 
$$
\mathbf p=(p_1,\ldots, p_{m_\pm})\in\{0,1\}^{m_\pm},\ \mathbf q=(q_1,\ldots, q_n)\in\Z_+^n,\ \mathbf k =(k_1,\ldots,k_r)\in\Z_+^r,
$$
given $\l\in (\h^\natural)^*$, set
$$
a_\pm^{\mathbf q}=(a^\pm_1)^{q_1}\cdots (a^\pm_{n})^{q_n},\ v_\pm^{\mathbf p}=(v^\pm_1)^{p_1}\cdots (v^\pm_{m_\pm})^{p_{m_\pm}},
$$
$$
h(\l)^{\mathbf k}=(h_1-\l(h_1))^{k_1}\cdots (h_{r}-\l(h_r))^{k_{r}}.
$$

 Theorem 3.25 of \cite{DK} implies that a basis of $\Zhu$ is given by
 $$
 \{a_-^{\mathbf q_-}*v_-^{\mathbf p_-}*h(0)^{\mathbf k}*v_+^{\mathbf p_+}*a_+^{\mathbf q_+}*L^k\},
 $$
 where $*$ denotes the product in $\Zhu$ \cite{DK}. It follows that, for all $\l\in (\h^\natural)^*$ and $\ell_0\in\C$, the set 
\begin{equation}\label{BLell}
 \mathcal B_{\l,\ell_0}=\{a_-^{\mathbf q_-}*v_-^{\mathbf p_-}*h(\l)^{\mathbf k}*v_+^{\mathbf p_+}*a_+^{\mathbf q_+}*(L-\ell_0)^{k}\}.
\end{equation}
 is also a basis of $\Zhu$.
 
 By a  $\Zhu$-module we mean a representation of $\Zhu$ as an associative superalgebra.
A $\Zhu$-module $M$ is called a highest weight module of highest weight $(\l,\ell_0)$ if $M$ is generated by a vector $v_{\l,\ell_0}$ such that  
\begin{align}\label{zh1}
&Lv_{\lambda,\ell_0}=\ell_0 v_{\lambda,\ell_0},
\\& a v_{\lambda,\ell_0} =
     \lambda (a) v_{\lambda,\ell_0} \hbox{ if } a \in \fh^\natural\label{zh2}, \\ 
     &a v_{\lambda,\ell_0} =
  0 \text{ if 
     $a \in \fn_0 (\sigma_R)_+$},\label{zh3}\\
    &u v_{\lambda,\ell_0} =
  0 \text{ if 
     $u \in \fn_{-1/2} (\sigma_R)'_+$.}\label{zh5}
\end{align}

\begin{remark}\label{basisMhw}
Since the set $\mathcal B_{\l,\ell_0}$ is a basis of $\Zhu$ it is clear that,   if $M$ is a highest weight vector of highest weight $(\l,\ell_0)$ and $v_{\l,\ell_0}$ is a highest weight vector, then $M$ is spanned by 
\begin{equation}\label{Vermabasis}
\{a_-^{\mathbf q_-}*v_-^{\mathbf p_-}\cdot v_{\l,\ell_0}\}.
\end{equation}
Set $S^{0}_{-1/2}=\emptyset$ if $\epsilon(\si_R)=0$ and $S^{0}_{-1/2}=\{m_-\}$ if $\epsilon(\si_R)=1$.
Since the set in \eqref{Vermabasis} is a set of linear generators for $M$,  the weight space of $M$ of weight $(\l,\ell_0)$ is
$$
M_{\l,\ell_0}=span(v_{\l,\ell_0},v^-_\gamma v_{\l,\ell_0}\mid\gamma\in S^{0}_{-1/2}).
$$
Since $\h^\natural\oplus \C L$  is even we have 
$$
M_{\l,\ell_0}=(M_{\l,\ell_0})_{\bar0}\oplus (M_{\l,\ell_0})_{\bar1}.
$$
Decompose $v_{\l,\ell_0}$ accordingly:
$$
v_{\l,\ell_0}=(v_{\l,\ell_0})_{\bar0}+(v_{\l,\ell_0})_{\bar1}.
$$
If $v^-_\gamma v_{\l,\ell_0}=0$ then $v_{\l,\ell_0}$ is either even or odd.
If $v^-_\gamma v_{\l,\ell_0}\ne 0$ then there is $\bar i $ such that $v^-_\gamma (v_{\l,\ell_0})_{\bar i}\ne0$. Then 
\begin{equation}\label{MMM}
M_{\l,\ell_0}=span((v_{\l,\ell_0})_{\bar i},v^-_\gamma (v_{\l,\ell_0})_{\bar i}).
\end{equation}
In particular $(v_{\l,\ell_0})_{\bar i}$ generates $M$. The outcome is that we can always assume that the highest weight vector is either even or odd. From now on this will be always assumed.
\end{remark}

The following two lemmas  record easy consequences of Remark \ref{basisMhw}.

\begin{lemma}
If $M$ is a highest weight module over  $\Zhu$, then it admits a unique maximal proper submodule.
\end{lemma}
\begin{proof}
 If $N$ is a proper submodule, since $v^-_\gamma$ is odd  and $N$ is $\Z_2$-graded, by \eqref{MMM} we have $N\cap M_{\l,\ell_0}\subset \C v^-_\gamma v_{\l,\ell_0}$, $\gamma\in S^{0}_{-1/2}$. This implies that the sum of all proper submodules is still proper. The statement follows.
\end{proof}

 \begin{lemma} Let $M$ be an irreducible finite-dimensional module over $\Zhu$. Then $M$ is a highest weight module.
 \end{lemma} 
\begin{proof}  Since $M$ is irreducible and $L$ is central, $L$ acts by $\ell_0I_M$ for some $\ell_0\in\C$. Since $M$ is finite dimensional and $\g^\natural$ is a semisimple Lie algebra,  $\h^\natural$ acts semisimply on it. Let $M_0$ be the   eigenspace corresponding to 
 the maximal eigenvalue for the action of $h_0$. $M_0$ is clearly $\h^\natural$-stable, hence  we can choose a weight $\l$ such that $(M_0)_{\l,\ell_0}\ne \{0\}$. As in in Remark \ref{basisMhw}, we have $((M_0)_{\l,\ell_0})_{\bar i}\ne \{0\}$ for some $\bar i$, so we can choose a nonzero vector $v_{\lambda,\ell_0}$ in it. Clearly \eqref{zh3} and \eqref{zh5} hold. By construction \eqref{zh1} and \eqref{zh2} hold. Since $M$ is  irreducible, $v_{\lambda,\ell_0}$ 
generates $M$.
\end{proof}

A Verma module of highest weight $(\l,\ell_0)$ for $\Zhu$ is  a highest weight module $M^Z(\l,\ell_0)$ of highest weight $(\l,\ell_0)$ such that the set of linear generators given in \eqref{Vermabasis} is a basis of $M$.

\begin{lemma}For all pairs $(\l,\ell_0)$, a Verma module over $\Zhu$ of highest weight $(\l,\ell_0)$ exists.
\end{lemma}
\begin{proof}
Given the pair $(\l,\ell_0)$, let $\C_{\l,\ell_0}$ be the one dimensional  representation of $\h^\natural\oplus \C L\oplus \n_0(\si_r)_+$ given by
$$
(h+cL+ n).\,1=\l(h)+c\ell_0,\ h\in\h^\natural,\ n\in\n_0(\si_r)_+.
$$

Set $\mathfrak b^\natural=\h^\natural\oplus \n_0(\si_R)_+$ and
$$Ind(\C_{\l,\ell_0})=\Zhu\otimes_{U(\mathfrak b^\natural)\otimes \C[L]}\C_{\l,\ell_0}.
$$

For simplicity, consider $\C_{\l,\ell_0}$ as a purely even space. 
The $\Z_2$-gradings on $\C_{\l,\ell_0}$ and $\Zhu$ define a $\Z_2$-grading on $\Zhu\otimes\C_{\l,\ell_0}$ by $p(a\otimes m)=p(a)+p(m)$ and, since $U(\mathfrak b^\natural)\otimes \C[L]$ is even, hence graded, the $\Z_2$-grading pushes down to $Ind(\C_{\l,\ell_0})$ making the latter a $\Zhu$-module.

Since $\Zhu$ is free  as a right $U(\mathfrak b^\natural)\otimes \C[L]$-module, a basis of 
$Ind(\C_{\l,\ell_0})$ is given by 
\begin{equation}\label{Vermabasis+}
 \{a_-^{\mathbf q_-}*v_-^{\mathbf p_-}*v_+^{\mathbf p_+}\otimes 1\}.
\end{equation}
 
Set $v_{\l+2\rho_R,\ell_0}=v^+_{1}*\cdots *v^+_{m_+}\otimes 1$. Then
$$
M^Z(\l+2\rho_R,\ell_0)=\Zhu.\,v_{\l+2\rho_R,\ell_0}
$$
is a Verma module. Indeed, since the set \eqref{Vermabasis+} generates $Ind(\C_{\l,\ell_0})$, the eigenvalues of $h_0$ on $Ind(\C_{\l,\ell_0})$ are less or equal to $(\l+2\rho_R)(h_0)$. It follows that
$v_{\l+2\rho_R,\ell_0}$ is a highest weight vector of weight $(\l+2\rho_R,\ell_0)$. This implies that the set
$$
 \{a_-^{\mathbf q_-}*v_-^{\mathbf p_-}\cdot  v_{\l+2\rho_R,\ell_0}\}
$$
generates $M^Z(\l+2\rho_R,\ell_0)$ and, since \eqref{Vermabasis+} is a basis of $Ind(\C_{\l,\ell_0})$, it is linearly independent.\end{proof}

We now show that the modules $M^Z(\l,\ell_0)$ satisfy the usual universal property for highest weight modules. If $\mathbf m=(m_1,\ldots,m_t)\in \Z_+^t$, we set $|\mathbf m|=\sum_{i=1}^t m_i$.

\begin{lemma}
The annihilator of $v_{\l,\ell_0}\in M^Z(\l,\ell_0)$ in $\Zhu$ is 
$$
J_{\l,\ell_0}=span(a_-^{\mathbf q_-}*v_-^{\mathbf p_-}*h(\l)^{\mathbf k}*v_+^{\mathbf p_+}*a_+^{\mathbf q_+}*(L-\ell_0)^k\mid |\mathbf k|+|\mathbf q_+|+|\mathbf p_+|+k>0).
$$
In particular $J_{\l,\ell_0}$ is a left ideal and 
$$
M^Z(\l,\ell_0)\simeq \Zhu/J_{\l,\ell_0}.
$$
\end{lemma}
\begin{proof}
The set $\{a_-^{\mathbf q_-}*v_-^{\mathbf p_-}\mid \mathbf q_-\in \Z_+^{n},\ \mathbf p_-\in\{0,1\}^{m_-}\}$ completes the set
$$
\{a_-^{\mathbf q_-}*v_-^{\mathbf p_-}*h(\l)^{\mathbf k}*v_+^{\mathbf p_+}*a_+^{\mathbf q_+}*(L-\ell_0)^k\mid |\mathbf k|+|\mathbf q_+|+|\mathbf p_+|+k>0\}
$$
to the  basis $\mathcal B_{\l,\ell_0}$ of $\Zhu$.
\end{proof}

\begin{cor}\label{unqueirr}
The Verma module $M^Z(\l,\ell_0)$ is the universal highest weight module over $\Zhu$ of highest weight $(\l,\ell_0)$. In particular its unique simple quotient $L^Z(\l,\ell_0)$ is the unique irreducible highest weight module of highest weight $(\l,\ell_0)$.
\end{cor}
\begin{proof}
If $M$ is a highest weight module of highest weight $(\l,\ell_0)$ and highest weight vector $v_{\l,\ell_0}$ then $J_{\l,\ell_0}\cdot v_{\l,\ell_0}=\{0\}$.
\end{proof}
Let $\phi$ be an almost compact conjugate linear involution of $\g$. 
Recall also that there is a conjugate linear anti-involution $\omega$ defined on generators by
$$\omega(L)=L,\ \omega(a)=-\phi(a),\ a\in\g^\natural,\ \omega(v)=\phi(v),\ v\in\g_{-1/2}.
$$

Fix a pair $(\l,\ell_0)$ with $\l$ purely imaginary (i. e. $\l(\phi(h))=-\overline{\l(h)}$) and $\ell_0$ real.
We now define an invariant Hermitian form on $M^Z(\l,\ell_0)$.

Let $p^\natural:\Zhu\to U(\g^\natural)\otimes \C[L]$ be the projection with respect to the decomposition
$$
\Zhu=span(a_-^{\mathbf q_-}*v_-^{\mathbf p_-}*h(0)^{\mathbf k}*v_+^{\mathbf p_+}*a_+^{\mathbf q_+}*L^k\!\mid |\mathbf p_-|+|\mathbf p_+|>0)\oplus U(\g^\natural)\otimes \C[L].
$$
Let $p^{\h^\natural}$ be the usual Harish-Chandra projection from $U(\g^\natural)$ to $U(\h^\natural)$.
Set 
$$p^Z=(p^{\h^\natural}\otimes I_{\C[L]})\circ p^\natural.
$$
Consider the pair $(\l,\ell_0)$ as an element of $(\h^\natural\oplus\C L)^*$ via $(\l,\ell_0)(h+L)=\l(h)+\ell_0$. View $U(\h^\natural)\otimes \C[L]$ as the polynomial algebra on $(\h^\natural\oplus\C L)^*$ and define a sesquilinear form 
$$
H:\Zhu\times \Zhu \to \C
$$
by setting
$$
H(a,b)=\overline{p^Z(\omega(b)a)(\l,\ell_0)}.
$$
This form is obviously $\omega$-invariant:
$$
H(a,bc)=\overline{p^Z(\omega(bc)a)(\l,\ell_0)}=\overline{p^Z(\omega(c)\omega(b)a)(\l,\ell_0)}=H(\omega(b)a,c).
$$
We now prove that the form $H(\cdot,\cdot)$ can be chosen to be  Hermitian: since the  eigenvalues of $ad(h_0)$ on $\g$ are real we can assume that $\phi(h_0)=-h_0$. With this assumption we have
$$
\phi(\n_0(\si_R)_+)=\n_0(\si_R)_-,\ \phi(\n_{-1/2}(\si_R)_+)=\n_{-1/2}(\si_R)'_-,\ \phi(\fg^0_{-1/2}(\sigma_R))=\fg^0_{-1/2}(\sigma_R), 
$$
hence we can choose the bases $\{v_i^\pm\}$ so that $\omega(v_i^+)=v^-_{m_--\epsilon(\si_R)+1-i}$ and, if $\epsilon(\si_R)=1$, $\omega(v^-_{m_-})=v^-_{m_-}$.
With this choice of bases it is clear that, if $\gamma\in S^0_{-1/2}$,
$$\omega(a_-^{\mathbf q_-}*v_-^{\mathbf p_-}*h(0)^{\mathbf k}*v_+^{\mathbf p_+}*a_+^{\mathbf q_+}*L^k)=\omega(a_+^{\mathbf q_+})*v_-^{\mathbf p'_+}*\omega(h(0)^{\mathbf k})*v_\gamma^{p_\gamma^-}*v_+^{\mathbf p'_-}*\omega(a^{\mathbf q_+})*L^k
$$
where if $\mathbf p\in\{0,1\}^{m_\pm}$ we let $\mathbf p'=(p_{m_+}, \ldots,p_1)$.
Since $[h,v^-_\gamma]=0$ we obtain
$$\omega(a_-^{\mathbf q_-}*v_-^{\mathbf p_-}*h(0)^{\mathbf k}*v_+^{\mathbf p_+}*a_+^{\mathbf q_+}*L^k)=\omega(a_+^{\mathbf q_+})*v_-^{\mathbf p'_+}*v_\gamma^{p_\gamma^-}*\omega(h(0)^{\mathbf k})*v_+^{\mathbf p'_-}*\omega(a_-^{\mathbf q_-})*L^k
$$
It follows that $p^\natural\circ\omega=\omega\circ p^\natural$. Similarly one has $\omega\circ p^{\h^\natural}=p^{\h^\natural}\circ \omega$ so 
$$
p^Z\circ \omega=\omega\circ p^Z.
$$
We are finally ready to compute
$$
H(b,a)=\overline{p^Z(\omega(a)b)(\l,\ell_0)}=\overline{p^Z(\omega(\omega(b)a))(\l,\ell_0)}=\overline{\omega(p^Z(\omega(b)a))(\l,\ell_0)},
$$
hence we need only to check that 
$$
\overline{\omega(p^Z(\omega(b)a))(\l,\ell_0)}=p^Z(\omega(b)a))(\l,\ell_0)=\overline{H(a,b)}.
$$
To check this last equality it is enough to prove that
$$
\overline{\omega(h(0)^{\mathbf k}L^k)(\l,L_0)}= (h(0)^{\mathbf k}L^k)(\l,L_0).
$$
Indeed, since $\l$ is purely imaginary and $\ell_0$ is real,
$$
\overline{\omega(h(0)^{\mathbf k}L^k)(\l,L_0)}= \overline{\l(\omega(h_1))^{k_1}\cdots \l(\omega(h_r))^{k_r} \ell_0^k}=\overline{\l(-\phi(h_1))^{k_1}\cdots \l(-\phi(h_r))^{k_r} \ell_0^k}.
$$
$$
=\overline{\overline{\l(h_1)}^{k_1}\cdots \overline{\l(h_r)}^{k_r} \ell_0^k}=\l(h_1)^{k_1}\cdots \l(h_r)^{k_r} \ell_0^k= (h(0)^{\mathbf k}L^k)(\l,L_0).
$$
We now show that $H(\cdot,\cdot)$ can be pushed down to define a Hermitan form on $M^Z(\l,\ell_0)$. We need to check that
$$
H(J_{\l,\ell_0},\Zhu)=H(\Zhu,J_{\l,\ell_0})=0.
$$
Since the form is Hermitan it is enough to show that 
$$
H(J_{\l,\ell_0},\Zhu)=0.
$$
Since $J_{\l,\ell_0}$ is a left ideal and $H(a,b)=\overline{p^Z(\omega(a)b)(\l,\ell_0)}$, it is enough to show that, if $a\in J_{\l,\ell_0}$, then $
p^Z(a)(\l,\ell_0)=0$. This is easily checked: if $|\mathbf q_-|+|\mathbf p_-|+|\mathbf q_+|+|\mathbf p_+|>0$ then
$$
p^Z(a_-^{\mathbf q_-}*v_-^{\mathbf p_-}*h(\l)^{\mathbf k}*v_+^{\mathbf p_+}*a_+^{\mathbf q_+}*(L-\ell_0)^k)=0.
$$
If $|\mathbf q_-|=|\mathbf p_-|=|\mathbf q_+|=|\mathbf p_+|=0$, then
$$
p^Z(a_-^{\mathbf q_-}*v_-^{\mathbf p_-}*h(\l)^{\mathbf k}*v_+^{\mathbf p_+}*a_+^{\mathbf q_+}*(L-\ell_0)^k)(\l,\ell_0)=h(\l)^{\mathbf k}(L-\ell_0)^k(\l,\ell_0)=0.
$$

We summarize the above discussion in the following result. Recall that a Hermitian form $H(\cdot,\cdot)$ on a superspace $V$ is called even if $H(V_{\bar0},V_{\bar1})=0$.

\begin{prop}\label{exhf}
If $\l$ is purely imaginary and $\ell_0\in \R$, then there is a unique even Hermitian form $H_{\l,\ell_0}(\cdot,\cdot)$ on $M^Z(\l,\ell_0)$ such that $H_{\l,\ell_0}(v_{\l,\ell_0},v_{\l,\ell_0})=1$. 

Furthermore, the radical of $H_{\l,\ell_0}(\cdot,\cdot)$ is the unique maximal $\Z_2$-graded proper submodule of $M^Z(\l,\ell_0)$, hence $H_{\l,\ell_0}(\cdot,\cdot)$ induces the unique nondegenerate invariant even Hermitian form on $L^Z(\l,\ell_0)$ such that $H_{\l,\ell_0}(v_{\l,\ell_0},v_{\l,\ell_0})=1$.
\end{prop}
\begin{proof}
The invariant Hermitian form $H(\cdot,\cdot)$ defined above has clearly the property that $H(1,1)=1$. Moreover, since $p^Z(a)=0$ if $a$ is odd, $H(\cdot,\cdot)$ is even. Suppose that $H'(\cdot,\cdot)$ is another even invariant Hermitian form on $M^Z(\l,\ell_0)$ such that $H'(v_{\l,\ell_0},v_{\l,\ell_0})=1$. Set  $(H-H')(\cdot,\cdot)=H(\cdot,\cdot)-H'(\cdot,\cdot)$. Clearly $(H-H')(\cdot,\cdot)$ is an invariant even Hermitian form. Moreover $(H-H')(v_{\l,\ell_0},v_{\l,\ell_0})=0$ and, since $(H-H')(\cdot,\cdot)$ is even,  
$$(H-H')(v_{\l,\ell_0},v_\gamma^-v_{\l,\ell_0})=(H-H')(v_\gamma^-v_{\l,\ell_0},v_{\l,\ell_0})=0.
$$
Observe that $\omega(v_\gamma^-)v_\gamma^-v_{\l,\ell_0}\in M^Z(\l,\ell_0)_{\l,\ell_0}$ and it has the same parity of $v_{\l,\ell_0}$. Since $M^Z(\l,\ell_0)_{\l,\ell_0}=span(v_\gamma^-v_{\l,\ell_0},v_{\l,\ell_0})$, it follows that $\omega(v_\gamma^-)v_\gamma^-v_{\l,\ell_0}=const.\,v_{\l,\ell_0}$. Thus
$$(H-H')(v_\gamma^-v_{\l,\ell_0},v_\gamma^-v_{\l,\ell_0})=(H-H')(\omega(v_\gamma^-)v_\gamma^-v_{\l,\ell_0},v_{\l,\ell_0})=\overline{const.}\,(H-H')(v_{\l,\ell_0},v_{\l,\ell_0})=0.
$$
The upshot is that $(H-H')(\cdot,\cdot)_{|M^Z(\l,\ell_0)_{\l,\ell_0}\times M^Z(\l,\ell_0)_{\l,\ell_0}}=0$. By invariance, if $(\l,\ell_0)\ne(\nu,\ell)$, 
$$(H-H')(M^Z(\l,\ell_0)_{\l,\ell_0},M^Z(\l,\ell_0)_{\nu,\ell})=0
$$  so we conclude that $(H-H')(M^Z(\l,\ell_0)_{\l,\ell_0},M^Z(\l,\ell_0))=0$. Since $M^Z(\l,\ell_0)_{\l,\ell_0}$ generates $M^Z(\l,\ell_0)$ we deduce $(H-H')(\cdot,\cdot)=0$ so $H'(\cdot,\cdot)=H(\cdot,\cdot)$.

It remains to prove that the radical of the form is the unique maximal $\Z_2$-graded proper submodule of $M^Z(\l,\ell_0)$. By invariance of the form, the radical is a submodule of $M^Z(\l,\ell_0)$ and, since $H(1,1)=1$, it is proper. If $N$ is a proper  $\Z_2$-graded submodule of $M^Z(\l,\ell_0)$, then $N\cap M^Z(\l,\ell_0)_{\l,\ell_0}\subset \C v_\gamma^-v_{\l,\ell_0}$. Since $H$ is even, $H(N,v_{\l,\ell_0})=0$ and, since $v_{\l,\ell_0}$ generates $M^Z(\l,\ell_0)_{\l,\ell_0}$, $N$ is in the radical of the form.
\end{proof}
\begin{remark} Universal minimal $W$-algebras $\Wu$ are defined for arbitrary simple basic finite-dimensional Lie superalgebras \cite{KRW}, \cite{KW1}, but in general the grading \eqref{grad} is not compatible with parity. The automorphism $\si_R$ in general is defined by 
$$\si_R(a)=(-1)^{2j}a,\,\text{ if $a\in\g_j$ in \eqref{grad}.}$$ 
Then the Zhu algebra $Zhu_R$ is still defined by the above commutation relations \cite{KMP}. On the other hand, it is isomorphic to  the finite $W$-algebra, associated to a minimal $sl_2$, studied in \cite{premet} for simple Lie algebras. Actually, this result holds for all quantum affine $W$-algebras \cite[Example 2.12]{DK}. Since $Zhu_R$ is independent of $k$, we obtain from \eqref{Zhu4} that in cases when $\g$ is a simple Lie algebra, Premet's constant in \cite{premet} is $c_0=-\half p(-h^\vee)$.
\end{remark}
\section{Ramond sector: necessary conditions for unitarity}\label{necessary}
 Let $\D$ and $\D^\natural$ be the sets of roots of $\g$ and $\g^\natural$ respectively. As in \cite[Table 1]{KMP1},  we choose a specific set  $\Pi$ of simple roots for $\g$. This set $\Pi$ has the property that, letting $\Pi_{odd}=\{\a\in\Pi\mid\a\text{ is odd}\}$, then $\Pi^\natural:=\Pi\setminus\Pi_{odd}$ is a set of simple roots for $\g^\natural$. We let $\Dp$ and $(\D^\natural)^+$ denote the sets of positive roots in $\D$ and $\D^\natural$ corresponding to $\Pi$ and $\Pi^\natural$ respectively. We also let $\xi$ be the highest weight (with respect to $\Pi^\natural$) of $\g_{1/2}$ as a $\g^\natural$-module. Since $\g_{-1/2}$ is isomorphic to $\g_{1/2}$ as a $\g^\natural$-module, we necessarily have that $\xi=-\a_{|\h^\natural}$  for all $\a\in\Pi_{odd}$.
 
Let  $\Da^{\tw}$ denote the set of roots of $\ga^{\tw}$;  explicitly
   $$\Da^{\tw}=\{s\d+\a\mid \a\in\D^\natural\cup \{\theta\},\,s\in\ZZ\}\cup   \{s\d+\a\mid \a\in\D_{1/2},\,s\in\half+\ZZ\}\cup \{s\d\mid s\in \pm \nat\}.$$
 We also set
 $
\Da^\natural=\{s\d+\a\mid \a\in\D^\natural\mid\,s\in\ZZ\}\cup \{s\d\mid s\in \pm \nat\}$.

We specialize to $\s=\s_R$ the triangular decomposition given in  \eqref{eq:2.6}, \eqref{eq:2.8}, and \eqref{eq:2.9}. Let  $\Da^{\tw}_+,\,\Pia$ be the corresponding sets of positive and simple roots and set $\Da_+^\natural=\Da^\natural\cap \Da^{\tw}_+$.

Let $\D_{i}=\{\a\in\D\mid\a(x)=i\}$ and $\overline \D_{i}=(\D_{i})_{|\h^\natural}$. 
Since $\ad f$ is an isomorphism of $\g^\natural$-modules between $\g_{1/2}$ and $\g_{-1/2}$ it follows that $\D_{\pm1/2}=\{\pm\half\theta+\eta\mid \eta\in\overline{\D}_{1/2}\}$. Therefore $\overline{\D}_{1/2}=\overline{\D}_{-1/2}=-\overline{\D}_{1/2}$.

The set $\overline\D_{1/2}$ is naturally  partially ordered setting $\beta\prec \gamma$ if and only if  $\gamma-\beta$ is a sum of roots in $\Pi^\natural$.

Recall from Section \ref{Setup} that the   triangular decomposition \eqref{eq:2.6} depends on the choice of  a  regular element $h_0\in\h^\natural$. We choose $h_0$ such that $\a(h_0)>0$ for all $\a\in\Pi^\natural$, so that $\Da_+^\natural$ is the set of positive roots in $\Da^\natural$ whose corresponding set of simple roots is $\Pia^\natural:=\Pi^\natural\cup\{\d-\theta_i\mid 1\leq i\leq s\}$.

 Set 
\begin{equation}\label{Deltaplusonehalf}
\overline\D^+_{1/2}=\{\eta\in\overline\D_{1/2}\mid \eta(h_0)>0\}.
\end{equation}
Below we give an explicit description of the set $\overline{\D}_{1/2}$ in all cases. A direct check using this  data  shows that three cases occur:
\begin{enumerate}
\item Case 1: $\g=psl(2|2)$. In this case $\overline\D_{1/2}=\{\xi,\xi-\theta_1=-\xi\}$. It follows that $\overline{\D}^+_{1/2}=\{\xi\}$. Set $\eta_{\min}=\xi$. It is obvious that $\eta_{\min}$ is the unique minimal element of $\overline\D^+_{1/2}$.
\item Case 2: $0\in\overline\D_{1/2}$. In this case $\prec$ is a total order. It follows that $\overline\D_{1/2}^+=\{\eta\in\overline\D_{1/2}\mid \eta\succ0\}$. Set $\eta_{\min}$ to be its minimum.
\item Case 3: $\g\ne psl(2|2)$ and $0\notin\overline\D_{1/2}$. In this case   the Hasse diagram  of the poset $\overline{\D}_{1/2}$ is as follows:
$$
{\xymatrixcolsep{.3pc}\xymatrixrowsep{.3pc}
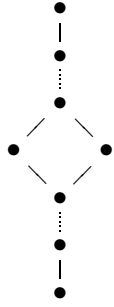
\xymatrix@R+3pt@C+2pt{ 
&&&{\bullet}\ar@{-}[d]\\
&&&{\bullet}\ar@{.}[d]\\
&&&{\bullet}\ar@{-}[dr]\ar@{-}[dl]\\ 
&&{\bullet}\ar@{-}[dr]&&{\bullet}\ar@{-}[dl]\\  
&&&{\bullet}\ar@{.}[d]\\ 
&&&{\bullet}\ar@{-}[d]\\ 
&&&{\bullet}
}}
$$
If $\eta_{\min}$ is a minimal element of  $\overline{\D}_{1/2}^+$, then $\overline{\D}_{1/2}^+\supset\{\eta\in\overline{\D}_{1/2}^+\mid \eta\succeq\eta_{\min}\}$. Let $\eta_1,\eta_2$ be the two non comparable elements in $\overline\D_{1/2}$.
 The map $\eta\mapsto -\eta$ is the unique order reversing involution without fixed points of the diagram. It follows that $\eta_1=-\eta_2$, hence $\{\eta_1,\eta_2\}\cap\overline\D_{1/2}^+$ has exactly one element. It is then clear that $\eta_{\min}$ is the unique element of $\{\eta_1,\eta_2\}\cap\overline\D_{1/2}^+$ (hence $\overline\D_{1/2}^+$ has a unique minimal element) and that $\overline{\D}_{1/2}^+=\{\eta\in\overline{\D}_{1/2}^+\mid \eta\succeq\eta_{\min}\}$.

\end{enumerate}  
 
The outcome of this discussion is that, if $\g$ is of type $spo(2|2r), F(4), D(2,1;a)$,  then there are two choices of $\Da^{\tw}_+$; one for each choice of $\eta_{\min}$. In all the other cases there is only one choice for $\Da^{\tw}_+$. Remarkably, in all cases $\overline\D_{1/2}^+$ has a unique minimal element $\eta_{\min}$.

We now describe explicitly case by case, for each choice of $\Da_+^{\tw}$,  the set $\overline{\D}_{1/2}$, the set $\overline{\D}_{1/2}^+$ with its minimal element $\eta_{\min}$, and the set of simple roots $\Pia$.

\bigskip
\noindent$\mathbf{\g=psl(2|2)}$. In this case $\overline{\D}_{1/2}=\{\xi,-\xi\}$ with $\xi=\tfrac{\d_1-\d_2}{2}$, $\overline{\D}_{1/2}^+=\{\xi\}$, $\eta_{\min}=\xi$ and  
$$
\Pia=\{\d-(\e_1-\e_2),-\half \d+\d_1-\e_2,\d-(\d_1-\d_2),-\half\d+\e_1-\d_2\}.
$$
The set $\Pia$ is actually a set of simple roots for $\ga^{\,\tw}$ and  the two simple isotropic roots are linearly independent (their restrictions to $\ha$ are equal). 

\medskip
\noindent$\g=\mathbf{spo(2|3)}$. In this case  $\overline\D_{1/2}=\{\xi, 0,-\xi\}$ with $\xi=\e_1$, $\overline\D^+_{1/2}=\{\xi\}$, $\eta_{\min}=\xi$ and
$$
\Pia=\{-\frac{\d}{2}+\d_1+\e_1
   ,\frac{\d}{2}-\d_1,\frac{\d}{2}
   +\d_1-\e_1\}.
   $$
$\mathbf{\g=spo(2|2r), r>2}$. In this case $\overline\D_{1/2}=\{\pm\e_i\mid i=1,\ldots,r\}$ with $\xi=\e_1$. There are two choices for $\overline\D^+_{1/2}$: either  $\overline\D^+_{1/2}=\{\e_i\mid i=1,\ldots,r\}$ or $\overline\D^+_{1/2}=\{\e_i\mid i=1,\ldots,r-1\}\cup\{-\e_r\}$.

For the first choice  $\eta_{\min}=\e_r$ and 
$$ 
\Pia=\{\d-(\e_1+\e_2),\e_1-\e_2,\ldots,\e_{r-2}-\e_{r-1},\e_{r-1}-\e_r, -\half\d+\d_1+\e_r,\d-2\d_1\}.
$$

For the second choice  $\eta_{\min}=-\e_r$ and 
$$ 
\Pia=\{\d-(\e_1+\e_2),\e_1-\e_2,\ldots,\e_{r-2}-\e_{r-1},\e_{r-1}+\e_r, -\half\d+\d_1-\e_r,\d-2\d_1\}.
$$
$\mathbf{\g=spo(2|2r+1), r>1}$. In this case $\overline\D_{1/2}=\{\pm\e_i\mid i=1,\ldots,r\}\cup\{0\}$,  $\xi=\e_1$, $\overline\D^+_{1/2}=\{\e_i\mid i=1,\ldots,r\}$, $\eta_{\min}=\e_r$ and 
$$
\Pia=\{\d-(\e_1+\e_2),\e_1-\e_2,\ldots,\e_{r-2}-\e_{r-1},\e_{r-1}-\e_r, -\half\d+\d_1+\e_r,\half\d-\d_1\}.
$$
$\bf{\g=D(2,1;a).}$ In this case $\overline\D_{1/2}=\{\pm\e_2\pm\e_3\}$ with $\xi=\e_2+\e_3$. There are two choices for $\overline\D^+_{1/2}$: either 
 $\overline\D^+_{1/2}=\{\e_2+\e_3,\e_2-\e_3\}
 $ or 
 $\overline\D^+_{1/2}=\{\e_2+\e_3,-\e_2+\e_3\}
 $.

For the first choice  $\eta_{\min}=\e_2-\e_3$ and
$$
 \Pia=\{-\half\d+\e_1+\e_2-\e_3,2\e_3,\d-2\e_2,\d-2\e_1\}.
 $$
 
For the second  choice  $\eta_{\min}=-\e_2+\e_3$ and
$$
 \Pia=\{-\half\d+\e_1-\e_2+\e_3,2\e_2,\d-2\e_3,\d-2\e_1\}.
 $$
$\mathbf{\g=F(4)}$.  In this case $\overline\D_{1/2}=\{\half(\pm\e_1\pm\e_2\pm\e_3)\}$ with $\xi=\half(\e_1+\e_2+\e_3)$. There are two choices for $\overline\D^+_{1/2}$: either 
 $$\overline\D^+_{1/2}=\{\half(\e_1+\e_2+\e_3),\half(\e_1+\e_2-\e_3),\half(\e_1-\e_2+\e_3),\half(-\e_1+\e_2+\e_3)\}
 $$ or 
 $$\overline\D^+_{1/2}=\{\half(\e_1+\e_2+\e_3),\half(\e_1+\e_2-\e_3),\half(\e_1-\e_2+\e_3),\half(\e_1-\e_2-\e_3)\}.
 $$
 
For the first choice   $\eta_{\min}=\half(-\e_1+\e_2+\e_3)$ and 
 $$
 \Pia=\{\d-(\e_1+\e_2),\e_1-\e_2,\e_2-\e_3,-\half(\d-\d_1+\e_1-\e_2-\e_3),\d-\d_1\}.
 $$
  
For the second  choice  $\eta_{\min}=\half(\e_1-\e_2-\e_3)$ and 
 $$
 \Pia=\{\d-(\e_1+\e_2),\e_2-\e_3,\e_3,-\half(\d-\d_1-\e_1+\e_2+\e_3),\d-\d_1\}.
 $$
$\mathbf{\g=G(3)}$.   In this case $\overline\D_{1/2}=\{\pm\e_1,\pm\e_2,\pm(\e_2+\e_1)\}\cup\{0\}$, $\xi=\e_2+\e_1$,
 $\overline\D^+_{1/2}=\{\e_1,\e_2,\e_2+\e_1\}$,  $\eta_{\min}=\e_1$ and 
$$
\Pia=\{
-\frac{\d}{2}+\delta_1+\e_1,
 \frac{\d}{2}-\delta_1 ,
 \e_2-\e_1 ,
 \d-\e_1-2 \e_2\}.
 $$

 \bigskip
Let $M$ be a $\sigma_R$-twisted $W^k_{\min}(\g)$-module.  If $v\in\g_{-1/2}$, then
$$Y^M(G^{\{v\}}, z)=\sum_{m\in  \tfrac{1}{2}+\ZZ} G^{\{v\},\tw}_{(m)}z^{-m-1}=\sum_{m\in \mathbb Z}G^{\{v\},\tw}_{m}z^{-m-3/2}.$$
In this section and the following ones we will write $G^{\{v\},\tw}_{m}=G^{\{v\}}_{m}$,  $J^{\{a\},\tw}_{m}=J^{\{a\}}_{m}$ for short.
 
Next, we would like to provide a more precise description of the action of $G_0^{\{v\}}$ on the highest weight vector of a highest weight module $M$.
 For computing it, we use  a {\it good choice} $\{X_\a\mid \a\in\D\}$ of root vectors as in \cite{KMP1}, and choose as a basis of $\g_{-1/2}$
the vectors 
\begin{equation}\label{utris}u_\a=X_\a+\sqrt{-1} N_{-\theta,\a}X_{\a-\theta},\quad \a\in\Dp,\,\text{$\a$ odd}.\end{equation}
Here $N_{\a,\beta}$ are structure constants w.r.t. $\{X_\a\}$. We observe that we can choose $u_\gamma=X_\gamma$ for $\gamma\in \D^\natural$, so that 
$u^\gamma=X_{-\gamma}$. We now check that the $N_{\a,\beta}$  are the structure constants w.r.t $\{u_\a\}$. Compute, for $\gamma\in\D^\natural$ and $\a$ odd positive,
\begin{align*}
[X_\gamma, u_{\a}]&=[X_\gamma, X_{\a}+\sqrt{-1} N_{-\theta,\a}X_{\a-\theta}]\\
&=N_{\gamma,\a} X_{\gamma+\a}+\sqrt{-1}N_{-\theta,\a}N_{\gamma,\a-\theta}X_{\gamma+\a-\theta}.\end{align*}
Since this vector lies in $\g_{-1/2}$, it is a linear combination of the $u_\a$'s, and this forces 
\begin{equation}\label{r2}[X_\gamma, u_{\a}]=N_{\gamma,\a} u_{\gamma+\a}.\end{equation} Substituting  $u_{\gamma+\a}$ by its expression \eqref{utris}, we deduce that 
\begin{equation}\label{r11}
N_{-\theta,\a}N_{\gamma,\a-\theta}=N_{-\theta,\gamma+\a}N_{\gamma,\a}.\end{equation}

Let $\phi$ be an almost compact conjugate linear involution of $\g$ \cite[Definition 1.1]{KMP1}.
\begin{lemma}\label{complication} For $\gamma\in \D^\natural$, $\a$ odd positive, set $c_{\gamma,\a}=N_{\gamma,-\a}N_{\a,-\gamma}$. Then 
\begin{equation}\label{dacc}\langle[u_\gamma,\phi(u_\a)],[u_\a,u^\gamma]\rangle=
\langle\phi(u_\a),u_\a\rangle c_{\gamma,\a}.\end{equation}
\end{lemma}
\begin{proof} 
To compute \eqref{dacc} we use the following formulas
$$ 
\phi(u_{\a})=-N_{-\theta,\a}u_{\theta-\a},\ \langle u_\a,u_\beta\rangle=-(N_{-\theta,\a}+N_{-\theta,\beta})\d_{\theta-\a,\beta}.
$$
The first one is given in Lemma 5.3 of \cite{KMP1} while the second is  (5.12) of \cite{KMP1}.
In particular
 we have
\begin{equation}\label{46}
\langle \phi(u_{\a}),u_{\a}\rangle=-N_{-\theta,\a}\langle u_{\theta-\a},u_{\a}\rangle=N_{-\theta,\a}(N_{-\theta,\theta-\a}+N_{-\theta,\a}).
\end{equation}
It follows that
\begin{align}
\langle[u_\gamma,\phi(u_\a)],[u_\a,u^\gamma]\rangle&=
\langle[X_\gamma,\phi(u_{\a})],[u_{\a},X_{-\gamma}]\rangle\notag\\
&=\langle[X_\gamma,-N_{-\theta,\a}u_{\theta-\a}],[u_{\a},X_{-\gamma}]\rangle\notag\\
&=N_{-\theta,\a} N_{\gamma,\theta-\a}N_{-\gamma,\a}\langle  u_{\gamma+\theta-\a}, u_{-\gamma+\a}\rangle\notag\\
&=-(N_{-\theta,\gamma+\theta-\a}+N_{-\theta,-\gamma+\a})N_{-\theta,\a} N_{\gamma,\theta-\a}N_{-\gamma,\a}.\label{referenza}
\end{align}
Using \eqref{r11} for $\gamma$ and $-\gamma$, formula \eqref{referenza} becomes
\begin{align}
\langle[u_\gamma,\phi(u_\a)],[u_\a,u^\gamma]\rangle&=-N_{-\theta,\a}^2N_{-\gamma,\a-\theta}N_{\gamma,\theta-\a}-N_{-\theta,\a}N_{-\theta,\theta-\a}N_{\gamma,-\a}N_{-\gamma,\a}\\\label{r44}
&=N_{-\theta,\a}^2c_{-\gamma,\theta-\a}+N_{-\theta,\a}N_{-\theta,\theta-\a}c_{\gamma,\a}.
\end{align}
We want to prove that $N_{-\gamma,\a-\theta}N_{\gamma,\theta-\a}=N_{\gamma,-\a}N_{-\gamma,\a}$. 
Observe that
$$
[X_{-\theta},[X_{\theta},X_{-\a}]]=X_{-\a}
$$
hence
$$
[X_\gamma,[X_{-\theta},[X_{\theta},X_{-\a}]]]=N_{\gamma,-\a}X_{-\a}.$$
On the other hand
$$
[X_\gamma,[X_{\theta},X_{-\a}]]=N_{\gamma,\theta-\a}[X_{\theta},X_{-\a}]
$$
so
$$ 
[X_\gamma,[X_{-\theta},[X_{\theta},X_{-\a}]]]=N_{\gamma,\theta-\a}[X_{-\theta},[X_{\theta},X_{-\a}]]=N_{\gamma,\theta-\a}X_{-\a}.
$$
thus 
$$
N_{\gamma,-\a}=N_{\gamma,\theta-\a}.
$$
Similarly one has
$$
N_{-\gamma,\a}=N_{-\gamma,\a-\theta}.
$$

This implies $c_{-\gamma,\theta-\a}=c_{\gamma,\a}$; \eqref{r44} now reads
\begin{align}
\langle[u_\gamma,\phi(u_\a)],[u_\a,u^\gamma]\rangle&=c_{\gamma,\a}(N_{-\theta,\a}^2+N_{-\theta,\a}N_{-\theta,\theta-\a}).\label{r45}
\end{align}
Substituting \eqref{46} into \eqref{r45} yields 	\eqref{dacc}.
\end{proof}
Let  $v$ be a weight vector of the $\g^\natural$-module $\g_{-1/2}$ of weight $\eta$.  Observe that $\eta=\a_{|\h^\natural}$, with $\a=\half\theta+\eta$, a positive odd root. If $\gamma\in \D^\natural$, set  $c_{\gamma,\eta}=c_{\gamma,\a}$.

\begin{lemma}\label{cgammaeta}If $\eta\ne0$ then
$$
c_{\gamma,\eta}=\begin{cases}
(\gamma|\eta)&\text{if $(\gamma|\eta)\le0$}\\
0&\text{otherwise}
\end{cases}.
$$
If $\eta=0$ then 
$$
c_{\gamma,\eta}=\begin{cases}
-\half&\text{if $\gamma$ is a root from $\D^\natural$ of minimal length,}\\
0&\text{otherwise.}
\end{cases}
$$
\end{lemma}
\begin{proof}
If $\eta\ne0$ then $\a=\half\theta+\eta$ is an isotropic root so we can use formula \cite[(4.6)]{GKMP}: if  $\a-\gamma$ is a root then
$$
c_{\gamma,\eta}=(\a|\gamma)=(\eta|\gamma).
$$
If $(\a|\gamma)<0$ then $\a-\gamma$ is a root so 
$c_{\gamma,\eta}=(\eta|\gamma)$ in this case.
If $(\a|\gamma)=0$ and $\a-\gamma$ is a root, then also $\a+\gamma$ is a root and they are both nonisotropic. This implies 
$
\a-\gamma=\theta/2=\a+\gamma$ which is  absurd. So, if $(\a|\gamma)=0$ then  $
c_{\gamma,\eta}=(\eta|\gamma)$ as well. 
If $(\a|\gamma)>0$
then  $\a-\gamma$ cannot be isotropic, for, in such a case,
$$
2(\a|\gamma)=(\gamma|\gamma)<0.
$$
It follows that $\a-\gamma$ is not isotropic. Since $\a+\gamma$ is a root, it must be isotropic by the argument above. It follows that
 $$
2(\a|\gamma)=-(\gamma|\gamma)
$$
and
$$
(\a-\gamma|\a-\gamma)=-2(\a|\gamma)+(\gamma|\gamma)=2(\gamma|\gamma)<0.
$$
On the other hand $\a-\gamma=\theta/2$ so 
$$
(\a-\gamma|\a-\gamma)=\half.
$$
The only possibility is that, when $(\a|\gamma)>0$, then $\a-\gamma$ is not a root hence $
c_{\gamma,\eta}=0$.

If $\eta=0$ then $\theta/2$ is a root, hence $\g^\natural$ acts on $\g_{-1/2}$ as the little adjoint representation. If $v$ has weight $0$ and $\gamma$ is a root of $\g^\natural$ then $[X_{\gamma},v]$ has weight $\gamma$, so, if $\gamma$ is long, $N_{\gamma,-\theta/2}=0$ so $c_{\gamma,0}=0$. If $\gamma$ is short, then 
$$
c_{\gamma,0}=N_{\theta/2,-\gamma}N_{\gamma,-\theta/2}
$$
and use formula (4.6) from \cite{GKMP} to compute it. Since $\gamma$ is short, the $\theta/2$-strings through  $\pm\gamma$ are given by
$\pm\gamma-\theta/2, \pm\gamma,\pm\gamma+\theta/2$, hence $N_{\theta/2,-\gamma}N_{\gamma,-\theta/2}=N_{\theta/2,\gamma}N_{-\gamma,-\theta/2}=-(\theta/2|\theta/2)=-1/2$.
\end{proof}

\begin{lemma} \label{JJ}Let $M$ be a $\si_R$-twisted highest weight $\Wu$-module with highest weight $(\nu,\ell_0)$ and highest weight vector $v_{\nu,\ell_0}$.  Let $v\in\g_{-1/2}$ be a weight vector for $\h^\natural$ of weight $\eta$. Then
\begin{align}\label{f1}&\sum_{\a,\be}\langle[u_\alpha,\phi(v)],[v,u^\be]\rangle J^{\{u^\a\}}_0J^{\{u_\beta\}}_0v_{\nu,\ell_0}
=\langle \phi(v),v\rangle\left((\eta|\nu)^2-\sum_{\gamma<0}c_{\gamma,\eta}(\nu|\gamma)\right)v_{\nu,\ell_0},\\
&\label{f2}\sum_{\a,\be}\langle[u_\alpha,\phi(v)],[v,u^\be]\rangle J^{\{u_\beta\}}_0J^{\{u^\a\}}_0v_{\nu,\ell_0}
=\langle \phi(v),v\rangle\left((\eta|\nu)^2+\sum_{\gamma>0}c_{\gamma,\eta}(\nu|\gamma)\right)v_{\nu,\ell_0}.
\end{align}
\end{lemma}
\begin{proof}  Note that $v= c u_\eta$ for some non-zero constant $c$. Using Lemma \ref{complication}, we have
\begin{align*}&\sum_{\gamma,\be}\langle[u_\gamma,\phi(v)],[v,u^\be]\rangle J^{\{u^\gamma\}}_0J^{\{u_\beta\}}_0v_{\nu,\ell_0}\\
&=\sum_{i,j}\langle[u_i,\phi(v)],[v,u^j]\rangle J^{\{u^i\}}_0J^{\{u_j\}}_0v_{\nu,\ell_0}
+\sum_{\gamma<0}\langle[u_\gamma,\phi(v)],[v,u^\gamma]\rangle J^{\{u^\gamma\}}_0J^{\{u_\gamma\}}_0v_{\nu,\ell_0}\\\notag
&=\langle \phi(v),v\rangle \sum_{i,j}\eta(u^j)\eta(u_i)\nu(u_j)\nu(u^i)v_{\nu,\ell_0}
+\sum_{\gamma<0}\langle[u_\gamma,\phi(v)],[v,u^\gamma]\rangle J^{\{u^\gamma\}}_0J^{\{u_\gamma\}}_0v_{\nu,\ell_0}\\\notag
&=\langle \phi(v),v\rangle(\eta|\nu)^2v_{\nu,\ell_0}
+\sum_{\gamma<0}\langle[u_\gamma,\phi(v)],[v,u^\gamma]\rangle J^{\{u^\gamma\}}_0J^{\{u_\gamma\}}_0v_{\nu,\ell_0}\\\notag
&=\langle \phi(v),v\rangle(\eta|\nu)^2v_{\nu,\ell_0}
+\sum_{\gamma<0}\langle[u_\gamma,\phi(v)],[v,u^\gamma]\rangle [J^{\{u^\gamma\}}_0,J^{\{u_\gamma\}}_0]v_{\nu,\ell_0}\\\notag
&=\left(\langle \phi(v),v\rangle(\eta|\nu)^2
-\sum_{\gamma<0}\langle[u_\gamma,\phi(v)],[v,u^\gamma]\rangle (\nu|\gamma)\right)v_{\nu,\ell_0}\notag\\
&=\left(\langle \phi(v),v\rangle(\eta|\nu)^2
-\sum_{\gamma<0}\langle[u_\gamma,\phi(c u_\eta)],[c u^\eta,u^\gamma]\rangle (\nu|\gamma)\right)v_{\nu,\ell_0}\notag\\
&\stackrel{\eqref{dacc}}{=}\left(\langle \phi(v),v\rangle(\eta|\nu)^2
-|c|^2\sum_{\gamma<0}\langle \phi(u_\eta),u_\eta\rangle c_{\gamma,\eta} (\nu|\gamma)\right)v_{\nu,\ell_0}\notag\\
&=\left(\langle \phi(v),v\rangle(\eta|\nu)^2
-\sum_{\gamma<0}\langle \phi(v),v \rangle c_{\gamma,\eta} (\nu|\gamma)\right)v_{\nu,\ell_0}.
\end{align*}
This proves  \eqref{f1}. For \eqref{f2} we have
\begin{align*}
&\sum_{\gamma,\be}\langle[u_\gamma,\phi(v)],[v,u^\be]\rangle J^{\{u_\beta\}}_0J^{\{u^\gamma\}}_0v_{\nu,\ell_0}=\\\notag
&\sum_{i,j}\langle[u_i,\phi(v)],[v,u^j]\rangle J^{\{u_j\}}_0J^{\{u^i\}}_0v_{\nu,\ell_0}
+\sum_{\gamma>0}\langle[u_\gamma,\phi(v)],[v,u^\gamma]\rangle J^{\{u_\gamma\}}_0J^{\{u^\gamma\}}_0v_{\nu,\ell_0}\\\notag
&=\left((\eta|\nu)^2\langle \phi(v),v\rangle +
\sum_{\gamma>0}\langle[u_\gamma,\phi(v)],[v,u^\gamma]\rangle(\nu|\gamma)\right) v_{\nu,\ell_0}.\end{align*}
We conclude as above.
\end{proof}

\begin{lemma}\label{G0squared}\ 
With the hypothesis of Lemma \ref{JJ}, we have 
\begin{align*}
&[G^{\{\phi(v)\}}_0, G^{\{v)\}}_0]v_{\nu,\ell_0}=\\
&\langle \phi(v),v\rangle\!\left( (-2(k+h^\vee) \ell_0+ (\nu|\nu+2\rho^\natural)  -\half p(k)+2(\eta|\nu)^2\!+\!
\sum_{\gamma>0}c_{\gamma,\eta} (\nu|\gamma)-
\sum_{\gamma<0} c_{\gamma,\eta}  (\nu|\gamma)\right)\!v_{\nu,\ell_0}.
\end{align*}
\end{lemma}
\begin{proof}
Using Borcherds' commutator formula
\begin{align*}
[G^{\{\phi(v)\}}_{0},G^{\{v\}}_{0}]&=\sum_j\binom{1/2}{j}({G^{\{\phi(v)\}}}_{(j)}G^{\{v\}})_{0}\\&=
({G^{\{\phi(v)\}}}_{(0)}G^{\{v\}})_{0}+1/2({G^{\{\phi(v)\}}}_{(1)}G^{\{v\}})-1/8({G^{\{\phi(v)\}}}_{(2)}G^{\{v\}}).\end{align*}
and recalling that \cite{AKMPP} 
\begin{align}[{G^{\{u\}}}_\l G^{\{v\}}]=&-2(k+h^\vee)\langle u,v\rangle L+\langle u,v\rangle\sum_{\alpha=1}^{\dim \g^\natural} \label{GGsimplifiedfurther}
:J^{\{u^\alpha\}}J^{\{u_\alpha\}}:\\\notag
&+2\sum_{\a,\be}\langle[u_\alpha,u],[v,u^\be]\rangle:J^{\{u^\a\}}J^{\{u_\be\}}:
 +2(k+1) (\partial+ 2\lambda) J^{\{[[e,u],v]^{\natural}\}}\\\notag
&+ 2 \lambda \sum_{\a,\be}\langle[u_\alpha,u],[v,u^\be]\rangle
 J^{\{ [u^\a,u_\be]\}}
       + 2\l^2\langle u,v\rangle p(k)\vac,
\end{align}
we get 
\begin{align*}\notag &[G^{\{u\}}_{0},G^{\{v\}}_{0}]v_{\nu, \ell_0}=\langle u,v\rangle(-2(k+h^\vee) \ell_0+ (\nu|\nu+2\rho^\natural))v_{\nu,\ell_0}\\
&\notag +2\sum_{\a,\be}\langle[u_\alpha,u],[v,u^\be]\rangle :J^{\{u^\a\}}J^{\{u_\beta\}}:_0v_{\nu,\ell_0}
\\
\notag
&+  \sum_{\a,\be}\langle[u_\alpha,u],[v,u^\be]\rangle
J^{\{ [u^\a, u_\be]\}}_0v_{\nu,\ell_0}
       -\tfrac{1}{2}\langle u,v\rangle p(k)v_{\nu,\ell_0}=\\
\notag
  &\langle u,v\rangle(-2(k+h^\vee) \ell_0+ (\nu|\nu+2\rho^\natural))v_{\nu,\ell_0}+2\sum_{\a,\be}\langle[u_\alpha,u],[v,u^\be]\rangle J^{\{u_\beta\}}_0J^{\{u^\a\}}_0v_{\nu,\ell_0}
\\
\notag
&+  \sum_{\a,\be}\langle[u_\alpha,u],[v,u^\be]\rangle
( J^{\{u^\a\}}_0J^{\{u_\beta\}}_0-J^{\{u_\beta\}}_0J^{\{u^\a\}}_0)v_{\nu,\ell_0}
       - \tfrac{1}{2} \langle u,v\rangle p(k)v_{\nu,\ell_0}=\notag\\
        &\langle u,v\rangle(-2(k+h^\vee) \ell_0+ (\nu|\nu+2\rho^\natural)-\tfrac{1}{2} p(k))v_{\nu,\ell_0}\notag\\
        &+\sum_{\a,\be}\langle[u_\alpha,u],[v,u^\be]\rangle J^{\{u_\beta\}}_0J^{\{u^\a\}}_0v_{\nu,\ell_0}
+  \sum_{\a,\be}\langle[u_\alpha,u],[v,u^\be]\rangle
 J^{\{u^\a\}}_0J^{\{u_\beta\}}_0v_{\nu,\ell_0}.
\end{align*}
Now apply Lemma \ref{JJ}.
\end{proof}

Let $\overline\D^+_{1/2}$ be the set defined in \eqref{Deltaplusonehalf} and observe that, by construction, $\overline\D^+_{1/2}=\{\eta\in(\h^\natural)^*\mid \text{ $\eta$ is a $\h^\natural$-weight of $\mathfrak n_{-1/2}(\si_R)'_+$}\}$.
Set 
\begin{equation}\label{tuttirho}\rho_R=\rho(\fn_{1/2}(\si_R)_+)_{|\h^\natural}=\half \sum_{\g_\a\subset \fn_{1/2}(\si_R)_+}\a_{|\h^\natural}=\half\sum_{\eta\in\overline\D^+_{1/2}}\eta.
\end{equation}
Applying the formulas of Lemma \ref{cgammaeta} we obtain the following refinement of Lemma \ref{G0squared}.
\begin{lemma}\label{G0squaredspecial}Under the hypothesis of Lemma \ref{JJ}, we have
\begin{enumerate}
\item Assume that $\theta/2\in\D$ and let $v\in\g_{-1/2}$ be a vector of weight $0$. Then
$$
G^{\{\phi(v)\}}_{0}G^{\{v\}}_{0}v_{\nu,\ell_0}=2\langle\phi(v), v\rangle\left((-2(k+h^\vee) \ell_0+ (\nu|\nu+2(\rho^\natural-\rho_R))-\tfrac{1}{2} p(k)\right)v_{\nu,\ell_0}.
$$
\item Assume that  $v\in\g_{-1/2}$ is a vector of weight $\eta\ne0$. Then
\begin{align*}
&[G^{\{\phi(v)\}}_{0},G^{\{v\}}_{0}]v_{\nu,\ell_0}\\&=\langle\phi(v), v\rangle\left((-2(k+h^\vee) \ell_0+ (\nu|\nu+2\rho^\natural)-\tfrac{1}{2} p(k)+2(\eta|\nu)^2\right)v_{\nu,\ell_0}\\
&+
\left(\sum_{\gamma>0, (\gamma|\eta)\le 0}(\gamma|\eta) (\nu|\gamma)-
\sum_{\gamma>0,(\gamma|\eta)\ge 0} (\gamma|\eta) (\nu|\gamma)\right)v_{\nu,\ell_0}.
\end{align*}
\end{enumerate}
\end{lemma}
\begin{proof}To prove (1), recall that in this case, as already observed in the proof of Lemma \ref{cgammaeta}, $\g^\natural$ acts on $\g_{-1/2}$ as the little adjoint representation, so its nonzero weights are precisely the roots in $\D^\natural$ of minimal length. Since we chose $h_0$ so that $\a(h_0)>0$ for $\a\in\D^\natural_+$, it follows immediately from \eqref{Deltaplusonehalf} that 
\begin{equation}\rho_R=\sum_{\gamma\in \D^\natural_+, \text{$\gamma$ short}} \gamma.\end{equation} By Lemma 	\ref{G0squared} and Lemma \ref{cgammaeta},
\begin{align*}
&[G^{\{\phi(v)\}}_0, G^{\{v\}}_0]v_{\nu,\ell_0}\\
&=\langle \phi(v),v\rangle\!\left( (-2(k+h^\vee) \ell_0+ (\nu|\nu+2\rho^\natural)  -\half p(k)-
\sum_{\gamma\in \D^\natural_+, \text{$\gamma$ short}} (\nu|\gamma)\right)\!v_{\nu,\ell_0}
\\
&=\langle \phi(v),v\rangle\!\left( (-2(k+h^\vee) \ell_0+ (\nu|\nu+2\rho^\natural)  -\half p(k)-
2(\nu|\rho_R)\right)\!v_{\nu,\ell_0}.
\end{align*}
Now observe that $\phi(v)= h \, v$, so that $[G^{\{\phi(v)\}}_0, G^{\{v)\}}_0]=\half h  (G^{\{v\}}_0)^2=\half G^{\{\phi(v)\}}_0G^{\{v\}}_0$.

If $v$ has weight $\eta\ne0$ then, by  Lemma \ref{cgammaeta},
\begin{align*}
\sum_{\gamma>0}c_{\gamma,\eta} (\nu|\gamma)-
\sum_{\gamma<0} c_{\gamma,\eta}  (\nu|\gamma)&=
\sum_{\gamma>0, (\gamma|\eta)\le 0}(\gamma|\eta) (\nu|\gamma)-
\sum_{\gamma<0,(\gamma|\eta)\le 0} (\gamma|\eta) (\nu|\gamma)\\
&=\sum_{\gamma>0, (\gamma|\eta)\le 0}(\gamma|\eta) (\nu|\gamma)-
\sum_{\gamma>0,(\gamma|\eta)\ge 0} (\gamma|\eta) (\nu|\gamma).
\end{align*}
Thus (2) follows from Lemma \ref{G0squared}. 
\end{proof}
Let $L^W(\lambda,\ell_0)$ be the irreducible $\si_R$-twisted positive energy $\Wu$-module such that $L^W(\lambda,\ell_0)_0=L^Z(\lambda,\ell_0)$ (see Theorem 2.30 of \cite{DK}). It is clear that $L^W(\lambda,\ell_0)$ is a highest weight module of highest weight $(\l,\ell_0)$. Conversely, if $M$ is an irreducible highest weight module of highest weight $(\l,\ell_0)$, then, by Lemma \ref{basisW}, the grading given by the action of $L_0$ defines the structure of a positive energy module. By  Theorem 2.30 of \cite{DK}, $M_0$ is irreducible. Since clearly $M_0$ is a highest weight module of highest weight $(\l,\ell_0)$, by Corollary \ref{unqueirr}, $M_0=L^Z(\l,\ell_0)$ hence, by  \cite[Theorem 2.30]{DK} again, $M=L^W(\l,\ell_0)$. This shows that $L^W(\l,\ell_0)$ is the unique irreducible $\si_R$-twisted highest weight $\Wu$-module of highest weight $(\l,\ell_0)$.

Consider $(\nu,\ell_0)$ with $\nu$ purely imaginary and $\ell_0$ real.  Combining Proposition \ref{exhf}  and \cite[Proposition 6.7]{KMP}, we get the existence of a unique   even invariant Hermitian form $H(\cdot,\cdot)$ on $L^W(\nu,\ell_0)$ such that $H(v_{\nu,\ell_0},v_{\nu,\ell_0})=1$ ($v_{\nu,\ell_0}$ is  a highest weight vector).
\begin{proposition}\label{norma}
Let $L^W(\nu,\ell_0)$ be the  irreducible $\si_R$-twisted unitary highest weight module  with highest weight vector $v_{\nu,\ell_0}$  and 
the Hermitian form $H(\cdot,\cdot)$, and set $||a||^2=H(a,a)$.
\begin{enumerate}
\item
Let $v \in \fn_{-1/2} (\sigma)'_+$ be a vector of $\h^\natural$-weight $\eta$. Then
\begin{align*}&||G^{\{\phi(v)\}}_{0}v_{\nu,\ell_0}||^2=\\\\&=\langle\phi(v), v\rangle\left((-2(k+h^\vee) \ell_0+ (\nu|\nu+2\rho^\natural)-\tfrac{1}{2} p(k)+2(\eta|\nu)^2\right)\\
&+\langle\phi(v), v\rangle
\left(\sum_{\gamma>0, (\gamma|\eta)\le 0}(\gamma|\eta) (\nu|\gamma)-
\sum_{\gamma>0,(\gamma|\eta)\ge 0} (\gamma|\eta) (\nu|\gamma))\right).
\end{align*}
\item Assume that $\theta/2\in\D$. Then 
\begin{align*}
&||G^{\{u_{\theta/2}\}}_{0}v_{\nu,\ell_0}||^2=2\langle\phi(u_{\theta/2}), u_{\theta/2}\rangle\left((-2(k+h^\vee) \ell_0+ (\nu|\nu+2(\rho^\natural-\rho_R))-\tfrac{1}{2} p(k)\right).\end{align*}
\end{enumerate}
\end{proposition}
\begin{proof}
We have
 \begin{align}\label{inizio}
H(G^{\{\phi(v)\}}_{0}v_{\nu,\ell_0},G^{\{\phi(v)\}}_{0}v_{\nu,\ell_0})&=H(G^{\{v\}}_{0}G^{\{\phi(v)\}}_{0}v_{\nu,\ell_0},v_{\nu,\ell_0})=H([G^{\{\phi(v)\}}_{0},G^{\{v\}}_{0}]v_{\nu,\ell_0},v_{\nu,\ell_0}).
\end{align}
Now apply Lemma \ref{G0squared} to get (1). As for (2), we have
 \begin{align}\label{inizio2}
H(G^{\{u_{\theta/2}\}}_{0}v_{\nu,\ell_0},G^{\{u_{\theta/2}\}}_{0}v_{\nu,\ell_0})&=H((G^{\{\phi(u_{\theta/2})\}}_{0}(G^{\{u_{\theta/2}\}}_{0})v_{\nu,\ell_0},v_{\nu,\ell_0}),
\end{align}
and we can apply (1) in Lemma \ref{G0squaredspecial}.
\end{proof}

If $L^W(\nu,\ell_0)$ is as in   Proposition \ref{norma}, then, since $\langle \phi(v),v\rangle >0$, by (1) of this proposition we have
\begin{equation}\label{maxeta}
\ell_0\ge\frac{1}{2(k+h^\vee)}( (\nu|\nu+2\rho^\natural)-\tfrac{1}{2} p(k)+F_\nu(\eta)),
\end{equation}
where
\begin{equation}	\label{ev}
F_\nu(\eta)=2(\eta|\nu)^2 + \sum_{\gamma>0, (\gamma|\eta)\le 0}(\gamma|\eta) (\nu|\gamma)-
\sum_{\gamma>0,(\gamma|\eta)\ge 0} (\gamma|\eta) (\nu|\gamma).
\end{equation}
Similarly,  if $\eta=0$ is a weight of $\g_{-1/2}$, then, by (2) of Proposition \ref{norma}, we have
\begin{equation}\label{maxetazero}
\ell_0\ge\frac{1}{2(k+h^\vee)}( (\nu|\nu+2(\rho^\natural-\rho_R))-\tfrac{1}{2} p(k)).
\end{equation}

Since $k+h^\vee<0$, the maximal value of the right hand side of \eqref{maxeta} is achieved for $\eta$ such that $F_\nu(\eta)$ is minimal.

 Let $P^+\subset (\h^\natural)^*$ be the set of dominant integral weights  for $\g^\natural$. We compute below the minimal value of $F_\nu(\eta)$ for $\nu\in P^+$ and $\eta\in \overline\D^+_{1/2}$ via a case by case inspection. In the cases where $\theta/2$ is an odd root (so that $\eta=0$ is a weight for $\g_{-1/2}$) we also show that $-2(\nu|\rho_R)\le F_{\nu}(\eta)$, hence \eqref{maxetazero} gives the best bound. 
 We gave above an explicit  description of the set $\overline\D^+_{1/2}$, hence we can compute  explicit expressions for $\rho_R$  that we list in Tables \ref{thetahalfisroot} and \ref{thetahalfisnotroot}, along with the values of $\eta_{\min}$, $\rho^\natural$, and $\xi$.  In  Tables \ref{thetahalfisroot} and \ref{thetahalfisnotroot} we denote by $\omega_j^i$ the fundamental weights of the simple ideal  $\g^\natural_i$ of $\g^\natural$.  We drop the superscript $i$ if  $\g^\natural$ is simple.

\renewcommand{\arraystretch}{1.5}
\begin{table}[h]
\begin{tabular}{c | c| c |c   }
$\g$&
$psl(2|2)$& $spo(2|2r+1),r\ge1$&
$G(3)$\\
\hline
$\eta_{\min}$&$\tfrac{\d_1-\d_2}{2}$&$\e_r$
&$\e_1$\\
\hline
$\rho_R$&$\omega_1$&$\omega_r$
&$\omega_1$
\\
\hline
$\rho^\natural$&$\tfrac{\d_1-\d_2}{2}$&$\sum_{i=1}^r(r-i+\half)\e_i$
&$2\e_1+3\e_2$\\
\hline
$\xi$& $\tfrac{\d_1-\d_2}{2}$ &  $\e_1$ & $\e_1+\e_2$
\end{tabular}
	\captionof{table}{Cases with only one choice for $\overline{\D}^+_{1/2}$.\label{thetahalfisroot}}
\end{table}

\renewcommand{\arraystretch}{1.5}
{\small
\begin{table}[h]
\begin{tabular}{c | c| c |c }
$\g$&
$spo(2|2r)$& $D(2,1;\tfrac{m}{n})$&$F(4)$\\
\hline
$\eta_{\min}$&$\e_r, -\e_r$&$\e_2-\e_3,-\e_2+\e_3$&$\half(\e_1-\e_2-\e_3),\half(-\e_1+\e_2+\e_3)$
\\
\hline
$\rho_R$&$\omega_r,\omega_{r-1}$&$\omega_1^1,\omega_2^2$
&$\omega_1,\omega_3$
\\
\hline
$\rho^\natural$&$\sum_{i=1}^r(r-i)\e_i$,&$\e_2+\e_3$
&$\tfrac{5}{2}\e_1+\tfrac{3}{2}\e_2+\half\e_3$
\\
\hline
$\xi$&$\e_1$&$\e_2+\e_3$&$\half(\e_1+\e_2+\e_3)$
\end{tabular}
\captionof{table}{Cases with two choices for $\overline{\D}^+_{1/2}$.\label{thetahalfisnotroot}}
\end{table}
}
 
$\mathbf{psl(2|2).}$ In this case $\D_+^\natural=\{\d_1-\d_2\}$ and   $\nu=r/2(\d_1-\d_2)$. 
Then 
\begin{equation}\label{Fpsl}
\min_{\eta}F_\nu(\eta)=F_\nu(\xi)=\half r^2+r.
\end{equation}

$\bf{spo(2|3).}$ In this case $\D_+^\natural=\{\e_1\}$ and $\nu=\tfrac{r}{2}\e_1$. 
Then 
\begin{equation}\label{spo231}
\min_{\eta}F_\nu(\eta)=F_\nu(\xi)=\frac{r^2+r}{8}=\frac{r}{4}+\frac{1}{8}r(r-1).
\end{equation}
Since in this case $\eta=0$ is a weight of $\g_{-1/2}$, we need also to compute
\begin{equation}\label{spo232}
-2(\nu|\rho_R)=\frac{r}{4},
\end{equation}
which gives the minimal value of $F_\nu(\eta)$.

$\bf{spo(2|2r), r>2.}$ In this case $\D_+^\natural=\{\e_i\pm\e_j, 1\leq i<j\leq r\}$. Since $\nu$ is dominant integral, $\nu=\sum_i m_i\e_i,\,m_i\in\half +\ZZ$ or $m_i\in \ZZ, m_1\geq\ldots\geq m_{r-1}\geq |m_r|$. If $\eta=\e_i$, then
\begin{align*}
&F_\nu(\eta)=m_i^2/2-1/2 \sum_{\gamma\in\{e_s+\e_i, i\ne s\}\cup\{\e_i-\e_s,s>i\}} (\nu|\gamma)-
\half \sum_{\gamma\in\{\e_s-\e_i,s<i\}} (\nu|\gamma)\\
&=m_i^2/2-1/2 \left(\sum_{\gamma\in\{\e_s\pm\e_i, s<i\}} (\nu|\gamma)+\sum_{\gamma\in\{\e_i\pm\e_s, s>i\}} (\nu|\gamma)\right)\\
&=m_i^2/2+1/4\left( \sum_{ s<i} (m_s\pm m_i)+ \sum_{ s>i} (m_i\pm m_s)\right)=m_i^2/2+1/2 \sum_{s<i} m_s+1/2(r-i)m_i.
\end{align*}
The minimum value for $F_\nu(\eta)$ is achieved if  $\eta=\pm\e_r$ and it is
\begin{equation}\label{Fpari}
F_\nu(\eta)=\half(m_r^2+ \sum_{s<r} m_s).
\end{equation}

$\bf{spo(2|2r+1), r>1}.$  $\D_+^\natural=\{\e_i\pm\e_j, 1\leq i<j\leq r\}\cup\{\e_i\}$, $\nu=\sum m_i\e_i,\,m_i\in\half +\ZZ$ or $m_i\in \ZZ, m_1\geq	\ldots\geq m_r\geq 0$. If $\eta=\e_i$, then
\begin{align*}
&F_\nu(\eta)=m_i^2/2-1/2 \sum_{\gamma\in\{\e_s+\e_i, i\ne s\}\cup\{ \e_i-\e_s,s>i\}\cup\{ \e_i\}} (\nu|\gamma)-
\half \sum_{\gamma\in\{\e_s-\e_i,s<i\}} (\nu|\gamma)
\\&=m_i^2/2+1/4\left( \sum_{ s<i} (m_s\pm m_i)+ m_i+\sum_{ s>i} (m_i\pm m_s)\right)\\&=m_i^2/2+1/2 \sum_{s<i} m_s+1/2(r-i+\half)m_i.\end{align*}
The minimum is achieved when $i=r$ and it is
\begin{equation}\label{resto1}
\min_{\eta}F_\nu(\eta)=\half(m_r^2+ \sum_{s<r} m_s+\half m_r)=\half\sum_{s\le r} m_s+\half m_r(m_r-\half).
\end{equation}
Since in this case $\eta=0$ is a weight of $\g_{-1/2}$, we need also to compute
\begin{equation}\label{resto2}
-2(\nu|\rho_R)=-\sum_{i=1}^r(\nu|\e_i)=\half\sum_{s\le r} m_s,
\end{equation}
which gives the minimal value of $F_\nu(\eta)$.

$\bf{D(2,1;a).}$ $\D_+^\natural=\{2\e_2,2\e_3\}$,  $\nu= m_1\e_1+m_2\e_2,\,m_i\in\ZZ_+$. Then a computer computation shows that the minimum of $F_\nu(\eta)$ is attained at $\e_2-\e_3$  and at $-\e_2+\e_3$  and its value is
\begin{equation}\label{FD21}
\min_\eta F_\nu(\eta)=\frac{(m_1-am_2)^2+2(m_1 +a ^2 m_2)}{2 (1+a )^2}.
\end{equation}

$\bf{F(4).}$  $\D_+^\natural=\{\e_i\pm\e_j, 1\leq i<j\leq 3\}\cup\{\e_i\}$. Since $\nu$ is dominant integral,  $\nu=\sum m_i\e_i,\,m_i\in\half +\ZZ$ or $m_i\in \ZZ, m_1\geq	m_2\geq r_3\geq 0$.
 A computer computation shows that the minimun of $F_\nu(\eta)$ is attained when  $\eta=\half(-\e_1+\e_2+\e_3)$ and when $\eta=\half(\e_1-\e_2-\e_3)$. Its value is
\begin{equation}\label{F}
 \min_{\eta}F_\nu(\eta)=\frac{2}{9}((-m_1+m_2+m_3)^2+5m_1+m_2 +m_3).
\end{equation}

$\bf{G(3).}$  We have $\D_+^\natural=\{\e_1,\e_2-\e_1,\e_2,\e_2+\e_1,\e_2+2\e_1,2\e_2+\e_1\}$, $\nu=m_1\e_1+m_2\e_2,\, 2m_1\ge m_2\ge m_1,\,m_i\in\ZZ_+$.
 A computer computation shows that the minimun of $F_\nu(\eta)$ is attained when  $\eta=\e_1$ and 
$$\min_\eta F_\nu(\eta)= \frac{1}{8} (2m_1-m_2)(2m_1-m_2-1)+\frac{1}{2}(m_1+m_2).$$
Since in this case $\eta=0$ is a weight of $\g_{-1/2}$, we need also to compute
\begin{equation}
-2(\nu|\rho_R)=\frac{1}{2}(m_1+m_2)\le\min_\eta F_\nu(\eta).\end{equation}
Our direct inspection shows  the following fact
\begin{proposition} We have $
\min_\eta F_\nu(\eta)=F_\nu(\eta_{\min})$, 
 and this value  is independent from the choice of the set $\overline\D^+_{1/2}$. 
\end{proposition}
 Recall \cite[Theorem 2.1]{KW1}, \cite[Section 7]{KMP1} that there is an embedding in $\Wu$ of the universal affine vertex algebra $\bigotimes_i V^{M_i(k)}(\g^\natural_i)$, where  the $M_i(k),\,i=1,\ldots,s,$ are given in Table \ref{numerical}, along with other quantities, explained in \cite[\S\! 7.2]{KMP1}. If $L^W(\nu,\ell_0)$ is unitary, then $\bigotimes_i V^{M_i(k)}(\g^\natural_i)\cdot v_{\nu,\ell_0}$ is a unitary representation of the latter affine vertex algebra, hence $\nu$  lies in  $P^{+}_k$, where
\begin{equation}\label{p+k}P^{+}_k=\left\{\nu\in P^+\mid  \nu(\theta^\vee_i)\le M_i(k)\text{ for all $i=1,\ldots,s$}\right\},\end{equation}
where $P^+$ is the set of dominant integral weights of $\g^\natural$.
\begin{table}[h]
\renewcommand\arraystretch{1.5}
{\scriptsize
\begin{tabular}{c | c| c|  c | c | c |c}
$\g$ & $\g^\natural$ & $h^\vee$ & $\bar h_i^\vee$ & $M_i(k)$ &$\chi_i$&$u_i=(\theta_i|\theta_i)$\\\hline
$psl(2|2)$& $sl_2$  & $-2$ &  $0$ & $-k-1$&$-1$&$-2$\\\hline
$spo(2|3)$& $sl_2$& $1/2$ & $-1/2$ & $-4k-2$&$-2$&$-1/2$\\\hline
$spo(2|m), m\ge4$ & $so_m$ &  $2-m/2$ & $1-m/2$ &  $-2k-1$&$-1$&$-1$\\\hline
$D(2,1;a)$ &  $sl_2\oplus sl_2$ &$0$ &  $-\tfrac{2}{1+a}, -\tfrac{2a}{1+a} $ & $-(1+a)k-1, -\frac{1+a}{a}k-1$ &$-1,-1$&$-\tfrac{2}{1+a}, -\tfrac{2a}{1+a} $\\[1pt]\hline
$F(4)$& $so_7$ &$-2$ & $-10/3$ & $-\frac{3}{2} k-1$ &$-1$&$-4/3$\\\hline
$G(3)$& $G_2$   &$-3/2$  &  $-3$ &$ -\frac{4}{3} k-1$&$-1$&$-3/2$
\end{tabular}
}
\captionof{table}{Numerical information\label{numerical}}
\end{table}

We summarize our findings in the following statement:
\begin{proposition}\label{32} Let $\nu\in P^+_k$ and $\ell_0\in\R$. Set 
\begin{equation}\label{Aknufirst}A(k,\nu)=\begin{cases}\tfrac{1}{2(k+h^\vee)}((\nu|\nu+2(\rho^\natural-\rho_R))    -\tfrac{1}{2} p(k))&\text{if $\theta/2\in\D$,}\\
\tfrac{1}{2(k+h^\vee)}\left((\nu|\nu+2\rho^\natural)    -\tfrac{1}{2} p(k)+F_{\nu}(\eta_{\min})\right)&\text{otherwise.}
\end{cases}
\end{equation}
If  $L^W(\nu,\ell_0)$ is   unitary, then 
$$\ell_0\geq  A(k,\nu).$$
\end{proposition}

\begin{lemma}\label{extremal} If $L^W(\nu,\ell_0)$ is unitary, then $\ell_0=A(k,\nu)$ in the following cases:
\begin{enumerate}
 \item  $\g\ne spo(2|3), D(2,1;a)$: $\nu\in P^+_k$ such that  $\nu-\rho_R\notin P^+$.
 \item  $\g=spo(2|3)$: $\nu=0$ and $\nu= \tfrac{M_1(k)}{2}\epsilon_1$.
 \item  $\g=D(2,1;a)$:  $\nu=r\e_3$ and $\nu= r\e_2+M_2(k)\e_3$, where  $r=0,\ldots,M_1(k)$ (resp.  $\nu=r\e_2$ and $\nu=M_1(k)\e_2+r\e_3,$ where $r=0,\ldots,M_2(k)$) if $\rho_R=\omega_1^1=\e_2$ (resp. $\rho_R=\omega_1^2=\e_3$).
\end{enumerate}
\end{lemma}
\begin{proof} Assume $\nu-\rho_R\notin P^+$. 
We first discuss the cases when $\theta/2$ is a root, namely $\g=spo(2|2r+1)$ and $\g=G(3)$.  Since $\rho_R=\omega^{i_0}_{ j_0}$ for some $i_0,j_0$, then  $\nu=\sum_{i,j} m^i_j\omega^i_j$ with  $m^{i_0}_{j_0}=0$. Let $\overline\a$ be the simple root corresponding to $\omega^{i_0}_{ j_0}$ (i. e. $\omega^{i_0}_{ j_0}(\overline \a^\vee)=\d_{\a,\overline \a}$ for all simple roots $\a$ of $\D^\natural$). Explicitly $\overline \a=\e_r$ for $\g=spo(2|2r+1)$ while $\overline \a=\e_1$ for $\g=G(3)$. Then $J^{\{X_{-\overline \a}\}}_0$ acts trivially on $v_{\nu,\ell_0}$. We note that $\overline\a=\eta_{\min}$. 
  If $v\in\g_{-1/2}$ has weight $\eta_{\min}$, then $G^{\{v\}}_0v_{\nu,\ell_0}=0$. Note that $v$ is a root vector for the root $-\theta/2+\eta_{\min}$, hence $[X_{-\overline \a},v]$ is a root vector for the root $-\theta/2$, thus $w=[X_{-\overline \a},v]$ is a nonzero multiple of $u_{\theta/2}$. Since
$$
0=J^{\{X_{-\overline \a }\}}_0G^{\{v\}}_0v_{\nu,\ell_0}=G^{\{[X_{-\overline \a},v]\}}_0v_{\nu,\ell_0}=G^{\{w\}}_0v_{\nu,\ell_0}=const.\,G^{\{u_{\theta/2}\}}_0v_{\nu,\ell_0},
$$
we deduce that $\Vert G^{\{u_{\theta/2}\}}_0v_{\nu,\ell_0}\Vert^2=0$ and Proposition \ref{norma} implies $\ell_0=A(k,\nu).$

If  $\g=D(2,1;a)$ and  $\eta_{\min}=\e_2-\e_3$, then $\rho_R=\omega^1_{1}=\e_2$. Let $\overline \a=2\e_{2}$. Since $\nu-\rho_R\notin P^+$, $\nu=r_3\e_3$, so $X_{-\overline\a}v_{\nu,\ell_0}=0$.
 If $v\in\g_{-1/2}$ has weight $\eta_{\min}$, then $G^{\{v\}}_0v_{\nu,\ell_0}=0$. Note that $v$ is a root vector for the root $-\e_1+\e_2-\e_3$, hence $[X_{-\overline \a},v]$ is a root vector for the root $-\e_1-\e_2-\e_3$, thus $w=[X_{-\overline \a},v]\ne0$. Moreover $w\in\g_{-1/2}$ is a vector of weight $-\eta_{\min}$ thus $w=const.\,\phi(v)$.
 Therefore
  $$
0=J^{\{X_{-\overline \a }\}}_0G^{\{v\}}_0v_{\nu,\ell_0}=G^{\{[X_{-\overline \a},v]\}}_0v_{\nu,\ell_0}=G^{\{w\}}_0v_{\nu,\ell_0}=const.\,G^{\{\phi(v)\}}_0v_{\nu,\ell_0}.
$$
We deduce that $\Vert G^{\{\phi(v)\}}_0v_{\nu,\ell_0}\Vert^2=0$ and Proposition \ref{norma} implies $\ell_0=A(k,\nu)$ also in this case.
The argument for $\g=D(2,1;a)$,  $\eta_{\min}=-\e_2+\e_3$, and $\nu=r_2\e_2$ is completely analogous using $\overline \a=2\e_3$.

We now turn to the remaining cases where the argument is similar but somewhat more complicated: we claim that for  $v\in\n_{-1/2}(\si_R)_+$ of weight $\eta_{\min}$, we have
\begin{equation}\label{gov}G^{\{\phi(v)\}}_{0}v_{\nu,\ell_0}=0.\end{equation}

Indeed, if $\g=psl(2|2)$, then $\nu-\rho_R\notin P^+$ means that $\nu=0$, thus $J^{\{a\}}_0v_{(\nu,\ell_0)}=0$ for all $a\in\g^\natural$. We know that $G^{\{u\}}_{0}v_{\nu,\ell_0}=0$ for all $u\in\n_{-1/2}(\si_R)_+$.
Since $[J^{\{a\}}_0,G^{\{vu\}}_{0}]=G^{\{[a,u]\}}_0$, then $G^{\{[a,u]\}}_0v_{0,\ell_0}=J^{\{a\}}_0G^{\{u\}}_{0}v_{0,\ell_0}$. Since $ad(U(\g^\natural))\n_{-1/2}(\sigma_R)=\g_{-1/2}$, we see that $G^{\{u\}}_{0}v_{\nu,\ell_0}=0$ for all $u\in \g_{-1/2}$, thus, in particular, \eqref{gov} holds in this case.

We now prove \eqref{gov} in the  cases   $\g=spo(2|2r)$, $\g=F(4)$.
Since $\rho_R=\omega^{i_0}_{ j_0}$ let $\overline\a$, as above, be the simple root corresponding to $\omega^{i_0}_{ j_0}$.
Since $J^{\{X_{-\overline \a}\}}_0$ acts trivially on $v_{\nu,\ell_0}$ and $G^{\{v\}}_0v_{\nu,\ell_0}=0$ if $v\in \g_{-1/2}$ has weight $\eta_{\min}$,
$$0=J^{\{X_{-\overline \a }\}}_0G^{\{v\}}_0v_{\nu,\ell_0}=G^{\{[X_{-\overline \a},v]\}}_0v_{\nu,\ell_0}.
$$ 
Let $\be=-\half\theta+\eta_{\min}$. This is the root of $\g_{-1/2}$ such that $\be_{|\h^\natural}=\eta_{\min}$. We check in each case that	\begin{enumerate}
	\item[(i)] $\be-\overline\a$ is a root,
	\item[(ii)] there exists a positive root $\gamma\in\D^\natural$  such that $\be-\overline\a+\gamma=-\theta/2-\eta_{\min}$.
	\end{enumerate}
 If (i), (ii) hold, then $[X_{-\overline \a},v]\ne0$ and both $[X_\gamma,[X_{-\overline \a},v]]$ and $\phi(v)$ are in $\g_{\be'}$ with $\be'=-\half\theta-\eta_{\min}$. Since all the roots have multiplicity one, this implies that $[X_\gamma,[X_{-\overline \a},v]]=const. \phi(v)$. Since $\gamma$ is positive $X_\gamma v_{\nu,\ell_0}=0$,  hence
$$0=J^{\{X_\gamma\}}_0J^{\{[X_{-\a_{\bar i}},v]\}}_0G^{\{v\}}_0v_{\nu,\ell_0}=G^{\{[X_\gamma,[X_{-\overline \a},v]\}}_0v_{\nu,\ell_0}= const. \,G^{\{\phi(v)\}}_{0}v_{\nu,\ell_0},
$$
 so that \eqref{gov} holds.

The following list describes for each case the roots $\overline \a$, $\be$, $\be-\overline \a$, and $\gamma$. 
\begin{itemize}
\item Case 1: $so(2|2r)$, $r>2$, $\eta_{\min}=\e_r$, $\rho_R=\omega_r$, $\overline \a=\e_{r-1}+\e_r$, $\be=-\d_1+\e_r$, $\be-\overline \a=-\d_1-\e_{r-1}$, $\gamma={\e_{r-1}-\e_r}$.
\item Case 2: $so(2|2r)$, $r>2$, $\eta_{\min}=-\e_r$, $\rho_R=\omega_{r-1}$, $\overline \a=\e_{r-1}-\e_r$, $\be=-\d_1-\e_r$, $\be-\overline \a=-\d_1-\e_{r-1}$, $\gamma={\e_{r-1}+\e_r}$.
\item Case 3: $F(4)$, $\eta_{\min}=\half(-\e_1+\e_2+\e_3)$, $\rho_R=\omega_{3}$, $\overline \a=\e_{3}$, $\be=\half(-\d_1-\e_1+\e_2+\e_3)$, $\be-\overline \a=\half(-\d_1-\e_1+\e_2-\e_3)$, $\gamma={\e_{1}-\e_2}$.
\item Case 4: $F(4)$, $\eta_{\min}=\half(\e_1-\e_2-\e_3)$, $\rho_R=\omega_{1}$, $\overline \a=\e_1-\e_{2}$, $\be=\half(-\d_1+\e_1-\e_2-\e_3)$, $\be-\overline \a=\half(-\d_1-\e_1+\e_2-\e_3)$, $\gamma={\e_{3}}$.
\end{itemize}
Having established  \eqref{gov}, by Proposition \ref{norma}, $||G^{\{\phi(v)\}}_{0}v_{0,\ell_0}||^2=0$ implies $\ell_0=A(k,0)$.

It remains only to check the statement for $\g=spo(2|3)$ or $\g=D(2,1;a)$ and $\nu=\sum_i \tfrac{r_i}{2}\theta_i$ with $r_i=M_i(k)$ for some $i$. Since
\begin{align*}
J^{\{X_{-\theta_i}\}}_{1}J^{\{X_{\theta_i}\}}_{-1}v_{\nu,\ell_0}&=[J^{\{X_{-\theta_i}\}}_{1},J^{X_{\{\theta_i}\}}_{-1}]v_{\nu,\ell_0}=-(\nu|\theta_i)+M_i(k)\tfrac{(\theta_1|\theta_i)}{2}
\\&=\tfrac{(\theta_1|\theta_i)}{2}(-(\nu|\theta_i^\vee)+M_i(k))=0,
\end{align*}
we see that
\begin{align*}
\Vert J^{X_{\theta_i}}_{-1}v_{\nu,\ell_0}\Vert^2&=H(J^{\{X_{\e_1}\}}_{-1}v_{\nu,\ell_0},J^{\{X_{\e_1}\}}_{-1}v_{\nu,\ell_0})=-H(J^{\{\phi(X_{\e_1})\}}_{1}J^{\{X_{\e_1}\}}_{-1}v_{\nu,\ell_0},v_{\nu,\ell_0}
\\
&=-const.\ H(J^{\{X_{-\e_1}\}}_{1}J^{\{X_{\e_1}\}}_{-1}v_{\nu,\ell_0},v_{\nu,\ell_0})=0.
\end{align*}
By unitarity we deduce 
\begin{equation}\label{Jtheta1}J^{\{X_{\theta_i}\}}_{-1}v_{\nu,\ell_0}=0.
\end{equation}
If $\g=spo(2|3)$ then \eqref{Jtheta1} implies
$$0=G_1^{\{\phi(u_{\d_1+\e_1})\}}J^{\{X_{\e_1}\}}_{-1}v_{\nu,\ell_0}=[G_1^{\{\phi(u_{\d_1+\e_1})\}},J^{\{X_{\e_1}\}}_{-1}]v_{\nu,\ell_0}=-
G_0^{\{[X_{\e_1},\phi(u_{\d_1+\e_1})]\}}v_{\nu,\ell_0}
$$
so, since $[X_{\e_1},\phi(u_{\d_1+\e_1})]=const.u_{\d_1}$,
we obtain
$$
 G_0^{\{u_{\d_1}\}}v_{\nu,\ell_0}=0.
$$
Again Proposition \ref{norma} implies that $\ell_0=A(k,\nu)$.

If $\g=D(2,1;a)$, then observe that $\eta_{\min}=\xi-\theta_i$.  Let $u_\xi\in\g_{-1/2}$ be a vector of weight $-\xi$. Note that $u_\xi$ is a root vector for the root $-\theta/2-\xi$. Since  $\theta_i+(-\theta/2-\xi)=-\theta/2-\eta_{\min}$ is a root, we see that $[X_{\theta_i}, u_\xi]$ is a nonzero vector in $\g_{-1/2}$ of weight $-\eta_{\min}$. If $v\in\g_{-1/2}$ has weight $\eta_{\min}$, it follows that $ [X_{\theta_i}, u_\xi]=const.\,\phi(v)$. By \eqref{Jtheta1},
$$
0=G_1^{\{u_{\xi}\}}J^{\{X_{\theta_i}\}}_{-1}v_{\nu,\ell_0}=[G_1^{\{u_{\xi}\}},J^{\{X_{\theta_i}\}}_{-1}]v_{\nu,\ell_0}=-
G_0^{\{[X_{\theta_1},u_{\xi}]\}}v_{\nu,\ell_0}=const.\, G_0^{\{\phi(v)\}}v_{\nu,\ell_0}.
$$
We can therefore conclude using Proposition \ref{norma}.
\end{proof}

We can finally state the necessary conditions for unitarity.

\begin{theorem}\label{32t} Let  $L^W(\nu,\ell_0)$ be a unitary $\si_R$-twisted $\Wu$-module, where $k$ is in the unitary range. Then
\begin{enumerate}
\item $\nu\in P^+_k$,
\item $\ell_0\geq  A(k,\nu)$,
\item If   $\nu$ is as in Lemma \ref{extremal} (1), (2), (3), then $\ell_0=A(k,\nu)$.
\end{enumerate}
\end{theorem}

\section{Ramond sector: sufficient conditions for unitarity}\label{sufficient}
Recall that  in Section \ref{Setup} we constructed the twisted $F(\g_{1/2})$-module $F(\g_{1/2},\si_R)$. Our first task is to show that this module is unitary.

Fix  bases $\{w_\gamma\}, \{w^\gamma\}$ of $\g_{1/2}$ dual w.r.t. $\langle\cdot,\cdot\rangle_{\rm ne}$, with $\gamma$ ranging over $S_{1/2}$.  Recall that there is an embedding of superspaces
$u\mapsto \Phi_u$ of $\g_{1/2}$ into $F(\g_{1/2})$. Sometimes we write $\Phi_\a$ (resp. $\Phi^\a$) as a shortcut for $\Phi_{w_\a}$ (resp. $\Phi_{w^\a}$).

We can assume
that $S_{1/2}=S_{1/2}^+\cup S^-_{1/2}$ with $\{w_\gamma\mid \gamma\in S^\pm_{1/2}\}$  basis of $\fn_{1/2}(\sigma)_\pm$. We will also need a more refined notation $ S^-_{1/2}= S^{-,0}_{1/2}\cup S^{-,'}_{1/2}$ to label bases of   
$\g^0_{1/2}(\si_R), \n_{1/2}(\si_R)'_-$  (see \eqref{eq:2.5}). We 
\begin{lemma}
The map $t^{-1/2}\Phi_{u}\mapsto[\Phi_{u}]$ extends to an isomorphism between $Cl(\g_{1/2})$ and $Zhu_{\si_R}(F(\g_{1/2}))$.
\end{lemma}
\begin{proof}
Recall that  $*$ denotes the product in $Zhu_{\si_R}(F(\g_{1/2}))$. By \cite[Theorem 2.13]{DK}, 
$$[[\Phi_{u}],[\Phi_{u}]]=\sum_j\binom{-1/2}{j}(\Phi_{u})_{(j)}(\Phi_{v})=\langle u, v\rangle=[t^{-1/2}\Phi_{u},t^{-1/2}\Phi_{v}],
$$
so the map $t^{-1/2}\Phi_{u}\mapsto [\Phi_u]$ extends to a homomorphism from $Cl(\g_{1/2})$ to $Zhu_{\si_R}(F(\g_{1/2}))$.

 We prove by induction on $\D_a$ that $[a]$ is in the subalgebra of   $Zhu_{\si_R}(F(\g_{1/2}))$ spanned by $[\Phi_u],\,u\in\g_{1/2}$, for all $a\in F(\g_{1/2})$. If $a=\vac$ the claim holds. Now assume $a=:T^k\Phi_u\,b:$ with $\D_b<\D_a$. Note that by  \cite[Theorem 2.13 (b)]{DK}
$$:T^k\Phi_u\,b:+:T^{k-1}\Phi_u\,Tb:=T\left(:T^{k-1}\Phi_u\,b:\right)\equiv \,const\,:T^{k-1}\Phi_u\,b:,$$ 
hence we can assume $k=0$. 
Since
$$(\Phi_u)_{(-1,\si_R)}b=\sum_{j\in\ZZ_+}\binom{\tfrac{1}{2}}{j}(\Phi_u)_{(-1+j)}b= :\Phi_u\,a:+\sum_{j\in\nat}\binom{\tfrac{1}{2}}{j}(\Phi_u)_{(-1+j)}b,
$$
we obtain
$$[a]=[:\Phi_u\,b:]= [\Phi_u]*[b]-\sum_{j\in\nat}\binom{\tfrac{1}{2}}{j}[(\Phi_u)_{(-1+j)}b],
$$
and the claim follows by induction.
It follows that the homomorphism from $Cl(\g_{1/2})$ to $Zhu_{\si_R}(F(\g_{1/2}))$ is surjective.

We now prove that the homomorphism is injective. Set
$$L^{\ge}(A,\si_R)=\bigoplus_{\mu\in- \tfrac{1}{2}+\ZZ_+} t^\mu\g_{1/2}\subset L(\g_{1/2},\si_R).
$$
Consider $Cl(\g_{1/2})$ as a $L^{\ge}(A,\si_R)\oplus \C K$-module by letting $t^{-1/2}\g_{1/2}$ act by left multiplication, $t^{\mu}\g_{1/2}$ act trivially if $\mu>0$, and $K$ act by $I_{Cl(\g_{1/2})}$.
Let 
\begin{equation}\label{M}
M=Ind^{\widehat L(\g_{1/2},\si_R)}_{L^\ge(\g_{1/2},\si_R)}(Cl(\g_{1/2})).
\end{equation}
By Proposition \ref{lii}, $M$ is a $\si_R$-twisted $F(\g_{1/2})$-module. 
Recall that
$$
L^{ne}=\frac{1}{2}
         \sum_{\alpha \in S_{1/2}} : (T \Phi^{\alpha})
         \Phi_{\alpha}:,
$$
 and that $M$ is spanned by  monomials 
\begin{align*} (t^{\mu_1}\Phi_1)(t^{\mu_2}\Phi_2)\cdots(t^{\mu_k}\Phi_k),\quad \mu_i\in-(\tfrac{1}{2}+\ZZ_+).\end{align*}
Since 
$$
[L_0^{ne,\tw},t^{\mu}\Phi_{u}]=-(\mu+\tfrac{1}{2})t^{\mu}\Phi_{u},
$$
the action of $L_0^{ne,\tw}$ is semisimple and the eigenspace decomposition gives a grading
$$
M=\bigoplus_{n\ge 0}M_n,
$$
turning $M$ into a positive energy representation.
Since 
$$
M_0=span((t^{-\tfrac{1}{2}}\Phi_1)(t^{-\tfrac{1}{2}}\Phi_2)\cdots(t^{-\tfrac{1}{2}}\Phi_k) )=Cl(\g_{1/2})
$$
and $(\Phi_{u})_0=(\Phi_{u})_{(-1/2)}$, the action of $(\Phi_{u})_0$ is given by left multiplication by $t^{-1/2}\Phi_{u}$. Since the action of $Cl(\g_{1/2})$ on itself by left multiplication is faithful and this action factors to $Zhu_{\si_R}(F(\g_{1/2}))$ we deduce that the homomorphism from  $Cl(\g_{1/2})$ to $Zhu_{\si_R}(F(\g_{1/2}))$ is injective.
\end{proof}

\begin{lemma}\label{l1}Consider 
${L_0^{ne,\rm tw}}-\tfrac{1}{16}\dim\g_{1/2}$ as an operator on $F(\g_{1/2},\s_R)$. Then the action is semisimple with eigenvalues in $\tfrac{1}{2}\ZZ_+$.
The corresponding eigenspace decomposition 
\begin{equation}\label{FSgr}
F(\g_{1/2},\s_R)=\bigoplus_{n\in\tfrac{1}{2}\ZZ_+}F(\g_{1/2},\s_R)_n
\end{equation}
 gives to $F(\g_{1/2},\s_R)$ the structure of a $\si_R$-twisted positive energy representation of $F(\g_{1/2})$.

The space $F(\g_{1/2},\s_R)_0$ of minimal energy is naturally isomorphic, as a $Zhu_{R}(F(\g_{1/2}))$-module, to the Clifford module
\begin{equation}\label{cm}
CM=Cl(t^{-1/2}\g_{1/2})/(Cl(t^{-1/2}\g_{1/2})\fn_{1/2}(\si_R)_+).
\end{equation}
\end{lemma}
\begin{proof}
Recall that
$$
L^{ne}=\frac{1}{2}
         \sum_{\alpha \in S_{1/2}} : (T \Phi^{\alpha})
         \Phi_{\alpha}:.
$$
We claim that 
\begin{equation}\label{Lne}
L^{ne,\tw}_0\cdot 1=\tfrac{1}{16}\dim\g_{1/2}.
\end{equation}
To check \eqref{Lne}, start with \eqref{Borcherds} putting $n=-1,\mu=\nu=1/2$:
%
\begin{align*} &  ({(T \Phi^{\alpha})}_{(-1)}\Phi_{\alpha})_{(1)}\cdot 1
+\binom{1/2}{2}({(T \Phi^{\alpha})}_{(1)}\Phi_{\alpha})_{(-1)}\cdot 1\\
    &= \sum_{j \in \Z_+}  (1/2+j)(\Phi_{\alpha})_{(-1/2-j)}{( \Phi^{\alpha})}_{(-1/2+j)}) \cdot 1,\, 
\end{align*}
which gives
$$({(T \Phi^{\alpha})}_{(-1)}\Phi_{\alpha})_{(1)}\cdot 1
+1/8 = \tfrac{1}{2}(\Phi_{\alpha})_{(-1/2)}{(\Phi^{\alpha})}_{(-1/2)}) \cdot 1. 
$$
Assume $\g_\a\subset \n_{1/2}(\si_R)_+$. Then 
$$\tfrac{1}{2}(\Phi_{\alpha})_{(-1/2)}{(\Phi^{\alpha})}_{(-1/2)}) \cdot 1=\tfrac{1}{2}[(\Phi_{\alpha})_{(-1/2)},{(\Phi^{\alpha})}_{(-1/2)})]=1/2.$$
Assume $\g_\a\subset \g^0_{1/2}(\si_R)$. Then 
$$\tfrac{1}{2}(\Phi_{\alpha})_{(-1/2)}{(\Phi^{\alpha})}_{(-1/2)}) \cdot 1=\tfrac{1}{4}[(\Phi_{\alpha})_{(-1/2)},{(\Phi^{\alpha})}_{(-1/2)})]=1/4.$$
Summing up 
$$L_0^{ne,tw}\cdot 1= \tfrac{1}{2}(-\tfrac{1}{8} \dim\g_{1/2} + \tfrac{1}{4}\cdot 2\cdot \dim  \n_{1/2}(\si_R)_++\tfrac{1}{4}\e(\si_R))=\tfrac{1}{16}\dim\g_{1/2}.
$$

Since $F(\g_{1/2},\s_R)$ is generated by  monomials 
\begin{align}\label{mon} (t^{\mu_1}\Phi_1)(t^{\mu_2}\Phi_2)\cdots(t^{\mu_k}\Phi_k),\quad \mu_i\in-(\tfrac{1}{2}+\ZZ_+)\end{align}
and
$$
[L_0^{ne,\tw},t^{\mu}\Phi_{u}]=-(\mu+\tfrac{1}{2})t^{\mu}\Phi_{u}
$$
\eqref{FSgr} follows readily.

The space of minimal energy is
$$F(\g_{1/2},\s_R)_0=span((t^{-\tfrac{1}{2}}\Phi_1)(t^{-\tfrac{1}{2}}\Phi_2)\cdots(t^{-\tfrac{1}{2}}\Phi_k) ).
$$
On $F(\g_{1/2},\s_R)_0$, $[\Phi_{u}]\in  Zhu_{R}(F(\g_{1/2}))$ acts by the action of $(\Phi_{u})_0=(\Phi_{u})_{(-1/2)}$, which is left multiplication by $t^{-\tfrac{1}{2}}\Phi_u$.
\end{proof}

Let $\phi$ be an almost compact conjugate linear involution of $\g$.
Set 
$$
\omega=e^{L_1}g,\quad g(a)=e^{-\pi\sqrt{-}1(\D_a+\tfrac{1}{2}p(a))}\phi(a),
$$
so that 
$$\omega([\Phi_u])=-[\Phi_{\phi(u)}].$$

\begin{lemma}\label{Fsigma} $F(\g_{1/2},\s_R)$ is a unitary $\si_R$-twisted $F(\g_{1/2})$-module.\end{lemma}
\begin{proof} To prove the existence of a $\omega$-invariant Hermitian form on $F(\g_{1/2},\s_R)$, we use \cite[Proposition 6.7]{KMP}. According to this result, it suffices
to show that there exists  an invariant Hermitian form on $F(\g_{1/2},\s_R)_0$, which in Lemma \ref{l1} is identified with the Clifford module  $CM$ \eqref{cm}. We  identify as vector spaces $CM$ and $\bigwedge \fn_{1/2}(\si_R)_-$ via $u_1\wedge\ldots\wedge u_k\mapsto (t^{-1/2}\Phi_{u_1})\ldots (t^{-1/2}\Phi_{u_k}),\,u_i\in \fn_{1/2}(\si_R)_-$. Define an Hermitian  form on $\bigwedge \fn_{1/2}(\si_R)_-$ by determinants:
\begin{equation}\label{fh}(w_{i_1}\wedge w_{i_2}\ldots \wedge w_{i_h}, w_{j_1}\wedge w_{j_2}\ldots \wedge w_{j_k})=\d_{hk}(-1)^k\det(\langle \phi(w_{i_r}),w_{j_s}\rangle_{ne}).\end{equation}
We now check that it is $\omega$-invariant. Since $\phi$ is an almost compact involution, we have that $\phi(h_0)=-h_0$, hence $\phi(\fn_{1/2}(\si_R)_+)=\fn'_{1/2}(\si_R)_-$. 
Consider $u_1\in \fn_{1/2}(\si_R)_-$ (resp. $u_1\in  \fn_{1/2}(\si_R)_+$); note that $t^{-1/2}\Phi_{u_1}\cdot  t^{-1/2}\Phi_{u_2}\ldots  t^{-1/2}\Phi_{u_k}$ maps, under our identifications,  to $u_1\wedge \ldots\wedge u_k$ 
(resp. $\sum_r(-1)^{r+1} \langle u_1, u_r\rangle_{ne} u_1\wedge \ldots \widehat{u_r} \wedge \ldots u_k$). If $u_1\in  \fn_{1/2}(\si_R)_+$

\begin{align*}
&(\phi(u_1)\wedge u_{2}\wedge u_{3}\ldots \wedge u_{k},  w_{1}\wedge w_{2}\ldots \wedge w_{k})=(-1)^k\det(\langle \phi(u_i),w_j\rangle_{ne})\\
&=(-1)^k\sum_{r} (-1)^{r+1} \langle u_1,w_r\rangle_{ne} \det\left((\langle \phi(u_i),w_r\rangle_{ne})\right)_{1 r}\\
&=-(u_2\wedge u_{3}\ldots \wedge u_{k}, u_1\cdot w_{1}\wedge w_{2}\ldots \wedge w_{k}).
\end{align*}
 If $u_1\in  \fn_{1/2}(\si_R)_-$ one argues similarly.\par
 Recall from \cite{KMP1} that $\langle \phi(\cdot),\cdot \rangle_{ne}$ is negative definite on $\g_{1/2}$.
 Let $\{u_1,\ldots,u_k\}$ be a basis of $\fn_{1/2}(\si_R)_-$ such that $\langle \phi(u_i),u_j\rangle_{ne}=-\d_{ij}$;  then 
 $\{u_{i_1}\wedge\ldots\wedge u_{i_k}\mid i_1<\ldots<i_k\}$ is an orthonormal basis of $\bigwedge \fn_{1/2}(\si_R)_-$ for the Hermitian form \eqref{fh}. \
 We now check that the invariant form $H(\cdot,\cdot)$ on $F(\g_{1/2},\si_R)$ which restricts to \eqref{fh} on $F(\g_{1/2},\si_R)_0$ is positive definite.
 Let $\{v_1,\ldots,v_s\}$ be a basis of $\fg_{1/2}$ such that $\langle \phi(v_i),v_j\rangle_{ne}=-\d_{ij}$; let $\{w_1,\ldots,w_t\}$ be an orthonormal basis of  $F(\g_{1/2},\si_R)_0$.  Then the vectors
 \begin{equation} (t^{\mu_1}\Phi_{v_{i_1}})(t^{\mu_2}\Phi_{v_{i_2}})\cdots(t^{\mu_k}\Phi_{v_{i_k}})w_j,\quad \mu_i\in-(\tfrac{1}{2}+\nat), 1\leq j\leq t,\end{equation}
where the pairs $(\mu_j,i_j)$ are ordered lexicographically, form a basis of $F(\g_{1/2},\si_R)$. We claim that this basis is orthonormal. Indeed,  the $\Phi_u$ are primary, so 
$H(m,t^\mu\Phi_u m')=-H(t^{-\mu}\Phi_{\phi(u)}m,m')$. It follows that 
\begin{align*}
&H((t^{\mu_1}\Phi_{v_{i_1}})(t^{\mu_2}\Phi_{v_{i_2}})\cdots(t^{\mu_k}\Phi_{v_{i_k}})w_j,(t^{\nu_1}\Phi_{v_{j_1}})(t^{\nu_2}\Phi_{v_{j_2}})\cdots(t^{\nu_k}\Phi_{v_{j_k}})w_s)\\
&=H((t^{-\nu_k}\Phi_{-\phi(v_{j_k})})\cdots(t^{-\nu_2}\Phi_{-\phi(v_{j_2})})(t^{-\nu_1}\Phi_{-\phi(v_{j_1})})(t^{\mu_1}\Phi_{v_{i_1}})(t^{\mu_2}\Phi_{v_{i_2}})\cdots(t^{\mu_k}\Phi_{v_{i_k}})w_j,w_s)
\\&=\prod_i\d_{\mu_i\,\nu_i}\prod_r\d_{i_r\,j_r}\d_{j\,s}.
\end{align*}
\end{proof}

Our second task is to construct unitary $\si_R$-twisted representations of $\Wu$ using the free field realization  introduced  in \cite[Theorem 5.2]{KW1}; it is the embedding 
$$\Psi: W^k_{\min}(\g)\to \mathcal V^k:=V^{k+h^\vee}(\C x)\otimes V^{\alpha_k}(\g^\natural)\otimes F(\g_{1/2})
$$  
explicitly given on the generators of $W^k_{\min}(\g)$ by 
\begin{align}
\label{FFR1}&   J^{\{ b \}} \mapsto b + \frac{1}{2}
      \sum_{\alpha \in S_{1/2}}: \Phi^{\alpha}
      \Phi_{[w_{\alpha},b]}: (b \in \fg^{\natural}), \\
\label{FFR2}&    G^{\{ v \}} \mapsto \sum_{\alpha \in S_{1/2}}
         : [v,w_{\alpha}]\Phi^{\alpha}:
         -(k+1)\sum_{\alpha \in S_{1/2}}
         (v|w_{\alpha})  T \Phi^{\alpha}\\
\notag&     +\frac{1}{3} \sum_{\alpha ,\beta \in S_{1/2}}
         : \Phi^{\alpha}
         \Phi^{\beta}\Phi_{[w_{\beta},[w_{\alpha},v]]}
         : (v \in \fg_{-1/2})\, , \\
\label{FFR4}&       L \mapsto \frac{1}{2 (k+h^\vee)}
         \sum_{\alpha \in S_0} :a_{\alpha}
         a^{\alpha}:+ \frac{k+1}{k+h^\vee}T x +\frac{1}{2}
         \sum_{\alpha \in S_{1/2}} : (T \Phi^{\alpha})
         \Phi_{\alpha}: .
\end{align}
Using the normalization 
\begin{equation}\label{norm}a=\sqrt{-1}\frac{\sqrt{2}}{\sqrt{|k+h^\vee|}}x,
\end{equation}
we have $V^{k+h^\vee}(\C x)=M(1)$. Here $M(1)$ is the Heisenberg vertex algebra (free boson) generated by the element $a$ with $\l$-bracket 
$$
[a_\l a]=\l.
$$

Recall that we denoted by $\xi$ the highest weight of the $\g^\natural$-module $\g_{-1/2}$. Recall from \cite{KMP1} that an element $\nu\in P^+_k$ is called an {\it extremal weight} if 
$\nu+\xi$ doesn't lie in $P^+_k$.
 Let $\nu\in P^+_k$ be a non extremal weight. Let $L(\nu)$ be the irreducible $V^{\alpha_k}(\g^\natural)$-module of highest weight $\nu$. Fix $\mu\in\C$ and let $M(1,\mu)$ be the Verma module for $M(1)$ such that $a_0$ acts as the multiplication by  $\mu$.
Clearly $Id_{M(1)}\otimes Id_{V^{\alpha_k}(\g^\natural)}\otimes \si_R$ is an automorphism of $\mathcal V^k$ and 
\begin{equation}\label{Nmunu} M(1,\mu)\otimes L(\nu)\otimes F(\g_{1/2},\s_R)\end{equation}
is a  $Id_{M(1)}\otimes Id_{V^{\alpha_k}(\g^\natural)}\otimes \si_R$-twisted $\mathcal V^k$-module.

Since $\Psi\circ \si_R=(Id_{M(1)}\otimes Id_{V^{\alpha_k}(\g^\natural)}\otimes \si_R)\circ \Psi$, the restriction to $\Psi(\Wu)$ defines the structure of a  $\si_R$-twisted $\Wu$-module on the $\mathcal V^k$-module in \eqref{Nmunu}.

Set 
\begin{equation}\label{sfix}s_k=\sqrt{-1}\frac{(k+1)}{\sqrt{2|k+h^\vee|}}.
\end{equation}

\begin{proposition}\label{FNME}Fix $\mu \in \C$ and  $\nu\in (\h^\natural)^*$. Set
$$
v(\mu,\nu)=v_{\mu}\otimes v_{\nu}\otimes1,
$$
and
\begin{equation}\label{nmunu}
N(\mu,\nu)=\Psi(W^k_{\min}(\g)). v(\mu,\nu)\subset M(1,\mu)\otimes L(\nu)\otimes F(\g_{1/2},\si_R).
\end{equation}
If 
\begin{equation}\label{ell0}
\ell_0(\mu,\nu)=\frac{\mu^2}{2}-s_k\mu+\frac{(\nu|\nu+2\rho^\natural)}{2(k+h^\vee)}+\tfrac{1}{16}\dim\g_{1/2},
\end{equation}
then $N(\mu,\nu)$ is a highest weight $\Wu$-module with highest weight vector  $v(\mu,\nu)$,  of highest weight $(\nu+\rho_R,\ell_0(\mu,\nu))$.
\end{proposition}
\begin{proof}
By \eqref{Borcherds} we have
\begin{align}\label{fermbor1}
&: \Phi^{\alpha}
      \Phi_{[w_{\alpha},b]}:^{\tw}_r=\half\langle w^\a, [w_{\alpha},b]\rangle_{ne}\d_{r,0}+ \sum_{j \in \Z_+} 
((\Phi^{\alpha})^{\tw}_{-1-j} (\Phi_{[w_{\alpha},b]})^{\tw}_{r+j+1}- (\Phi_{[w_{\alpha},b]})^{\tw}_{r-j} (\Phi^{\alpha})^{\tw}_{j})\, ,
\\
&: T \Phi^{\alpha}
         \Phi_{\alpha}:^{\tw}_r
     = \sum_{j \in \Z_+} 
(\half+j) (\Phi^{\alpha})^{\tw}_{-1-j} (\Phi_{\alpha})^{\tw}_{r+j+1}+ (\Phi_{\alpha})^{\tw}_{r-j} ( \Phi^{\alpha})^{\tw}_{j})-1/8
     \d_{r,0}\, ,\\
&: [v,w_{\alpha}]\Phi^{\alpha}:^{\tw}_r=\sum_{j \in \Z} 
[v,w_{\alpha}]_{-j} (\Phi^{\alpha})^{\tw}_{r+j}\, ,
\\
\label{fermbor2}
&: \Phi^{\alpha}
         \Phi^{\beta}\Phi_{[w_{\beta},[w_{\alpha},v]]}
         :^{\tw} _r=\half\langle w^{\alpha},w^\be\rangle_{ne}
       ( \Phi_{[w_{\beta},[w_{\alpha},v]]})^{\tw}_r
         -\half\langle w^{\alpha},[w_{\beta},[w_{\alpha},v]]\rangle_{ne}
         (\Phi^{\beta})^{\tw}_r\\
&+ \sum_{j \in \Z_+} 
((\Phi^{\alpha})^{\tw}_{-1-j} : \Phi^{\beta}\Phi_{[w_{\beta},[w_{\alpha},v]]}
         :^{\tw} _{r+j+1}+  : \Phi^{\beta}\Phi_{[w_{\beta},[w_{\alpha},v]]}
         : ^{\tw}_{r-j}( \Phi^{\alpha})^{\tw}_{j}).\notag
\end{align}
         
We start by checking \eqref{m3}. If $r>0$,
\begin{align*}
&\Psi(J^{\{b\}})_rv(\mu,\nu)=v_\mu\otimes b_rv_\nu\otimes 1+\half \sum_{\alpha \in S_{1/2}}v_\mu\otimes v_\nu\otimes : \Phi^{\alpha}
      \Phi_{[w_{\alpha},b]}:^{\tw}_r1
\\
&=\half \sum_{\alpha \in S_{1/2},j\in\ZZ_+}(v_\mu\otimes v_\nu\otimes  (\Phi^{\alpha})^{\tw}_{-j-1}
      (\Phi_{[w_{\alpha},b]})^{\tw}_{r+j+1}-v_\mu\otimes v_\nu\otimes  
      (\Phi_{[w_{\alpha},b]})^{\tw}_{r-j}(\Phi^{\alpha})^{\tw}_{j})1
\\&
=-\half \sum_{\alpha \in S_{1/2}}v_\mu\otimes v_\nu\otimes  
      (\Phi_{[w_{\alpha},b]})^{\tw}_{r}(\Phi^{\alpha})^{\tw}_{0}1=\half \sum_{\alpha \in S_{1/2}}v_\mu\otimes v_\nu\otimes (\Phi^{\alpha})^{\tw}_{0}(\Phi_{[w_{\alpha},b]})^{\tw}_{r} 1=0.
\end{align*}
Likewise
\begin{align*}
&\Psi(G^{\{v\}})_rv(\mu,\nu)=\\
&\sum_{j \in \Z} 
[v,w_{\alpha}]^{\tw}_{-j}(v_\mu\otimes v_\nu)\otimes (\Phi^{\alpha})^{\tw}_{r+j} 1 -(k+1)\sum_{\alpha \in S_{1/2}}
         (v|w_{\alpha})v_\mu\otimes v_\nu\otimes  ( T \Phi^{\alpha})^{\tw}_r1
\\&
+\frac{1}{6} \sum_{\alpha ,\beta \in S_{1/2}}v_\mu\otimes v_\nu\otimes (\langle w^{\alpha},w^\be\rangle_{ne}
       ( \Phi_{[w_{\beta},[w_{\alpha},v]]})^{\tw}_r
         -\half\langle w^{\alpha},[w_{\beta},[w_{\alpha},v]]\rangle_{ne}
         (\Phi^{\beta})^{\tw}_r)1
\\&
+\frac{1}{3} \sum_{\alpha ,\beta \in S_{1/2},j\in\ZZ_+}v_\mu\otimes v_\nu\otimes((\Phi^{\alpha})^{\tw}_{-1-j} : \Phi^{\beta}\Phi_{[w_{\beta},[w_{\alpha},v]]}
         :^{\tw} _{r+j+1}+  : \Phi^{\beta}\Phi_{[w_{\beta},[w_{\alpha},v]]}
         : ^{\tw}_{r-j}( \Phi^{\alpha})^{\tw}_{j})1
\\&
= (k+1)\sum_{\alpha \in S_{1/2}}
         (v|w_{\alpha})(r+\half)v_\mu\otimes v_\nu\otimes  (\Phi^{\alpha})^{\tw}_r1
+\frac{1}{3} \sum_{\alpha ,\beta \in S_{1/2}}v_\mu\otimes v_\nu\otimes  : \Phi^{\beta}\Phi_{[w_{\beta},[w_{\alpha},v]]}
         : ^{\tw}_{r}( \Phi^{\alpha})^{\tw}_{0}1
\\&=
\frac{1}{3} \sum_{\alpha ,\beta \in S_{1/2}}v_\mu\otimes v_\nu\otimes  : \Phi^{\beta}\Phi_{[w_{\beta},[w_{\alpha},v]]}
         : ^{\tw}_{r}( \Phi^{\alpha})^{\tw}_{0}1
\\&
=
\frac{1}{3} \sum_{\alpha ,\beta \in S_{1/2},j\in\Z_+}v_\mu\otimes v_\nu\otimes  
(( \Phi^{\beta})^{\tw}_{-j-1}(\Phi_{[w_{\beta},[w_{\alpha},v]]})^{\tw}_{r+j+1}
         -(\Phi_{[w_{\beta},[w_{\alpha},v]]})^{\tw}_{r-j}( \Phi^{\beta})^{\tw}_{j})( \Phi^{\alpha})^{\tw}_{0}1
\\&
=
-\frac{1}{3} \sum_{\alpha ,\beta \in S_{1/2}}v_\mu\otimes v_\nu\otimes  
(\Phi_{[w_{\beta},[w_{\alpha},v]]})^{\tw}_{r}( \Phi^{\beta})^{\tw}_{0}( \Phi^{\alpha})^{\tw}_{0}1=0.
\end{align*}
Finally
\begin{align*}
\Psi(L)_rv(\mu,\nu)&=\frac{1}{2 (k+h^\vee)}
         \sum_{\alpha \in S_0} :a_{\alpha}
         a^{\alpha}:_r(v_\mu\otimes v_\nu)\otimes1+ \frac{k+1}{k+h^\vee}(T x)_r v_\mu\otimes v_\nu\otimes 1
 \\&+\frac{1}{2}
         \sum_{\alpha \in S_{1/2}} v_\mu\otimes v_\nu\otimes: (T \Phi^{\alpha})
         \Phi_{\alpha}:^{\tw}_r1
\\&
=
         -(r+1) \frac{k+1}{k+h^\vee} x_r v_\mu\otimes v_\nu\otimes 1
         +\frac{1}{2}
         \sum_{\alpha \in S_{1/2}} v_\mu\otimes v_\nu\otimes: (T \Phi^{\alpha})
         \Phi_{\alpha}:^{\tw}_r1 .
\\&
=\frac{1}{2}
         \sum_{\alpha \in S_{1/2},j\in\Z_+} (\half+j)v_\mu\otimes v_\nu\otimes ((\Phi^{\alpha})^{\tw}_{-1-j} (\Phi_{\alpha})^{\tw}_{r+j+1}+ (\Phi_{\alpha})^{\tw}_{r-j} ( \Phi^{\alpha})^{\tw}_{j})1=0 .
\end{align*}
This concludes  the check that \eqref{m3} holds.
 Let us now verify \eqref{m4}: if $a\in \g^\natural$, then
\begin{equation}\label{spin}
\Psi(J^{\{a\}})_0v(\mu,\nu)=v_\mu\otimes a_0v_\nu\otimes 1+\half \sum_{\alpha \in S_{1/2}}v_\mu\otimes v_\nu\otimes : \Phi^{\alpha}
      \Phi_{[w_{\alpha},a]}:^{\tw}_01.
\end{equation}
If, in particular,  $a\in \n_0(\si_R)_+$ then, since $\langle w^\a,[w_\a,a]\rangle_{ne}=0$,
$$
\Psi(J^{\{a\}})_0v(\mu,\nu)=\half \sum_{\alpha \in S_{1/2},j\in\ZZ_+}(v_\mu\otimes v_\nu\otimes  ((\Phi^{\alpha})^{\tw}_{-j-1}
      (\Phi_{[w_{\alpha},a]})^{\tw}_{j+1}- 
      (\Phi_{[w_{\alpha},a]})^{\tw}_{-j}(\Phi^{\alpha})^{\tw}_{j})1
$$
$$
=-\half \sum_{\alpha \in S_{1/2}}(v_\mu\otimes v_\nu\otimes   
      (\Phi_{[w_{\alpha},a]})^{\tw}_{0}(\Phi^{\alpha})^{\tw}_{0}1.
$$
If $w_\a\in \n_{1/2}(\si_R)'$, then $w^\a\in\n_{1/2}(\si_R)_+$ and $(\Phi^{\alpha})^{\tw}_{0}1=0$. Since
$\langle w^\a,[w_\a,a]\rangle_{ne}=0$,
$$
      (\Phi_{[w_{\alpha},a]})^{\tw}_{0}(\Phi^{\alpha})^{\tw}_{0}1=- (\Phi^{\alpha})^{\tw}_{0}(\Phi_{[w_{\alpha},a]})^{\tw}_{0}1.
$$
If $w_a \notin \n_{1/2}(\si_R)'$, then $[w_\a,a]\in\n_{1/2}(\si_R)_+$ so $(\Phi_{[w_{\alpha},a]})^{\tw}_{0}1=0$.

We now turn to checking \eqref{m1}: if $h\in\h^\natural$ then
\begin{align*}
&\Psi(J^{\{h\}})_0v(\mu,\nu)=v_\mu\otimes h_0v_\nu\otimes 1+\half \sum_{\alpha \in S_{1/2}}v_\mu\otimes v_\nu\otimes : \Phi^{\alpha}
      \Phi_{[w_{\alpha},h]}:^{\tw}_01
\\&=\nu(h)v_\mu\otimes v_\nu\otimes 1+\half \sum_{\alpha \in S_{1/2}}\a(h)(v_\mu\otimes v_\nu\otimes 
      (\Phi_{{\alpha}})^{\tw}_{0}(\Phi^{\alpha})^{\tw}_{0}1-\tfrac{1}{4}\sum_{\a\in S_{1/2}}\a(h)(v_\mu\otimes v_\nu\otimes 1)
\\&=\nu(h)v_\mu\otimes v_\nu\otimes 1+\half \sum_{\alpha \in S^+_{1/2}}\a(h)(v_\mu\otimes v_\nu\otimes 
      (\Phi_{{\alpha}})^{\tw}_{0}(\Phi^{\alpha})^{\tw}_{0}1-\tfrac{1}{4}\sum_{\a\in S_{1/2}}\a(h)(v_\mu\otimes v_\nu\otimes 1)
\\&=\nu(h)v_\mu\otimes v_\nu\otimes 1+(\half \sum_{\alpha \in S^+_{1/2}}\a(h)-\tfrac{1}{4}\sum_{\a\in S_{1/2}}\a(h))(v_\mu\otimes v_\nu\otimes 1),
\end{align*}
hence
\begin{equation}\label{azcl}
\Psi(J^{\{h\}})_0v(\mu,\nu)=(\nu+\rho_R)(h)(v_\mu\otimes v_\nu\otimes 1).\end{equation}
We now compute $\ell_0$ checking \eqref{m2}: recall from \cite[(9.8)]{KMP1} that
$$
\Psi(L)=\half:aa:+s_kTa+L^{\g^\natural}+L^{ne},
$$
and it is easy to check that $\half:aa:_0+s_k(Ta)_0v_\mu=\half\mu^2-s_k\mu$. Using \eqref{Lne}, one readily obtains
$$
L_0(v_{\mu}\otimes v_{\nu}\otimes 1)=\left(\frac{\mu^2}{2}-s_k\mu+\frac{(\nu|\nu+2\rho^\natural)}{2(k+h^\vee)}+\tfrac{1}{16}\dim\g_{1/2}\right)(v_{\mu}\otimes v_{\nu}\otimes 1 ).
$$

It remains only to check \eqref{m5}: if $v\in\n_{-1/2}(\si_R)_+$ then
\begin{align*}
&\Psi(G^{\{ v \}})_0 (v_\mu\otimes v_\nu\otimes 1)
\\&= \sum_{\alpha \in S_{1/2}}
         : [v,w_{\alpha}]\Phi^{\alpha}:^{\tw}_0(v_\mu\otimes v_\nu\otimes 1)
     -(k+1)\sum_{\alpha \in S_{1/2}}
         (v|w_{\alpha}) v_\mu\otimes v_\nu\otimes (T \Phi^{\alpha})^{\tw}_01\\
    &+\tfrac{1}{3} \sum_{\alpha ,\beta \in S_{1/2}}
        v_\mu\otimes v_\nu\otimes : \Phi^{\alpha}
         \Phi^{\beta}\Phi_{[w_{\beta},[w_{\alpha},v]]}:^{\tw}_0 1
\\&= \sum_{\alpha \in S_{1/2},j\in\Z}
         [v,w_{\alpha}]_{-j}(v_\mu\otimes v_\nu)\otimes ( \Phi^{\alpha})^{\tw}_{j}1
     +\tfrac{k+1}{2}\sum_{\alpha \in S_{1/2}}
         (v|w_{\alpha}) v_\mu\otimes v_\nu\otimes (\Phi^{\alpha})^{\tw}_01
\\&+\tfrac{1}{3} \sum_{\alpha ,\beta \in S_{1/2}}
        v_\mu\otimes v_\nu\otimes : \Phi^{\alpha}
         \Phi^{\beta}\Phi_{[w_{\beta},[w_{\alpha},v]]}:^{\tw}_0 1, 
\end{align*}
so that 
\begin{align*}
&\Psi(G^{\{ v \}})_0 (v_\mu\otimes v_\nu\otimes 1)= \sum_{\alpha \in S^+_{1/2}\cup S^{-,0}_{1/2}}
         [v,w_{\alpha}]_{0}(v_\mu\otimes v_\nu)\otimes ( \Phi^{\alpha})^{\tw}_{0}1
   \\&+\tfrac{k+1}{2}\sum_{\alpha \in S^+_{1/2}\cup S^{-,0}_{1/2}}
         (v|w_{\alpha}) v_\mu\otimes v_\nu\otimes (\Phi^{\alpha})^{\tw}_01
+\tfrac{1}{3} \sum_{\alpha ,\beta \in S_{1/2}}
        v_\mu\otimes v_\nu\otimes : \Phi^{\alpha}
         \Phi^{\beta}\Phi_{[w_{\beta},[w_{\alpha},v]]}:^{\tw}_0 1.
   \end{align*}
Note that $h_0$ has positive eigenvalues on $v$ (which we can assume to be a root vector) and non-negative eigenvalues on 
$w_\a,\, \alpha \in S^+_{1/2}\cup S^{-,0}_{1/2}$. Hence $[v,w_\a]\in \n_0(\si_R)_+$, and the first summand vanishes; moreover, for the same reason,  $(v|w_\a)=0$, and the second summand vanishes. We are left with evaluating the third summand.

\begin{align}
&\label{prima}\sum_{\alpha ,\beta \in S_{1/2}}: \Phi^{\alpha}
         \Phi^{\beta}\Phi_{[w_{\beta},[w_{\alpha},v]]}
         :^{\tw} _0 1=\\\notag&\half\sum_{\alpha ,\beta \in S_{1/2}}(\langle w^{\alpha},w^\be\rangle_{ne}
       ( \Phi_{[w_{\beta},[w_{\alpha},v]]})^{\tw}_0
         -\langle w^{\alpha},[w_{\beta},[w_{\alpha},v]]\rangle_{ne}
         (\Phi^{\beta})^{\tw}_0 ) 1
\\&+ \sum_{\alpha ,\beta \in S_{1/2}}\sum_{j \in \Z_+} 
((\Phi^{\alpha})^{\tw}_{-1-j} : \Phi^{\beta}\Phi_{[w_{\beta},[w_{\alpha},v]]}
         :^{\tw} _{j+1}+  : \Phi^{\beta}\Phi_{[w_{\beta},[w_{\alpha},v]]}
         : ^{\tw}_{-j}( \Phi^{\alpha})^{\tw}_{j}) 1. \label{terza}
     \end{align}
     We evaluate the summands in \eqref{terza}. The first is 
   \begin{equation*}
      \sum_{j \in \Z_+} 
((\Phi^{\alpha})^{\tw}_{-1-j} : \Phi^{\beta}\Phi_{[w_{\beta},[w_{\alpha},v]]}:^{\tw} _{j+1} )1=0.
\end{equation*}  
      The second is 
	\begin{align*}&\sum_{j \in \Z_+}  : \Phi^{\beta}\Phi_{[w_{\beta},[w_{\alpha},v]]}: ^{\tw}_{-j}( \Phi^{\alpha})^{\tw}_{j}) 1= : \Phi^{\beta}\Phi_{[w_{\beta},[w_{\alpha},v]]}: ^{\tw}_{0}( \Phi^{\alpha})^{\tw}_{0} 1=\\
	&\left(\half\langle w^\beta, [w_{\beta},[w_{\alpha},v]]\rangle_{ne}+ \sum_{j \in \Z_+} 
((\Phi^{\beta})^{\tw}_{-1-j} (\Phi_{[w_{\beta},[w_{\alpha},v]]})^{\tw}_{j+1}- (\Phi_{[w_{\beta},[w_{\alpha},v]]})^{\tw}_{-j} (\Phi^{\beta})^{\tw}_{j})\, \right)( \Phi^{\alpha})^{\tw}_{0} 1
\\
&=\left(\half\langle w^\beta, [w_{\beta},[w_{\alpha},v]]\rangle_{ne}- (\Phi_{[w_{\beta},[w_{\alpha},v]]})^{\tw}_{0} (\Phi^{\beta})^{\tw}_{0})\, \right)( \Phi^{\alpha})^{\tw}_{0} 1.
\end{align*} 
Hence \eqref{prima} becomes
\begin{align*}
&\label{prima}\sum_{\alpha ,\beta \in S_{1/2}}: \Phi^{\alpha}
         \Phi^{\beta}\Phi_{[w_{\beta},[w_{\alpha},v]]}
         :^{\tw} _0 1=\\\notag&\half\sum_{\alpha ,\beta \in S_{1/2}}(\langle w^{\alpha},w^\be\rangle_{ne}
       ( \Phi_{[w_{\beta},[w_{\alpha},v]]})^{\tw}_0
         -\langle w^{\alpha},[w_{\beta},[w_{\alpha},v]]\rangle_{ne}
         (\Phi^{\beta})^{\tw}_0 ) 1\\
         &+\sum_{\a,\beta\in S_{1/2}}\left(\half\langle w^\beta, [w_{\beta},[w_{\alpha},v]]\rangle_{ne}- (\Phi_{[w_{\beta},[w_{\alpha},v]]})^{\tw}_{0} (\Phi^{\beta})^{\tw}_{0})\, \right)( \Phi^{\alpha})^{\tw}_{0} 1.
\end{align*}
The first term equals 
$\half\sum_{\alpha  \in S_{1/2}}(
       ( \Phi_{[w^{\a},[w_{\alpha},v]]})^{\tw}_0 1 $, which vanishes since $$[w^{\a},[w_{\alpha},v]]\in \n_{1/2}(\si_R)_+.$$ 
For the second term, note that we may  assume $\beta\in  S_{1/2}^+\cup S_{1/2}^{0,-}$; but then 
$\langle w^{\alpha},[w_{\beta},[w_{\alpha},v]]\rangle_{ne}=0$, since $[w_{\beta},[w_{\alpha},v]$ cannot have weight $\a$. The third term is handled in the same way.  The last term equals

\begin{equation}\label{ultima}\sum_{\a,\beta\in S_{1/2}^+\cup S_{1/2}^{0,-}, \a\ne -\beta}(\Phi_{[w_{\beta},[w_{\alpha},v]]})^{\tw}_{0} (\Phi^{\beta})^{\tw}_{0}\, ( \Phi^{\alpha})^{\tw}_{0} 1+\sum_{\a\in S_{1/2}^+\cup S_{1/2}^{0,-}}(\Phi_{[w_{-\a},[w_{\alpha},v]]})^{\tw}_{0} (\Phi^{-\a})^{\tw}_{0}\, ( \Phi^{\alpha})^{\tw}_{0} 1.
\end{equation}
The first term in \eqref{ultima} vanishes since $$(\Phi_{[w_{\beta},[w_{\alpha},v]]})^{\tw}_{0} (\Phi^{\beta})^{\tw}_{0}\, ( \Phi^{\alpha})^{\tw}_{0} 1= (\Phi^{\beta})^{\tw}_{0}\, ( \Phi^{\alpha})^{\tw}_{0} (\Phi_{[w_{\beta},[w_{\alpha},v]]})^{\tw}_{0}1=0.$$ For the second term
\begin{align*}&\sum_{\a\in S_{1/2}^+\cup S_{1/2}^{0,-}}(\Phi_{[w^{\a},[w_{\alpha},v]]})^{\tw}_{0} (\Phi_{\a})^{\tw}_{0})\, ( \Phi^{\alpha})^{\tw}_{0} 1\\&=\sum_{\a\in S_{1/2}^{0,-}}(\Phi_{[w^{\a},[w_{\alpha},v]]})^{\tw}_{0} (( \Phi^{\alpha})^{\tw}_{0})^2 1+
\sum_{\a\in S_{1/2}^+}(\Phi_{[w^{\a},[w_{\alpha},v]]})^{\tw}_{0}  1
\\&=\half\sum_{\a\in S_{1/2}^{0,-}}(\Phi_{[w^{\a},[w_{\alpha},v]]})^{\tw}_{0} 1+
\sum_{\a\in S_{1/2}^+}(\Phi_{[w^{\a},[w_{\alpha},v]]})^{\tw}_{0}  1 =0 .
\end{align*}
In the last equality we used that $(( \Phi^{\alpha})^{\tw}_{0})^2 =\half [(\Phi^{\alpha})^{\tw}_{0},(\Phi^{\alpha})^{\tw}_{0}]=\half$.
\end{proof}

Recall \cite{KMP1} that $\mu\in P^+_k$ is called {\sl extremal} if $\mu(\theta_i^\vee)\le M_i(k)+\chi_i$ for all $i$ (see Table \ref{numerical}).
 \begin{theorem}\label{bknu}Let 
\begin{equation}\label{BKnurhoR}
B(k,\nu,\rho_R)=-\frac{(k+1)^2}{4(k+h^\vee)}+\frac{(\nu-\rho_R|\nu-\rho_R+2\rho^\natural)}{2(k+h^\vee)}+\tfrac{1}{16}\dim\g_{1/2}.
\end{equation}

If $\nu-\rho_R\in P^+_k$ is non extremal then, for all $\ell\ge B(k,\nu,\rho_R)$, the $\si_R$-twisted $\Wu$-module $L^W(\nu,\ell_0)$   is unitary. In particular, $\nu\in P^+_k$.
\end{theorem}
\begin{proof} 

Fix $\mu\in\R$. 
By Proposition \ref{FNME}, $N(\mu+s_k,\nu-\rho_R)$  is a highest weight module with highest weight $(\nu,\ell_0(\mu+s_k,\nu))$ (cf. \eqref{ell0}) so that, using \eqref{sfix},
$$
\ell_0(\mu+s_k,\nu)=\frac{\mu^2}{2}-\frac{(k+1)^2}{4(k+h^\vee)}+\frac{({\nu-\rho_R}|{\nu-\rho_R}+2\rho^\natural)}{2(k+h^\vee)}+\tfrac{1}{16}\dim\g_{1/2}.
$$
Hence, for all $\ell\ge B(k,\nu,\rho_R)$, choosing  $\mu=2\sqrt{\ell_0-B(k,\nu,\rho_R)}$,
$N(\mu+s_k,\nu-\rho_r)$  is a highest weight module with highest weight $(\nu,\ell)$.

Since $M_i(k)+\chi_i\in\ZZ_+$ and $(\nu-\rho_R)(\theta_i^\vee)\le M_i(k)+\chi_i $ for all $i$,  $L({\nu-\rho_R})$ is integrable for $V^{\a_k}(\g^\natural)$, hence it is an irreducible unitary  highest weight $V^{\a_k}(\g^\natural)$--module. 
Let
$$
(\,\cdot\,,\cdot\,)_{\mu+s_k}=H_{\mu+s_k}\otimes H^{\g^\natural}\otimes H^F,
$$
where $H_{\mu+s_k}$ is the invariant hermitian form on $M(1,\mu+s_k)$, $H^{\g^\natural}$ is the invariant hermitian form on 
$L({\nu-\rho_R})$ and $H^F$ is the invariant Hermitian form on $F(\g_{1/2},\si_R)$ constructed in Lemma \ref{Fsigma}.   By \cite[Proposition 9.2]{KMP1}, $(\,\cdot\,,\cdot\,)_{\mu+s_k}$ is a $\phi$-invariant form on $N(\mu+s_k,\nu-\rho_R)$.

Since $H_{\mu+s_k}$, $H^{\g^\natural}$, and  $H^F$ are all positive definite, $(\,\cdot\,,\cdot\,)_{\mu+s_k}$ is positive definite as well hence, by restriction,  $N(\mu+s_k,\nu-\rho_R)$ is unitary. In particular, $V^{\beta_k}(\g^\natural)(v_\mu\otimes v_{\nu-\rho_R}\otimes 1)$ is unitary, hence it has to be integrable, so  that $\nu\in P^+_k$.
\end{proof}

\begin{cor}\label{suffcond}
\begin{enumerate}
\item If $\g\ne spo(2|3), D(2,1;a)$ and $\nu\in P^+_k$, $\nu-\rho_R\in P^+, \ell\ge B(k,\nu,\rho_R)$,  then the $\si_R$-twisted $\Wu$-module $L^W(\nu,\ell)$ is unitary.
\item If $\g=spo(2|3),$ then for all $\nu\in P^+_k$, $\nu\ne0, \tfrac{M_1(k)}{2}\epsilon_1$ and $\ell\ge B(k,\nu,\rho_R)$,  there  the $\si_R$-twisted $\Wu$-module $L^W(\nu,\ell)$  is unitary.
\item If $\g=D(2,1;a),$ then for all $\nu\in P^+_k$ which do not lie in $\{r_3\e_3, r_2\e_2+M_2(k)\e_3\}$ if $\rho_R=\e_2$ or in $\{r_2\e_2, M_1(k)\e_2+r_3\e_3\}$ if $\rho_R=\e_3$ then,  for  $\ell\ge B(k,\nu,\rho_R)$, the $\si_R$-twisted $\Wu$-module $L^W(\nu,\ell)$  is unitary.\end{enumerate}
\end{cor}
\begin{proof} (1). 
Using Table \ref{numerical}, observe that  $\rho_R(\theta_1^\vee)= -\chi_1$ , Then, if $\nu\in P^+_k$, we have 
$$(\nu-\rho_R)(\theta_i^\vee)=\nu(\theta_1^\vee)-\rho_R(\theta_1^\vee)\leq M_1(k)+\chi_1,$$
so $\nu-\rho_R$ is not extremal; being dominant by assumption, it  lies in $P^+_k$, hence we can apply Theorem \ref{bknu}.

(2). We have $\nu=\tfrac{r}{2}\e_1,\,1\leq r \leq M_1(k)-1$, so 
$(\nu-\rho_R)(\theta_i^\vee)= \tfrac{r-1}{2}\e_1$ belongs to  $P^+_k$ and it is not extremal.

(3). We have $\nu=r_2\e_2+r_3\e_3,\,r_2,r_3\in\Z_+, r_2\leq M_1(k),\, r_3\leq M_2(k)$. So if $\rho_R=\e_2$ (resp. $\rho_R=\e_3$) $\nu-\rho_R$ is not dominant precisely when  $\nu=r_3\e_3$ (resp. $\nu=r_2\e_2$) and  it is extremal if $\nu=r_2\e_2+M_2(k)\e_3$ (resp. $\nu=M_1(k)\e_2+r_3\e_3$). 
\end{proof}

\section{Euler-Poincar\'e Characters}
First of all we specialize to the Ramond sector the results of Proposition \ref{Hhw}.  Recall from \eqref{eps} the definition of $\epsilon$.
\begin{lemma}\label{gammas}If $\si=\si_R$, then
$$
  \gamma'=2\rho_R-\rho-\tfrac{1}{2}\epsilon(\si_R)\theta/2,\ 
  \gamma_{1/2} =\rho_R-\frac{1}{4}\epsilon(\si_R)\theta/2
$$
and
\begin{displaymath}
  s_{fg}= \frac{k}{8(k+h^\vee)}(\dim \g_{1/2}-2\epsilon(\si_R))\, , \quad
     s_{gh}=-\frac{1}{16}\dim\g_{1/2}.
\end{displaymath}
\end{lemma}
\begin{proof}Specializing \eqref{eq:3.2} to $\si_R$ we find
$$
s_e=0,\ s_{u_i}=-\half, \text{ if }u_i\in\n_{1/2}(\si_R)_+,\ s_{u_i}=\half \text{ if } u_i\in\n_{1/2}(\si_R)_-,\ s_{a_i}=0\text{ if } a_i\in\n_{0}(\si_R)_+,\ 
$$
and
$$
s_{a_i}=1\text{ if }\ a_i\in\n_{0}(\si_R)_-,s_{u_i}=\half\text{ if }\ u_i\in\n_{-1/2}(\si_R)_+,\ s_{u_i}=\tfrac{3}{2}\text{ if }\ u_i\in\n_{-1/2}(\si_R)_-,\ s_f=1.
$$
so
$$
  \gamma'=\frac{1}{4}\sum_{\g_\a\subset \n_{1/2}(\si_R)_+}\!\!\!\!\! \a
-\frac{1}{4}\sum_{\g_\a\subset\n_{1/2}(\si_R)_-}\!\!\!\!\!  \alpha 
      +\frac{1}{2}\sum_{\g_\a\subset \n_{0}(\si_R)_-}\!\!\!\!\! \alpha
      -\frac{1}{4}\sum_{\g_\a\subset\n_{-1/2}(\si_R)_+} \!\!\!\!\! \alpha
      -\frac{3}{4}\sum_{\g_\a\subset \n_{-1/2}(\si_R)_-}\!\!\!\!\! \alpha
      -\frac{1}{2}\theta.
 $$
 Since:
 
  $\g_\a\subset \n_{1/2}(\si_R)_+$ if and only if $\a=\theta/2+\eta$   with $\eta\in\overline\D^+_{1/2}$,
  
    $\g_\a\subset \n_{-1/2}(\si_R)_+$ if and only if  $\g_\a\subset \g^0_{-1/2}(\si_R)$ or  $\a=-\theta/2+\eta$   with $\eta\in\overline\D^+_{1/2}$ ,
    
$\g_\a\subset \n_{1/2}(\si_R)_-$ if and only if $\g_\a\subset \g^0_{1/2}(\si_R)$ or  $\a=\theta/2-\eta$   with $\eta\in\overline\D^+_{1/2}$,

$\g_\a\subset \n_{-1/2}(\si_R)_-$ if and only if  $\a=-\theta/2-\eta$   with $\eta\in\overline\D^+_{1/2}$,

\noindent we can write
\begin{align*}
  \gamma&'=\frac{1}{4}\sum_{\eta \in \overline\D^+_{1/2}} (\theta/2+\eta)
-\frac{1}{4}\sum_{\eta \in \overline\D^+_{1/2}}(\theta/2-\eta)-\frac{1}{4}\epsilon(\si_R)\theta/2
 \\&
      -\frac{1}{2}\sum_{\alpha \in \n_{0}(\si_R)_+}\alpha
      -\frac{1}{4}\sum_{\eta \in \overline\D^+_{1/2}} (-\theta/2+\eta)+\frac{1}{4}\epsilon(\si_R)\theta/2
      -\frac{3}{4}\sum_{\eta \in \overline\D^+_{1/2}} (-\theta/2-\eta)         
      -\theta/2.
\end{align*}
 By an obvious calculation we find
\begin{align*}
  \gamma'&=\sum_{\eta \in \overline\D^+_{1/2}} \eta
 -\frac{1}{2}\sum_{\alpha \in \n_{0}(\si_R)_+}\alpha
     +(\dim\n_{1/2}(\si_R)_+-1)\theta/2
\\&=2\rho_R-\rho^\natural+\tfrac{1}{2}(\dim\g_{1/2}-\e(\s_R)-2)\theta/2
\\&=2\rho_R-\rho^\natural+(-1+\tfrac{1}{2}\dim\g_{1/2})\theta/2 -\half\e(\s_R)\theta/2
\\&=2\rho_R-\rho -\half\e(\s_R)\theta/2.
\end{align*}
 Likewise  $\gamma_{1/2}$ specializes to 
\begin{align*}
  \gamma_{1/2} &=\frac{1}{4}\sum_{\g_\a\subset \n_{1/2}(\si_R)_+}\!\!\!\!\! 
 \alpha -\frac{1}{4}\sum_{\g_\a\subset \n_{1/2}(\si_R)_-}\!\!\!\!\! 
 \alpha
\\&
 =\frac{1}{4}\sum_{\eta \in \overline\D^+_{1/2}} (\theta/2+\eta)-\frac{1}{4}\sum_{\eta \in \overline\D^+_{1/2}}(\theta/2-\eta)-\frac{1}{4}\epsilon(\si_R)\theta/2=\rho_R-\frac{1}{4}\epsilon(\si_R)\theta/2.
\end{align*}
Finally
$$
  s_{fg}=\frac{k}{4(k+h^\vee)}\left(\sum_{\g_\alpha \subset \n_{1/2}(\si_R)_+}
      \!\!3/4-\sum_{\g_\alpha \subset \n_{1/2}(\si_R)_-}
      1/4-\sum_{\g_\alpha \subset \n_{-1/2}(\si_R)_+}
      1/4+\sum_{\g_\alpha \subset \n_{-1/2}(\si_R)_-}
      3/4\right)
      $$
      $$
      =\frac{k}{8(k+h^\vee)}(3\dim \n_{1/2}(\si_R)_+-\dim \n_{1/2}(\si_R)_-)   =\frac{k}{8(k+h^\vee)}(\dim \g_{1/2}-2\e(\si_R)),
$$
      and
      $$     s_{gh}=-\frac{1}{4}
     \sum_{\alpha \in S_{1/2}} s^2_{\alpha}=-\frac{1}{16}\dim\g_{1/2}.
$$
\end{proof}

\begin{cor}\label{Hvermaisverma}
If $\widehat\L\in\ha^*$, then  $H_j(M(\widehat \L))=0$ if $j\ne0$ and
$
H_0(M(\widehat \L))=M^W(\nu,\ell(\widehat\L))$ with 
\begin{align}
 \ell(\widehat \L) &= \frac{1}{2(k+h^\vee)}\left( (\widehat\L_{|\h^\natural}|\widehat\L_{|\h^\natural})-2(\widehat\L|2\rho_R-\rho-\half\epsilon(\si_R)\theta/2)
\right)
  -\widehat\L(x+D) \notag
\\
&\quad+ \frac{k}{8(k+h^\vee)}(\dim\g_{1/2}-2\epsilon(\si_R)) -\frac{1}{16}\dim\g_{1/2}\label{simpleells}.
\end{align}
and $\nu=\widehat \L_{|\h^\natural}-\rho_R$.
\end{cor}
\begin{proof} Using the calculations in Lemma \ref{gammas}, specialize  to $\s=\s_R$ the results of Proposition \ref{Hhw}, which summarizes Proposition 3.1, Theorem 3.1,  and Proposition 4.1 of \cite{KW}.
\end{proof}

  By \cite[Corollary 3.2]{KW} and  \cite[(4.12)]{KW}, the Weyl vector  corresponding to $\Pia$ is
\begin{equation}\label{rt}
 \widehat\rho^{\,\tw}=-\gamma'+h^\vee \L_0=-2\rho_R+\rho+\half\e(\s_R)\theta/2+h^\vee \L_0.
\end{equation}
 It follows that \eqref{simpleells} can be rewritten as
 \begin{equation}\label{ellnussimple}
\ell(\widehat\L )=\frac{\Vert \widehat\L+\widehat\rho^{\tw})\Vert^2-\Vert\widehat\rho^{\tw}\Vert^2}{2(k+h^\vee)}-\widehat\L(x+D)+a(k)=\frac{(\widehat\L|\widehat\L+2\widehat\rho^{\tw})}{2(k+h^\vee)}-\widehat\L(x+D)+a(k),
 \end{equation}
where
\begin{equation}\label{aofk}
a(k)=\frac{k}{8(k+h^\vee)}(\dim\g_{1/2}-2\epsilon(\si_R)) -\frac{1}{16}\dim\g_{1/2}.
\end{equation}

Next, we introduce the framework for computation of characters of highest weight modules, following \cite{GK2}.
Let $\mathcal B$ be a basis of $\C\oplus(\h^\natural)^*$. In the set of formal series $\sum_{(z,\mu)\in\C\oplus (\h^\natural)^*}b_{(z,\mu)} q^ze^\mu$ with $b_{(z,\mu)} \in\mathbb Q$, we consider the algebra $\mathcal R(\mathcal B)$ of finite linear combinations of series of the form  $\sum_{(n,\mu)\in\Z_+\mathcal B}b_{(n,\mu)} q^{z-n} e^{\l-\mu}$. Given a series $Y=\sum_{(z,\mu)\in\C\oplus (\h^\natural)^*}b_{(z,\mu)} q^ze^\mu$, we define the support of $Y$ as the set $\text{Supp}\,Y=\{(z,\mu)\mid b_{(z,\mu)}\ne0\}$.

   Let $ev:\ha^*\to\C\oplus (\h^\natural)^*$ be the map defined by \begin{equation}\label{evaluation}
   ev(\widehat \L)=-\widehat\L(x+D)+\widehat\L_{|\h^\natural}.
   \end{equation}
   
    Set $\Pi^R=ev(\Pia)$. By the explicit description of $\Pia$ given above, it is easy to check that $\Pi^R\setminus\{0\}$ is a basis of $\C\oplus (\h^\natural)^*$.
The choice of an ordering of $\Pi^R\setminus\{0\}$ defines a  lexicographic total order $\le_{\Pi^R}$ on $\C\oplus (\h^\natural)^*$. We use this total order to  topologize $\mathcal R(\Pi^R)$ by choosing as a fundamental set of open neighborhoods of $0$ the sets
$$
V_{(z_0,\l)}=\{Y\in \mathcal R(\Pi^R)\mid z+\mu\le_{\Pi^R}z_0+\l\text{ for all } z+\mu\in \text{Supp}\,Y\}.
$$

If $\beta\in \C\oplus(\h^\natural)^*$, $\beta=z+\l$, we write $e^\beta$ for $q^ze^\l$, so that
\begin{equation} e^{ev(\widehat\L)}=q^{-\widehat \L(x+D)}e^{\widehat\L_{|\h^\natural}}=q^{-m-\l(x)}e^{\l_{|\h^\natural}}\quad\text{if $\widehat\L=k\L_0+m\d+\l,\,\l\in\h^*$}.
\end{equation}
If $\beta\in\Z_+\Pi^R$ and $a\in\mathbb Q\setminus\{0\}$, we note that $(1-ae^{\pm\beta})$ are both invertible in $\mathcal R(\Pi^R)$ with inverses
\begin{align}\label{geometric1}
(1-ae^{-\beta})^{-1}=\sum_{n=0}^\infty a^ne^{-n\beta},
\\
(1-ae^{\beta})^{-1}=-\frac{e^{-\beta}}{a}\sum_{n=0}^\infty a^{-n}e^{-n\beta}.\label{geometric2}
\end{align}
We would like to extend the map $ev$ to a map from  $\mathcal R(\Pia)$ to $\mathcal R(\Pi^R)$ by mapping $e^{\widehat\L}$ to $e^{ev(\widehat\L)}$, but this is not possible: for example $ev(\sum_{n=0}^\infty e^{-n(\d-\theta)})$ does not make sense.
To make sense of $ev$, we must restrict to the following set
$$
\mathcal R(\Pia)_{fin}=\{Y\in \mathcal R(\Pia)\mid \text{ for all }\mu\in \text{Supp}\,Y,\ (\mu+\mathbb Q(\d-\theta))\cap  \text{Supp}\,Y\text{ is finite}\}.
$$

We will need the following special case of the Lemma in \S\ 2.2.8 of \cite{GK2}:

\begin{lemma} \label{conergence}Let $\L\in\h^*$ and set $\widehat \L=(k+h^\vee)\L_0+\L$. Assume  that $2\frac{(\widehat \L|\a)}{(\a|\a)}\in\ZZ$ for all real roots in $\Da^\natural$. Let  $J$ be a set of linearly independent isotropic odd roots in $\Da^{\tw}_+$. 
Then the series
$$
Y_1=\sum_{w\in\Wa^\natural}det(w)e^{w({\widehat\L})},\ Y_2=\sum_{w\in\Wa^\natural}det(w)\frac{e^{{w({\widehat\L})}}}{\prod_{\be\in J}(1+e^{-w(\beta)})}
$$
are elements of $\mathcal R(\Pia)_{fin}$, hence the series
$$
ev(Y_1)=\sum_{w\in\Wa^\natural}det(w)e^{ev(w({\widehat\L}))},\ ev(Y_2)=\sum_{w\in\Wa^\natural}det(w)\frac{e^{ev(w({\widehat\L}))}}{\prod_{\be\in J}(1+e^{-ev(w(\beta))})}
$$
converge and define  elements of $\mathcal R(\Pi^R)$.
\end{lemma}
\begin{proof}We note that, since $k+h^\vee< 0$, $\widehat\L$ satisfies the hypothesis of \S\ 2.2.8 of \cite{GK2} with $W'=\What^\natural$, hence $
Y_1$ and $Y_2$ 
converge and are elements of $\mathcal R(\Pia)$. 

We need only to show that, if $\mu\in\text{Supp}\,Y_i$, then the set $\{\mu'\in\text{Supp}\,Y_i\mid \mu'-\mu\in\mathbb Q(\d-\theta)\}$ is finite. Let $\l=\bar w(\widehat\L)$ be the unique maximal element in the $\Wa^\natural$ orbit of  $\widehat\L$. Since $Y_i$ are skew-invariant for the action of $\Wa^\natural$ and $\d-\theta$ is fixed by $\Wa^\natural$, we can assume that $\mu$ is in the support of $\frac{e^{\l}}{\prod_{\be\in J}(1+e^{-\bar w(\beta)})}$ so
$\mu=\l-\sum_ir_i\gamma_i$ with $r_i\in\Z_+$ and $\gamma_i\in\pm\bar w(J)$, $\gamma_i$  positive. 

Let $\mu'\in \text{Supp}\,Y_i$ be such that $\mu-\mu'\in  \mathbb Q(\d-\theta)$. Then $\mu'=w(\l)-\sum_ir'_iw(\gamma_i)$ with $r_i'\in\Z$ and
\begin{equation}\label{dminusthetaw}
a(\d-\theta)=\l-w(\l)-\sum_ir_i\gamma_i+\sum_ir'_iw(\gamma_i).
\end{equation}
It follows that
\begin{equation}\label{sumwlwgamma}
\sum_ir_i\gamma_i(x+D)=(\l-w(\l)+\sum_ir'_iw(\gamma_i))(x+D).
\end{equation}
Since $\l-w(\l)$ is a sum of positve roots in $\Da^\natural$, we see that 
$$
(\l-w(\l))(x+D)=(\l-w(\l))(D)\in\Z_+.
$$
 If $w(\gamma_i)$ is a positive root, then $w(\gamma_i)(x+D)\ge0$ and $r'_i\ge0$. If $w(\gamma_i)$ is a negative root, then $w(\gamma)(x+D)\le0$ and $r'_i<0$. The outcome is that
 $\sum_ir'_iw(\gamma_i)(x+D)\in \half\Z_+$. Thus there are only finitely many pairs $(m,n)\in\half\Z_+\times \half\Z_+$ such that $m=(\l-w(\l))(x+D)$, $n=\sum_ir'_iw(\gamma_i)(x+D)$ and $m+n=\sum_ir_i\gamma_i(x+D)$. 
 
 By the combinatorics of reflection groups, for any given $m$, there are only finitely many $w\in\Wa^\natural$ such that $m=(\l-w(\l))(D)$. It follows that there is a finite subset $X$ of $\Wa^\natural$ such that \eqref{sumwlwgamma} holds iff $w\in X$.
Hence  \eqref{dminusthetaw} can be satisfied only if $w\in X$ and 
\begin{equation}\label{eqrprime}
0=(\l-w(\l)-\sum_ir_i\gamma_i)_{|\h^\natural}+\sum_ir'_iw(\gamma_i)_{|\h^\natural}.
\end{equation}
  Since $\{w(\gamma_i)\}$ is linearly independent,  \eqref{dminusthetaw} has only finitely many solutions.
\end{proof}

Set
\begin{align}\label{RW}
\widehat F^{NS}&=
\prod^{\infty}_{n=1} \frac{(1-q^n)^{\dim \fh}\prod_{\alpha \in \Delta^\natural_{+}}
        (1-q^{n-1} e^{-\alpha})
        (1-q^n e^{\alpha})}{\prod_{\alpha \in \Delta_{-1/2}}
  (1+q^{n-\frac{1}{2}}e^{\alpha_{|\h^\natural}})},\\
  \widehat F^{R}&=\prod^{\infty}_{n=1} \frac{ (1-q^n)^{\dim \fh}\prod_{\alpha \in \Delta^\natural_{+}}
        (1-q^{n-1} e^{-\alpha})
        (1-q^n e^{\alpha})}{\prod_{\eta\in(\overline{\D}^+_{1/2})'}(1+q^{n-1}e^{-\eta})\prod_{\eta\in\overline{\D}^+_{1/2}}(1+q^ne^{\eta})},\label{DenominatorRamond}
        \end{align}
where
$(\overline{\D}^+_{1/2})'=\overline{\D}^+_{1/2}$ if $\theta/2$ is not a root and $(\overline{\D}^+_{1/2})'=\overline{\D}^+_{1/2}\cup\{0\}$ otherwise. Let $\Pia^{NS}=\Pi\cup\{\d-\theta\}$.
 If $\Pi^{NS}=ev(\Pia^{NS})=ev(\Pi)\cup\{0\}$, these infinite products can naturally be seen as invertible elements of   $\mathcal R(\Pi^{NS})$, $\mathcal R(\Pi^{R})$ respectively.
  
  The character $ch\,M$ of an non-twisted  highest weight $\Wu$-module $M$ is  defined as the trace of $q^{L_0}J^{\{h\}}_0$, $h\in\h^\natural$,  and can naturally be seen as an element of $\mathcal R(\Pi^{NS})$. Similarly,
  the character $ch\,M$ of a Ramond twisted   highest weight $\Wu$-module $M$ is  defined as the trace of $q^{L^{\tw}_0}J^{\{h\},\tw}_0$, $h\in\h^\natural$, and can naturally be seen as an element of $\mathcal R(\Pi^{R})$. Note that $\widehat F^{NS}$ and $\widehat F^R$ are the denominators of the characters of non-twisted (resp. $\s_R$-twisted) Verma modules over $\Wu$. This is easily seen by computing the character of $M^W(0,0)$ using the basis given in  \cite[(6.13)]{KW1} for the NS sector and in \eqref{basisW} in the Ramond sector. The formulas in \eqref{RW} and in \eqref{DenominatorRamond} give precisely these characters when expanded according to \eqref{geometric1} and \eqref{geometric2}.

    Let $\widehat R$ (resp. $\widehat R^{\tw}$) be the Weyl denominator for $\ga$ (resp. $\ga^{\,\tw}$). The following result  is a special case  of Theorem 3.1 of \cite{KRW} and formula (3.14)  of \cite{KW}:
\begin{theorem}\label{822} (a)  If $M$ is a highest weight $\ga$-module of highest weight $\widehat\L$ and $\widehat R\,ch(M)\in\mathcal R(\Pia^{NS})_{fin}$ then
$$
\sum_j(-1)^j\widehat F^{NS}ch\,H_j(M)=q^{\tfrac{(\widehat\L|\widehat\L+2\rhat)}{2(k+h^\vee)}}ev( \widehat R\,ch(M)).
$$

\noindent (b) If $M$ is a highest weight $\ga^{\,\tw}$-module of highest weight $\widehat\L$ and $\widehat R^{\tw}ch(M)\in\mathcal R(\Pia)_{fin}$ then
$$
\sum_j(-1)^j\widehat F^{R}ch\,H_j(M)=e^{-\rho_R}q^{\tfrac{(\widehat\L|\widehat\L+2\rhat^{\,\tw})}{2(k+h^\vee)}+a(k)}ev(\widehat R^{\,\tw}ch(M)),
$$
where $a(k)$ is given by \eqref{aofk}.
\end{theorem}

A weight $\widehat\L$ of $\ga$ is called {\sl degenerate} if $\widehat \L(\alpha_0^\vee)\in\ZZ_+$. It is shown in \cite{KW} that $H(L(\widehat\L))=0$ if and only if $\widehat \L$ is degenerate.
Recall  that a $V^k(\g)$-module $M$ is called integrable if it is integrable with respect to $\ga_i^\natural$ for all $i=1,\ldots,s$ (see \eqref{gnatural}).

An easy consequence of Theorem \ref{822} is the following Proposition.
\begin{proposition}\label{typical}
Assume that 
\begin{equation}\label{pp} (\widehat \L+\widehat\rho^{\tw}|\a)\ne  n \tfrac{(\a|\a)}{2}\text{ for all $n\in\mathbb N$ and $\a\in \Da^{\tw}_+\setminus (\Dap)^{\natural}$}\end{equation}
and that $L(\widehat\L)$ is integrable.
Then $H(L(\widehat \L))\ne 0$ and
\begin{align*}
&\widehat F^{R}\sum_j(-1)^jch\,H_j(L(\widehat\L))
&\\
&=q^{\tfrac{(\widehat\L|\widehat\L+2\rhat^{\,\tw})}{2(k+h^\vee)}+a(k)}e^{-\rho_R}\sum_{w\in\Wa^\natural}det(w)q^{-(w(\widehat\L+\rhat^{\,\tw})-\rhat^{\,\tw})(x+D)}e^{w(\widehat\L+\rhat^{\,\tw})-\rhat^{\,\tw})_{|\h^\natural}}.
\end{align*}
\end{proposition}
\begin{proof}
The hypothesis implies in particular that $\widehat\L$ is nondegenerate, hence $H(L(\widehat\L))\ne 0$. Arguing as in  
Proposition 11.5  of \cite{KMP1}, we see that
$$
R^{\tw}ch\,L(\widehat\L)=\sum_{w\in \Wa^\natural}det(w)e^{w(\L+\rhat^{\,\tw})-\rhat^{\,\tw}}.
$$
By Lemma \ref{conergence}, $\widehat R^{\tw} ch\,L(\widehat\L)\in\mathcal R(\Pia)_{fin}$ so, applying Theorem \ref{822}, we conclude.
\end{proof}

A similar result is the following.
\begin{proposition}\label{atypicalnonzero}
Assume that $(\widehat\L+\rhat^{\,\tw}|\d-\theta)\notin\nat$ and that $L(\widehat\L)$ is integrable and maximally atypical. Assume that $\Pia$ contains a maximal set $J$  of pairwise orthogonal isotropic roots such that $(\widehat\L+\rhat^{\,\tw}|\beta)=0$ for all $\be\in J$.
Then $H(L(\widehat \L))\ne 0$ and
\begin{align*}
&\widehat F^{R}\sum_j(-1)^jch\,H_j(L(\widehat\L))\\
&=q^{\tfrac{(\widehat\L|\widehat\L+2\rhat^{\,\tw})}{2(k+h^\vee)}+a(k)}e^{-\rho_R}\sum_{w\in\Wa^\natural}det(w)\frac{q^{-(w(\widehat\L+\rhat^{\,\tw})-\rhat^{\,\tw})(x+D)}e^{w(\widehat\L+\rhat^{\,\tw})-\rhat^{\,\tw})_{|\h^\natural}}}{\prod_{\be\in J}(1+q^{w(\be)(x+D)}e^{-w(\be)_{|\h^\natural}})}.
\end{align*}
\end{proposition}
\begin{proof}
The hypothesis implies in particular that $\widehat\L$ is nondegenerate, hence $H(L(\widehat\L))\ne 0$. 

It is shown in \cite{GK2} that the hypothesis imply that 
\begin{equation}\label{KWC}
ch\,L(\widehat\L)=\frac{1}{\widehat R^{\tw}}\sum_{w\in \Wa^\natural}det(w)\frac{e^{w(\L+\rhat^{\,\tw})-\rhat^{\,\tw}}}{\prod_{\be\in J}(1+e^{-w(\be)})}.
\end{equation}
Formula \eqref{KWC} is a special case of \cite[Formula (14)]{GK2} if $\g\ne D(2,1;\frac{m}{n})$
and of \cite[Section 6.1]{GK2} if $\g= D(2,1;\frac{m}{n})$ and $\widehat\L_{|\h^\natural}=0$ (for other $\widehat \L$ \eqref{KWC} holds only conjecturally).
Here we use that $\ga^{\tw}\cong\ga$
\cite[Remark 8.5]{VB}, since $\si_R$ is an inner automorphism of $\g$ in all cases considered. 
By Lemma \ref{conergence}, $\widehat R^{\tw} ch\,L(\widehat\L)\in\mathcal R(\Pia)_{fin}$ so, applying Theorem \ref{822}, we conclude.
\end{proof}

\begin{remark}\label{remarkqexpansion}
The equalities in Propositions \ref{typical} and \ref{atypicalnonzero} are to be understood as equalities in $\mathcal R(\Pi^{NS})$ or $\mathcal R(\Pi^{R})$. 
To be more explicit, observe that
 $\Pi^{NS}=\{\gamma_0, \gamma_1,\ldots,\gamma_r,0\}$ with $\{\gamma_i\}$ a basis of  $\C\oplus(\h^\natural)^*$, $\gamma_0=-\half +\overline\gamma_0$,  $\overline\gamma_0\in(\h^\natural)^*$ and $\gamma_i\in(\h^\natural)^*$ for $i>0$. Then a series $$
 Y=\sum_{(n,\mu)\in\Z_+\Pi^{NS}}b_{(n,\mu)} q^{z-n} e^{\l-\mu}\in\mathcal R(\Pi^{NS})
 $$
  can be rewritten as
\begin{equation}\label{qexpansion}
 Y=\sum_{n\in\Z_+}a_n q^{z+\tfrac{n}{2}}
\end{equation}
 with
 $$
a_n= \left(\sum_{n_1,\ldots,n_r}b_{(n,\sum_in_i\gamma_i)}e^{\l-n\overline{\gamma}_0-\sum_in_i\gamma_i}\right)\in \mathcal R(\{\gamma_1,\ldots,\gamma_r\}).
$$
We refer to the expression in \eqref{qexpansion} as the $|q|<1$ expansion of $Y$.

The same argument works verbatim for $\mathcal R(\Pi^R)$ except that $\gamma_0=-1+\overline\gamma_0$ so the $|q|<1$ expansion reads
$$
Y=\sum_n a_nq^{z+n}.
$$
with $a_n\in\mathcal R(\{\gamma_1,\ldots,\gamma_r\})$.
 \end{remark}

\section{Unitarity between $A(k,\nu)$ and $B(k,\nu,\rho_R)$}\label{UbetweenAB}
 For   $\nu\in P^+_k$ and $s\in\C$, set 
\begin{equation}\label{nuhats}
\widehat \nu_{s}=k\L_0+s\theta+\nu+\rho_R,
\end{equation}
and set $\ell(s)=\ell(\widehat \nu_s)$. 
An obvious calculation shows that
\begin{align}\label{ls}
  \ell(s)&= \frac{(\nu-\rho_R |\nu-\rho_R+2\rho^\natural)}{2(k+h^\vee)}+\frac{s(s-k-1+\half\epsilon(\si_R))}{k+h^\vee}+ \frac{2}{k+h^\vee}(\rho_R|\rho^\natural-\rho_R)
\\&+ \frac{k}{8(k+h^\vee)}(\dim\g_{1/2}-2\epsilon(\si_R)) -\frac{1}{16}\dim\g_{1/2}.\notag
\end{align}

Let us call
a weight $\nu\in P^+_k$ {\sl Ramond extremal} if
\begin{equation}\label{Rextremal}
\nu-\rho_R\notin P^+_k\text{ or }\nu-\rho_R\text{ is extremal}.
\end{equation}

We restrict our attention to the irreducible highest weight modules $L^W(\nu,\ell)$ that satisfy the necessary conditions for unitarity proven in Section \ref{necessary}, thus $\nu\in P^+_k$, $\ell$ is real with $\ell\ge A(k,\nu)$, and, if $\nu$ is Ramond extremal, $\ell=A(k,\nu)$.

We want to calculate  the difference 
$$d(s)=\ell(s)-B(k,\nu,\rho_R).$$
We use the following fact, which is verified by case-wise inspection. 
\begin{lemma}
\begin{equation}\label{ps}(\rho_R|\rho^\natural-\rho_R)= \frac{h^\vee-\half\e(\si_R)}{16}(\dim\g_{1/2}-\e(\si_R)).\end{equation}
\end{lemma}

\begin{lemma}\label{differ}
\begin{equation}\label{diffe} d(s)=\frac{(s-(\tfrac{k+1}{2}-\tfrac{\epsilon(\si_R)}{4}))^2}{k+h^\vee}.
\end{equation}
\end{lemma}
\begin{proof}By \eqref{BKnurhoR} and \eqref{ls},
\begin{align*}
\ell(s)-B(k,\nu,\rho_R)&=\frac{s(s-k-1+\half\epsilon(\si_R))}{k+h^\vee}+ \frac{2}{k+h^\vee}(\rho_R|\rho^\natural-\rho_R)
\\&+ \frac{k}{8(k+h^\vee)}(\dim\g_{1/2}-2\epsilon(\si_R))
+\frac{(k+1)^2}{4(k+h^\vee)} -\frac{1}{8}\dim\g_{1/2}.
\end{align*}
Using \eqref{ps} and relation $\dim\g_{1/2}=-2(h^\vee-2)$, we find
\begin{align*}
\ell(s)&-B(k,\nu,\rho_R)=\frac{s(s-k-1+\half\epsilon(\si_R))}{k+h^\vee}+\frac{h^\vee-\half\e(\si_R)}{8(k+h^\vee)}(\dim\g_{1/2}-\e(\si_R))
\\&+ \frac{k}{8(k+h^\vee)}(\dim\g_{1/2}-2\epsilon(\si_R))
+\frac{(k+1)^2}{4(k+h^\vee)} -\frac{1}{8}\dim\g_{1/2}\\
&=\frac{s(s-k-1+\half\epsilon(\si_R))}{k+h^\vee}
+\frac{\e(\si_R)^2}{16(k+h^\vee)}- \frac{k+1}{4(k+h^\vee)}\epsilon(\si_R)
+\frac{(k+1)^2}{4(k+h^\vee)} \\
&=\frac{s(s-k-1+\half\epsilon(\si_R))}{k+h^\vee}
+\frac{(-k-1+\half\epsilon(\si_R))^2}{4(k+h^\vee)} \\
&=\frac{(s-(\tfrac{k+1}{2}-\tfrac{\epsilon(\si_R)}{4}))^2}{k+h^\vee}.
\end{align*}
\end{proof}

We now compute the values of $s\in\C$ such that $\ell(s)=A(k,\nu)$.
For this we need the following computation.
\begin{lemma} If $\theta/2$ is not a root of $\g$, then 
\begin{equation}\label{Fetamin}
F_\nu(\eta_{\min})=2(\nu|\eta_{\min})\left((\nu|\eta_{\min})+2(\rho^\natural-\rho_R|\eta_{\min})\right)-(\nu|\rho_R).
\end{equation}
\end{lemma}
\begin{proof} We proceed by a case-wise inspection.

$\g=psl(2|2)$.
We have $\eta_{\min}=\rho_R=\rho^\natural=\half(\d_1-\d_2)$. If $\nu=\tfrac{r}{2}(\d_1-\d_2)$, then $(\nu|\eta_{\min})=(\nu|\rho_R)=-\tfrac{r}{2}$, $(\rho^\natural-\rho_R|\eta_{\min})=0$ so
$$2(\nu|\eta_{\min})^2+4(\nu|\eta_{\min})(\rho^\natural-\rho_R|\eta_{\min})-2(\nu|\rho_R)=\half r^2+r,
$$
which is \eqref{Fpsl}.

$\g=spo(2|2r)$.
We have $\eta_{\min}=\pm \e_r,\,\rho_R=\half \sum_{i<r} \e_i\pm\half\e_r$, and in turn  $(\nu|\eta_{\min})=\mp1/2 m_r\quad  (\nu|\rho_R)=-1/4\sum_{i<r} m_i\mp1/4 m_r$, $(\rho_R|\eta_{\min})= -1/4$, $(\rho^\natural-\rho_R|\eta_{\min})=1/4$, so 
$$
2(\nu|\eta_{\min})^2+4(\nu|\eta_{\min})(\rho^\natural-\rho_R|\eta_{\min})=\half m_r^2\mp \half m_r
 +\half\sum_{i<r} m_i\pm \half m_r=
\half m_r^2+\half\sum_{i<r} m_i,$$
which is the expression for  $F_\nu(\eta_{\min})$ given in \eqref{Fpari}.

$\g=F(4)$.
We have $\eta_{\min}=\half(-\e_1+\e_2+\e_3)$ or $\eta_{\min}=\half(\e_1-\e_2-\e_3)$ and $\rho_R=\omega_3$ or $\rho_R=\omega_1$ respectively. In the first case   $(\nu|\eta_{\min})=1/3(m_1-m_2-m_3)$, $(\nu|\rho_R)=-1/3(m_1+m_2+m_3)$, $(\rho_R|\eta_{\min})= -1/6$, 
$(\rho^\natural-\rho_R|\eta_{\min})=1/3$, so 
\begin{align*}
2(\nu|\eta_{\min})^2&+4(\nu|\eta_{\min})(\rho^\natural-\rho_R|\eta_{\min})-2(\nu|\rho_R)\\&=\tfrac{2}{9}(m_1-m_2-m_3)^2+\tfrac{4}{9}(m_1-m_2-m_3)+\tfrac{2}{3}(m_1+m_2+m_3),
\end{align*}
 which is \eqref{F}. The same holds in the second case: $(\nu|\eta_{\min})=1/3(-m_1+m_2+m_3),  (\nu|\rho_R)=-2/3m_1$, $(\rho_R|\eta_{\min})= -1/3$, $(\rho^\natural-\rho_R|\eta_{\min})=1/6$ so 
 \begin{align*}
&2(\nu|\eta_{\min})^2+4(\nu|\eta_{\min})(\rho^\natural-\rho_R|\eta_{\min})-2(\nu|\rho_R)=\\&\tfrac{4}{9}(m_1-m_2-m_3)^2-\tfrac{2}{9}(m_1-m_2-m_3)+\tfrac{4}{3}m_1=\tfrac{2}{9}((-m_1+m_2+m_3)^2+5m_1+m_2 +m_3).\end{align*}

$\g=D(2,1;a).$
We have $\eta_{\min}=\e_2-\e_3, -\e_2+\e_3),\,\rho_R=\omega^1_1,\omega^2_1$. In the first case   $(\nu|\eta_{\min})=\tfrac{-m_1+a m_2}{2(1+a)},$ $(\nu|\rho_R)=\tfrac{-m_1}{2(1+a)}$, $(\rho_R|\eta_{\min})= \tfrac{-1}{2(1+a)}$, so 
\begin{align*}
&2(\nu|\eta_{\min})^2+4(\nu|\eta_{\min})(\rho^\natural-\rho_R|\eta_{\min})-2(\nu|\rho_R)=\\&
2\left(\frac{-m_1+a m_2}{2(1+a)}\right)^2+\frac{a(-m_1+a m_2)}{(1+a)^2}-\frac{-m_1}{(1+a)},
\end{align*}
which is \eqref{FD21}. The other case is similar.
\end{proof}

\begin{lemma}\label{Aknusecond}\ 
\begin{enumerate}
\item If $\theta/2$ is a root of $\g$, then 
 $$
A(k,\nu)=B(k,\nu,\rho_R),
$$
so $\ell(s)=A(k,\nu)$ if and only if $s=\frac{2k+1}{4}$.
\item If $\theta/2$ is not a root of $\g$, then 
\begin{equation}\label{AB}
A(k,\nu)=
B(k,\nu,\rho_R)+\frac{(\nu-\rho_{R}+\rho^\natural|\eta_{\min})^2}{k+h^\vee},
\end{equation}
so $\ell(s)=A(k,\nu)$  if and only if 
\begin{equation}\label{spartic}s=\frac{k+1}{2}\pm (\nu-\rho_{R}+\rho^\natural|\eta_{\min}).\end{equation}
\end{enumerate}
\end{lemma}
\begin{proof} If $\theta/2$ is a root of $\g$, we have
$$
A(k,\nu)=\tfrac{1}{2(k+h^\vee)}((\nu|\nu+2(\rho^\natural-\rho_R))    -\tfrac{1}{2} p(k)),
$$
while 
$$
B(k,\nu,\rho_R)=-\frac{(k+1)^2}{4(k+h^\vee)}+\frac{(\nu-\rho_R|\nu-\rho_R+2\rho^\natural)}{2(k+h^\vee)}+\tfrac{1}{16}\dim\g_{1/2}
$$
$$
=-\frac{(k+1)^2}{4(k+h^\vee)}+\frac{(\nu|\nu+2\rho^\natural)}{2(k+h^\vee)}-\frac{(\rho_R|2\nu-\rho_R+2\rho^\natural)}{2(k+h^\vee)}+\tfrac{1}{16}\dim\g_{1/2}
$$
We need to check that 
$$
\frac{p(k)}{4(k+h^\vee)}
=\frac{(k+1)^2}{4(k+h^\vee)}+\frac{(\rho_R|-\rho_R+2\rho^\natural)}{2(k+h^\vee)}-\tfrac{1}{16}\dim\g_{1/2}.
$$
This is easily checked case by case:
\begin{itemize}
\item $spo(2|2r+1)$, $r\ge1$:
$$k^2+(7/4-r/2) k+5/8-r/4=(k+1)^2-r^2/2+r/4-(r/2+1/4)(k+3/2-r).
$$ 
\item $G(3)$: 
$$k^2+1/4k-3/8=(k+1)^2-4-7/4(k-3/2).
$$ 
\end{itemize}

If $\theta/2$ is not a root of $\g$ then
$$A(k,\nu)=\tfrac{1}{2(k+h^\vee)}\left((\nu|\nu+2\rho^\natural)    -\tfrac{1}{2} p(k)+F_{\nu}(\eta_{\min})\right),
$$
so the difference $A(k,\nu)-B(k,\nu,\rho_R)$ is 
\begin{equation}\label{diff}\tfrac{1}{2(k+h^\vee)}\left(   -\tfrac{1}{2} p(k)+F_{\nu}(\eta_{\min})\right)+\frac{(k+1)^2}{4(k+h^\vee)}+\frac{(\rho_R|2\nu-\rho_R+2\rho^\natural)}{2(k+h^\vee)}-\tfrac{1}{16}\dim\g_{1/2}.
\end{equation}
Substituting \eqref{Fetamin} in \eqref{diff} we get
\begin{align*}
&\frac{1}{2(k+h^\vee)}\left(   -\tfrac{1}{2} p(k)+2(\nu|\eta_{\min})^2+4(\nu|\eta_{\min})(\rho_R|\eta_{\min})\right)+\frac{(k+1)^2}{4(k+h^\vee)}\\
&+\frac{(\rho_R|-\rho_R+2\rho^\natural)}{2(k+h^\vee)}-\tfrac{1}{16}\dim\g_{1/2},
\end{align*}
so we have our claim, provided that
$$
\frac{p(k)}{4(k+h^\vee)}
=\frac{(k+1)^2}{4(k+h^\vee)}+\frac{(\rho_R|-\rho_R+2\rho^\natural)}{2(k+h^\vee)}-\tfrac{1}{16}\dim\g_{1/2}-\frac{(\rho^\natural-\rho_{R}|\eta_{\min})^2}{k+h^\vee}.
$$ 
This formula is proved by inspection.
\end{proof}

An immediate consequence of the computation above is the following result.
\begin{cor}\label{thetahalfisaroot}If $\theta/2$ is a root of $\g$ and $\nu\in P^+_k$ is not Ramond extremal, then $L^W(\nu,\ell)$ is unitary if and only if 
$\ell\ge A(k,\nu)$.

If $\theta/2$ is not a root of $\g$, $\nu\in P^+_k$ is not Ramond extremal, and $(\nu-\rho_{R}+\rho^\natural|\eta_{\min}) =0$, then $L^W(\nu,\ell)$ is unitary if and only if 
$\ell\ge A(k,\nu)$.
\end{cor}
\begin{proof}
In these cases $A(k,\nu)=B(k,\nu,\rho_R)$.
\end{proof}

It remains to discuss the cases when $A(k,\nu)<B(k,\nu,\rho_R)$.
For handling these cases we want to compute  characters of $H(L(\widehat\nu_s))$. A first step in this direction is given by Propositions \ref{typical} and \ref{atypicalnonzero}, therefore we check if $\widehat\nu_s$ satisfies their hypothesis.


\begin{lemma}\label{hc} Assume that $\theta/2$ is not a root of $\g$. Assume $\nu\in P^+_k$ is not Ramond extremal. If
\begin{equation}\label{hyp}
0\le s-\tfrac{k+1}{2} < |(\nu-\rho_{R}+\rho^\natural|\eta_{\min})|,\end{equation} 
then \eqref{pp} holds.
\end{lemma}
\begin{proof}
 By \cite[Corollary 3.2]{KW} and  \cite[(4.12)]{KW}, $\widehat\rho^{\tw}=-\gamma'+h^\vee \L_0$, so, by 
Lemma \ref{gammas}, for $\widehat\nu_s$, given by \eqref{nuhats}, we have
\begin{equation}\label{nuhats+}\widehat \nu_s+\widehat\rho^{\tw}= (k+h^\vee)\L_0+s\theta+\nu-\rho_R+\rho+\tfrac{1}{2}\epsilon(\si_R)\theta/2.\end{equation}
Note that  the elements of  $\Da^{\tw}_+\setminus (\Dap)^{\natural}$ are precisely
\begin{enumerate}
\item $\beta-j\d,\,\g_\beta\subset \n_j(\s_R)_+, j\ne 0,$
\item $\beta+(m-j)\d,\,\g_\beta\subset \g_j,\, m\in \mathbb N, j\ne 0.$
\end{enumerate}

We check that $(\widehat \nu_s+\widehat \rho^{\tw}|\a)\neq 0$ for all  isotropic roots. If $\a$ is such a root, then
\begin{align}\label{r1}
 \a&=\pm(-\half\d+\theta/2)+\eta,\ \eta\in\overline{\D}^+_{1/2},
 \\ \a&=\pm(-\half\d+\theta/2)\pm\eta+n\d,\ \eta\in\overline{\D}^+_{1/2}, n\in\nat\label{root2}.
\end{align}
We start considering  the roots \eqref{r1}. We have two cases.

First case: $(\nu-\rho_R+\rho^\natural|\eta_{\min})<0$. We compute, using \eqref{nuhats+},
\begin{align*}&((k+h^\vee)\L_0+s\theta+\nu-\rho_R+\rho|\pm(-\half\d+\theta/2)+\eta)=\\&=\mp\half(k+h^\vee)+(\nu-\rho_R+\rho^\natural|\eta)\pm\tfrac{h^\vee-1}{2}\pm s 
\\&=\pm (s-\tfrac{ k+1}{2})+(\nu-\rho_R+\rho^\natural|\eta)
\\&=\pm (s-\tfrac{ k+1}{2})+(\nu-\rho_R+\rho^\natural|\eta_{\min}+\sum_{\beta\in \D_+^\natural,\,n_\beta\geq 0}n_\beta\beta)
\\&\le \pm (s-\tfrac{ k+1}{2})+(\nu-\rho_R+\rho^\natural|\eta_{\min})<0.
\end{align*}
The last inequality follows from \eqref{hyp}.

Second case: $(\nu-\rho_R+\rho^\natural|\eta_{\min})>0$. Computing as above
\begin{align*}&((k+h^\vee)\L_0+s\theta+\nu-\rho_R+\rho|\pm(-\half\d+\theta/2)+\eta)=
\\&=\pm(s-\tfrac{ k+1}{2})-(\nu-\rho_R+\rho^\natural|-\eta)
\\&=\pm(s-\tfrac{ k+1}{2})-(\nu-\rho_R+\rho^\natural|\eta_{\min}-\sum_{\beta\in \D_+^\natural,\,n_\beta \geq 0}n_\beta\beta)
\\&\leq\pm(s-\tfrac{ k+1}{2})-(\nu-\rho_R+\rho^\natural|\eta_{\min})<0.\end{align*}
We now deal with  the roots \eqref{root2}. The case when $\a=\pm(-\half\d+\theta/2)+\eta+n\d$ is handled as above, since $k+h^\vee<0$.
In the remaining case we use the following relations (see \cite[(11.23)]{KMP1} for \eqref{rr1} and  \cite[(11.18), (11.19)]{KMP1} for \eqref{rr3}):
\begin{align}
\label{rr1}
&k+\tfrac{h^\vee}{2}\leq -\tfrac{1}{2} +(\xi|\nu-\rho_R )-(\nu-\rho_R + \rho^{\natural}|\eta),\ \eta \in \pm \D^+_{1/2},\ \eta\ne-\xi,\\
\label{rr3}
&
(\rho^\natural|\xi)=\half(h^\vee-1).
\end{align}
Formula \eqref{rr1} follows from the non-extremality of $\nu-\rho_R$. In the subsequent computation we use \eqref{rr1} with $\eta=\pm\eta_{\min}$. When $\g=spo(2|3)$ or $psl(2|2)$ (and only in these cases) it happens that $\eta_{min}=\pm\xi$. A direct check shows that \eqref{rr1} still holds. We have
\begin{align*}&((k+h^\vee)\L_0+s\theta+\nu-\rho_R+\rho|\pm(-\half\d+\theta/2)-\eta+n\d)=\\&=\mp\half(k+h^\vee)-(\nu-\rho_R+\rho^\natural|\eta)\pm\tfrac{h^\vee-1}{2}\pm s +n(k+h^\vee)
\\&<\pm (s-\tfrac{ k+1}{2})-(\nu-\rho_R+\rho^\natural|\eta)+k+h^\vee\\
&\le \pm (s-\tfrac{ k+1}{2})-(\nu-\rho_R+\rho^\natural|\eta)+\tfrac{h^\vee}{2}-\tfrac{1}{2} +(\xi|\nu-\rho_R)-(\nu-\rho_R + \rho^{\natural}|\eta_{\min})\\
&\le \pm (s-\tfrac{ k+1}{2})-(\rho^\natural|\eta)+\tfrac{h^\vee}{2}-\tfrac{1}{2}-(\nu-\rho_R + \rho^{\natural}|\eta_{\min})\\
&= \pm (s-\tfrac{ k+1}{2})+(\rho^\natural|\xi-\eta)-(\rho^\natural|\xi)+\tfrac{h^\vee}{2}-\tfrac{1}{2}-(\nu-\rho_R + \rho^{\natural}|\eta_{\min})\\
&\le \pm (s-\tfrac{ k+1}{2})-(\nu-\rho_R + \rho^{\natural}|\eta_{\min})
\end{align*}
and we can conclude if $(\nu-\rho_R + \rho^{\natural}|\eta_{\min})>0$. Otherwise, we repeat the argument with $-\eta_{\min}$.

Finally, we check that $(\widehat \nu_s+\widehat \rho^{\tw}|\a)\neq \tfrac{n}{2}(\a|\a)$ for $n\in\nat$ and $\a=m\d+\theta,\,m\in\ZZ_+$ or
 $\a=m\d-\theta,\,m\in\nat$. Observe that 
 \begin{align*}&((k+h^\vee)\L_0+s\theta+\nu-\rho_R+\rho|m\d-\theta)=\\&\leq k+h^\vee-2 s -h^\vee+1
\\&= -2(s -\tfrac{k+1}{2})\leq 0,
\end{align*}
by our initial assumption. Also, by replacing $\eta_{\min}$ by its opposite, we can assume
$(\nu-\rho_{R}+\rho^\natural|\eta_{\min})\geq 0$. We have, using again \eqref{rr1}, \eqref{rr3}
\begin{align*}&((k+h^\vee)\L_0+s\theta+\nu-\rho_R+\rho|m\d+\theta)=\\&=m(k+h^\vee)+2 s +h^\vee-1
\\&\leq 2(s -\tfrac{k+1}{2})+k+h^\vee\leq 2 |(\nu-\rho_{R}+\rho^\natural|\eta_{\min})|+k+h^\vee
\\&\le (\nu-\rho_{R}|\xi-(-\eta_{\min}))+(\rho^\natural|\eta_{\min})+\tfrac{h^\vee}{2}-\tfrac{1}{2}
\\&\le (\rho^\natural|\eta_{\min}+\xi)\leq 0.
\end{align*}
\end{proof}

For calculation of characters of $H(L(\widehat\nu_s))$ we will also need
\begin{lemma}\label{hcthetahalf} Assume that $\theta/2$ is  a root of $\g$. Assume $\nu\in P^+_k$ is not Ramond extremal and
$s=\tfrac{2k+1}{4}$.
 Then \eqref{pp} holds.
\end{lemma}
\begin{proof}First of all observe that
$(\nu-\rho_R+\rho^\natural|\eta)<0$ for all $\eta\in\Dp_{-1/2}$. This is due to the fact that $\nu$ is not Ramond extremal so $\nu-\rho_R+\rho^\natural$ is dominant and regular for $\D^\natural$. Moreover  $\eta$ is a short positive root of $\g^\natural$. Note also that
$$
(\widehat\nu_s+\rhat^{\,\tw}|\d-\theta)=((k+h^\vee)\L_0+s\theta+\nu-\rho_R+\rho+\tfrac{\theta}{4}|\d-\theta)=0.
$$

We check that $(\widehat \nu_s+\widehat \rho^{\tw}|\a)\neq 0$ for all  isotropic roots as described in \eqref{r1} and \eqref{root2}

We start considering  the roots \eqref{r1}. Since 
$(\nu-\rho_R+\rho^\natural|\eta_{\min})<0$,
\begin{align*}&(\widehat\nu_s+\rhat^{\,\tw}|\pm(-\half\d+\theta/2)+\eta)=(\nu-\rho_R+\rho^\natural|\eta)<0.
\end{align*}

We now deal with  the roots \eqref{root2}. The case when $\a=\pm(-\half\d+\theta/2)+\eta+n\d$ is handled as above, since $k+h^\vee<0$.
In the remaining case we argue as in Lemma \ref{hc} using \eqref{r1} and \eqref{root2}:
 We have
\begin{align*}&(\widehat\nu_s+\rhat^{\,\tw}|\pm(-\half\d+\theta/2)-\eta+n\d)=-(\nu-\rho_R+\rho^\natural|\eta)+n(k+h^\vee)
\\&\le-(\nu-\rho_R+\rho^\natural|\eta)+k+h^\vee\\
&\le -(\nu-\rho_R+\rho^\natural|\eta)+\tfrac{h^\vee}{2}-\tfrac{1}{2} +(\xi|\nu-\rho_R)+(\nu-\rho_R + \rho^{\natural}|\eta_{\min})\\
&\le -(\rho^\natural|\eta)+\tfrac{h^\vee}{2}-\tfrac{1}{2}+(\nu-\rho_R + \rho^{\natural}|\eta_{\min})\\
&= (\rho^\natural|\xi-\eta)-(\rho^\natural|\xi)+\tfrac{h^\vee}{2}-\tfrac{1}{2}+(\nu-\rho_R + \rho^{\natural}|\eta_{\min})\\
&\le (\nu-\rho_R + \rho^{\natural}|\eta_{\min})<0.
\end{align*}

Finally, we check that $(\widehat \nu_s+\widehat \rho^{\tw}|\a)\neq \tfrac{n}{2}(\a|\a)$ for $n\in\nat$ and $\a=m\d+\theta,\,m\in\ZZ_+$, 
 $\a=m\d-\theta,\,m\in\nat$, $\a=m\d+\half\d-\theta/2,\,m\in\ZZ_+$, 
 $\a=m\d+\half\d+\theta/2,\,m\in\Z_+$. Observe that 
 \begin{align*}&(\widehat\nu_s+\rhat^{\,\tw}|m\d-\theta)=(m-1)(k+h^\vee)\le0,
\end{align*}
and
\begin{align*}&(\widehat\nu_s+\rhat^{\,\tw}|m\d+\theta)=m(k+h^\vee)+k+1 +h^\vee-1=(m+1)(k+h^\vee)<0.
\end{align*}
Likewise
 \begin{align*}&(\widehat\nu_s+\rhat^{\,\tw}|m\d+\half(\d-\theta))=m(k+h^\vee)\le0,
\end{align*}
and
\begin{align*}&(\widehat\nu_s+\rhat^{\,\tw}|m\d+\half(\d+\theta))=m(k+h^\vee)+k+h^\vee=(m+1)(k+h^\vee)<0.
\end{align*}
\end{proof}
 
 We also need to check integrability of $L(\widehat\nu_s)$. 
\begin{lemma}\label{notextremal}
  If $\nu\in P^+_k$ is not Ramond extremal then the $\ga^{\,\tw}$-module
$L(\widehat\nu_s)$ is integrable for all $s\in\C$.
\end{lemma}
\begin{proof}A set of simple roots for $\Da^\natural_+$ is $\Pia^\natural=\{\d-\theta_i\mid i=1,\ldots,s\}\cup\Pi^\natural$. 

If $\be$ is an odd isotropic root we let $r_\be$ denote the corresponding odd reflection (see \cite{KWNT} for details on odd reflections).
First of all we observe that, if $\nu$ is not Ramond extremal, then
\begin{equation}\label{nudominant}(\widehat\nu_s+\rhat^{\,\tw}|\a^\vee)\in\nat
\text{ for all }\a\in\Pia^\natural.
\end{equation}
Indeed
$$
(\widehat\nu_s+\rhat^{\,\tw}|(\d-\theta_i)^\vee)=((k+h^\vee)\L_0+(s+\e(\s_R)/4)\theta+(\nu-\rho_R)+\rho|(\d-\theta_i)^\vee)$$
and, since $\nu-\rho_R$ is non-extremal, as in the proof of Lemma 11.4 of \cite{KMP1}, it follows that 
$$
(\widehat\nu_s+\rhat^{\,\tw}|(\d-\theta_i)^\vee)=M_i(k)+\chi_i+1-(\nu-\rho_R|\theta_i^\vee)\in\nat.
$$
If $\a\in\Pi^\natural$, then
$$
(\widehat\nu_s+\rhat^{\,\tw}|\a^\vee)=(\nu-\rho_R+\rho^\natural|\a^\vee)\in\nat.
$$

Assume first that $\theta/2$ is not a root of $\g$. 
Both
 $\a_0=\d-\theta$ and $\a_1=-\half\d+\theta/2+\eta_{\min}$ are in $\Pia$ and $\a_0+\a_1=\half\d-\theta/2+\eta_{\min}$ is an odd isotropic root. Set $\widehat\Pi'=r_{\a_0+\a_1}r_{\a_1}\widehat \Pi$ and let $(\widehat\nu_{s})'$,   $(\rhat^{\,\tw})'$ be  the highest weight of $L(\widehat\nu_s)$ and the Weyl vector with respect to $\widehat\Pi'$. 
 
 We note that  $\Pia^\natural\subset \widehat\Pi\cup\widehat\Pi'$, so,
to check that $L(\widehat\nu_s)$ is integrable, it is enough to prove  for each root in $\a=\Pia^\natural\cap\Pia$ that $\widehat\nu_s(\a^\vee)\in\Z_+$ and, if  $\a=\Pia^\natural\cap\Pia'$, that $(\widehat\nu_s)'(\a^\vee)\in\Z_+$.
If $(\widehat\nu_s|\a_1)\ne0$ and $(\widehat\nu_s|\a_1+\a_0)\ne0$ this check is equivalent to \eqref{nudominant}.

If $(\widehat\nu_s|\a_1)=0$ and  $(\widehat\nu_s+\a_1|\a_0+\a_1)\ne0$ then $(\widehat\nu_s)'+(\rhat^{\,\tw})'=\widehat\nu_s+\rhat^{\,\tw}+\a_1$. If  $(\widehat\nu_s|\a_1)\ne0$ and  $(\widehat\nu_s+\a_1|\a_0+\a_1)=0$ then 
$(\widehat\nu_s)'+(\rhat^{\,\tw})'=\widehat\nu_s+\rhat^{\,\tw}+\a_0+\a_1$. If  $(\widehat\nu_s|\a_1)=(\widehat\nu_s+\a_1|\a_0+\a_1)=0$ then $(\widehat\nu_s)'+(\rhat^{\,\tw})'=\widehat\nu_s+\rhat^{\,\tw}+2\eta_{\min}$. In all cases    
$(\widehat\nu_s)'+(\rhat^{\,\tw})'=\widehat\nu_s+\rhat^{\,\tw}+\gamma$  with $\gamma_{|\h^\natural}=p\eta_{\min}$, $p=1$ or $2$.

If $\a\in\Pia^\natural\cap \Pia$ then, as computed above, $(\widehat\nu_s+\rhat^{\,\tw}|\a^\vee)\in\nat$.
If $\a\in\Pia^\natural\cap \Pia'$, 
we have to check that 
$$
(\widehat\nu_s+\rhat^{\,\tw}+\gamma|\a^\vee)\in\nat.
$$
If $\a\in\Pi^\natural\cap\Pi'$, then, one checks that $(\eta_{\min}|\a^\vee)\ge0$ hence
$$
(\widehat\nu_s+\rhat^{\,\tw}+\gamma|\a^\vee)=(\nu-\rho_R+\rho^\natural+p\eta_{\min}|\a^\vee)\in\nat.$$
If $\a=\d-\theta_i\in\Pia'\setminus\Pia$ then $\g=D(2,1;a)$ and $(\gamma|(\d-\theta_i)^\vee)=-p(\eta_{\min}|\theta_i^\vee)=p$ hence
$$
(\widehat\nu_s+\rhat^{\,\tw}+\gamma|(\d-\theta_i)^\vee)\in \nat.
$$
 
We now discuss the cases $\g=spo(2|2r+1)$, $r>1$ and $\g=G(3)$. Set $\a_1=-\half\d+\theta/2+\eta_{\min}$ and  $\Pia'=r_{\a_1}\Pia$. We note that $\Pia^\natural\subset \Pia\cup\Pia'$. If $(\widehat\nu_s|\a_1)\ne0$, we can conclude using \eqref{nudominant}.

 If $(\widehat\nu_s|\a_1)=0$,
we observe that $\eta_{\min}$ is the only element in $\Pia^\natural\setminus \Pia$ and 
$$
((\widehat\nu_s)'+(\rhat^{\,\tw})'|\eta_{\min}^\vee)=(\widehat\nu_s+\rhat^{\,\tw}|\eta_{\min}^\vee)+(\eta_{\min}|\eta_{\min}^\vee)\in \nat.
$$
It remains to discuss the case $\g=spo(2|3)$. In this case $\Pia^\natural=\{\d-\eta_{\min},\eta_{\min}\}$. Set $\a_1=-\half\d+\d_1+\eta_{\min}$, $\a_2=\half\d+\d_1-\eta_{\min}$ and $\Pia'=r_{\a_1}\Pia$, $\Pia''=r_{\a_2}\Pia$. We note that $\Pia^\natural\subset \Pia'\cup\Pia''$. If $(\widehat\nu_s|\a_i)\ne0$ for $i=1,2$, then we can conclude using \eqref{nudominant}. If $(\widehat\nu_s|\a_1)=0$ then $\Pia^\natural\cap\Pia'=\{\eta_{\min}\}$ and
$$
((\widehat\nu_s)'+(\rhat^{\,\tw})'|\eta_{\min}^\vee)=(\widehat\nu_s+\rhat^{\,\tw}|\eta_{\min}^\vee)+(\eta_{\min}|\eta_{\min}^\vee)\in \nat.
$$
If $(\widehat\nu_s|\a_2)=0$ then $\Pia^\natural\cap\Pia''=\{\d-\eta_{\min}\}$ and
$$
((\widehat\nu_s)''+(\rhat^{\,\tw})''|(\d-\eta_{\min})^\vee)=(\widehat\nu_s+\rhat^{\,\tw}|(\d-\eta_{\min})^\vee)-(\eta_{\min}|(\d-\eta_{\min})^\vee)\in \nat.
$$

\end{proof}

\begin{lemma}\label{isextremal}
 If $\nu\in P^+_k$ is Ramond extremal and $s$ is such that $\ell(s)=A(k,\nu)$, then $L(\widehat\nu_s)$ is integrable.
\end{lemma}
\begin{proof}We argue as in Lemma \ref{notextremal}, using the same notation. 
First we prove that 
\begin{equation}\label{nuplusrho}
(\widehat\nu_{s}+\rhat^{\,\tw}|\a^\vee)\in\nat\text{ for all $\a\in\Pia^\natural\cap\Pia$.}
\end{equation}

 If $\a\in\Pi^\natural$, then 
$$
(\widehat\nu_s|\a^\vee)=(\nu+\rho_R|\a^\vee)\in\Z_+,
$$
so, if $\a\in\Pi^\natural\cap\Pia$, then \eqref{nuplusrho} holds.
Moreover,
$$
(\widehat\nu_s+\rhat^{\,\tw}|(\d-\theta_i)^\vee)=M_i(k)+\chi_i+1-(\nu-\rho_R|\theta_i^\vee)=M_i(k)-(\nu|\theta_i^\vee)+(\rho_R|\theta_i^\vee).
$$
Note that, if $(\d-\theta_i)\in\Pia$, then $(\rho_R|\theta_i^\vee)>0$ hence
\eqref{nuplusrho} holds   for all $\a\in\Pia^\natural\cap\Pia$.
Next, we check that 
\begin{equation}\label{nuplusrhoprime}
((\widehat\nu_{s})'+(\rhat^{\,\tw})'|\a^\vee)\in\nat\text{ for all $\a\in\Pia^\natural\cap\Pia'$.}
\end{equation}

Assume $\theta/2$ is not a root of $\g$. 
Since $\ell(s)=A(k,\nu)$,  \eqref{spartic} holds. Moreover,
\begin{align}
&(\widehat \nu_{s}+\widehat\rho^{\tw}|\pm(\half \d-\theta/2)+\eta_{\min})\\
&=\notag
((k+h^\vee)\L_0+(\tfrac{k+1}{2} \pm (\nu-\rho_{R}+\rho^\natural|\eta_{\min}))\theta+\nu+\rho-\rho_R|\pm(\half \d-\theta/2)+\eta_{\min})
\\
&=\pm(\half (k+h^\vee)-\tfrac{k+1}{2}\pm(\nu-\rho_{R}+\rho^\natural|\eta_{\min})-\half(h^\vee-1))+(\nu-\rho_{R}+\rho^\natural|\eta_{\min})\notag\\
&=(\nu-\rho_{R}+\rho^\natural|\eta_{\min})\pm(\nu-\rho_{R}+\rho^\natural|\eta_{\min}).\notag
\end{align}
The outcome of this computation is that, if $(\nu-\rho_{R}+\rho^\natural|\eta_{\min})\ne0$ and $s=\tfrac{k+1}{2}\pm(\nu-\rho_{R}+\rho^\natural|\eta_{\min})$, then 
$$
(\widehat \nu_{s})'+(\widehat\rho^{\tw})'=\widehat \nu_{s}+\widehat\rho^{\tw}\pm(\half \d-\theta/2)+\eta_{\min},
$$
while, if $(\nu-\rho_{R}+\rho^\natural|\eta_{\min})=0$, 
$$
(\widehat \nu_{s})'+(\widehat\rho^{\tw})'=\widehat \nu_{s}+\widehat\rho^{\tw}+2\eta_{\min}.
$$

If $\a\in\Pi^\natural\cap\Pi'$, then one checks that $(\eta_{\min}|\a^\vee)\ge0$ so in all cases 
$$
((\widehat \nu_{s})'+(\widehat\rho^{\tw})'|\a^\vee)=(\nu-\rho_R+\rho^\natural+p\eta_{\min}|\a^\vee)\ge(\nu-\rho_R+\rho^\natural+\eta_{\min}|\a^\vee)=(\widehat \nu_{s}+(\widehat\rho^{\tw})'|\a^\vee)
$$
hence $((\widehat \nu_{s})'+(\widehat\rho^{\tw})'|\a^\vee)\ge(\nu+\rho_R|\a^\vee)+1\in\nat$.

If $\d-\theta_i\in \Pia'\setminus\Pia$ then $\g=D(2,1;a)$ and $(\eta_{\min}|\theta_i^\vee)=-1$, so 
\begin{align*}
&((\widehat \nu_{s})'+(\widehat\rho^{\tw})'|(\d-\theta_i)^\vee)=M_i(k)+\chi_i+1-(\nu-\rho_R+p\eta_{\min}|\theta_i^\vee)\\
&\ge M_i(k)-(\nu|\theta_i^\vee)+(\rho_R|\theta_i^\vee)+1\in \nat.
\end{align*}
This completes the proof in the cases when $\theta/2$ is not a root of $\g$.

We now discuss the cases $\g=spo(2|2r+1)$, $r>1$, and $\g=G(3)$. As in the previous cases we need only to check that
$((\widehat\nu_s)'+(\rhat^{\,\tw})'|\a^\vee)\in\nat$ for all $\a\in\Pia^\natural\cap\Pia'$. Since $\ell(s)=A(k,\nu)$, by Lemma \ref{Aknusecond}, $s=\frac{2k+1}{4}$. By Proposition \ref{82} below
$$
(\widehat\nu_s+\rhat^{\,\tw}|\a_1)=0.
$$
It follows that
$(\widehat\nu_s)'+(\rhat^{\,\tw})'=\widehat\nu_s+\rhat^{\,\tw}+\a_1$. If $\a\in\Pia^\natural\setminus \Pia$, then $\a=\eta_{\min}$, so
$$
((\widehat \nu_{s})'+(\widehat\rho^{\tw})'|\a^\vee)=(\nu-\rho_R+\rho^\natural+\eta_{\min}|\eta_{\min}^\vee)=(\widehat \nu_{s}+(\widehat\rho^{\tw})'|\eta_{\min}^\vee)=(\nu+\rho_R|\eta_{\min}^\vee)+1\in\nat.
$$

We finally discuss the case $\g=spo(2|3)$. Since $\ell(s)=A(k,\nu)$, by Lemma \ref{Aknusecond}, $s=\frac{2k+1}{4}$. If $\nu=0$, then
$
(\widehat\nu_s+\rhat^{\,\tw}|\a_1)=0
$
and
\begin{align*}
&(\widehat \nu_{s}+\widehat\rho^{\tw}|\a_2)=
((k+h^\vee)\L_0+\tfrac{k+1}{2} \theta+\rho-\rho_R|\half \d+\theta/2-\eta_{\min})=k+h^\vee\ne0,\notag
\end{align*}
so
$$
((\widehat \nu_{s})'+(\widehat\rho^{\tw})'|\eta_{\min}^\vee)=(\eta_{\min}|\eta_{\min}^\vee)=2,
$$
while
$$
((\widehat \nu_{s})''+(\widehat\rho^{\tw})''|(\d-\eta_{\min})^\vee)=((\widehat \nu_{s})+(\widehat\rho^{\tw})|(\d-\eta_{\min})^\vee)=M_1(k)-1+(\rho_R|\eta_{\min}^\vee)=M_1(k)\in\nat.
$$

If instead $\nu=\tfrac{M_1(k)}{2}\eta_{\min}$, then 
$
(\widehat\nu_s+\rhat^{\,\tw}|\a_1)=0
$ and
\begin{align*}
&(\widehat \nu_{s}+\widehat\rho^{\tw}|\a_1)=
((k+h^\vee)\L_0+\tfrac{k+1}{2} \theta+\nu+\rho-\rho_R|-\half \d+\theta/2+\eta_{\min})=k+h^\vee\ne0,
\end{align*}
so
$$
((\widehat \nu_{s})'+(\widehat\rho^{\tw})'|\eta_{\min}^\vee)=((\widehat \nu_{s})+(\widehat\rho^{\tw})|\eta_{\min}^\vee)=M_1(k)\in \nat,
$$
and
$$
((\widehat \nu_{s})''+(\widehat\rho^{\tw})''|(\d-\eta_{\min})^\vee)=((\widehat \nu_{s})+(\widehat\rho^{\tw})+\a_2|(\d-\eta_{\min})^\vee)=(\eta_{\min}|\eta_{\min}^\vee)\in\nat.
$$
\end{proof}

\begin{cor}\label{910}
(a) If $\nu$ is not Ramond extremal, then $H(L(\widehat \nu_s))$ is a Ramond twisted $\Ws$-module for all $s\in \C$.

\noindent (b) If $\nu$ is  Ramond extremal and  $\ell(s)=A(k,\nu)$, then $H(L(\widehat \nu_s))$ is a Ramond twisted $\Ws$-module.
\end{cor}
\begin{proof}
Combining Lemmas \ref{notextremal} and \ref{isextremal} with  \cite[Theorem 5.3.1]{GS}, we see that the modules $L(\widehat \nu_s)$, described in the statement, are $\s_R$-twisted $V_k(\g)$-modules, so their quantum Hamiltonian reduction gives $\s_R$-twisted modules for $H(V_k(\g))=\Ws$.
\end{proof}
  It was conjectured in \cite{KRW} that, for an admissible (in the sense of \cite{KRW}) highest weight $V^k(\g)$-module $L(\widehat\L)$, either $H(L(\widehat\L))$ is irreducible or $0$, and the latter happens iff $\widehat\L$ is degenerate. This conjecture was subsequently proved by Arakawa \cite{Araduke}. 
We used this result in \cite{KMP1} to compute character formulas for irreducible unitary highest weight modules over $\Wu$.

We would like to use the same approach for Ramond twisted irrreducible highest weight modules but, for that, we need the following Conjecture, which is a ``twisted" analogue of that in \cite{KRW}.

\begin{Conjecture}\label{Arakawa}\
\begin{enumerate}
\item[a)]If  $M$ is in category $\mathcal O$ of $\ga^{\tw}$-modules, then $H_j(M)=0$ if $j\ne0$. 
\item[b)]  
Assume that $\widehat\nu_s$ is nondegenerate so that $H(L(\widehat\nu_s))$ is nonzero. Then
\begin{itemize}
\item If $\theta/2$ is not a root of $\g$, then $H_0(L(\widehat\nu_s))$ is   irreducible.
\item  If $\theta/2$ is a root of $\g$, then $H_0(L(\widehat \nu_{s}))$ is  irreducible, or a direct sum of two irreducible $\Wu$-modules. In the last case $\ell(s)=A(k,\nu)$  and the second module is isomorphic to the first with the opposite parity. 
\end{itemize}
\end{enumerate}
\end{Conjecture}

Some evidence towards this conjecture is provided in Section 13: see Remark \ref{evv}. Consequences of Conjecture \ref{Arakawa} are the following results.
\begin{proposition} \label{Lnust} Assume Conjecture \ref{Arakawa} and that  $\widehat \nu_s$ is non-degenerate.

If $\ell(s)\ne A(k,\nu)$,
then $H_0(L(\widehat\nu_{s}))= L^W(\nu,\ell(s))$ and, if $\ell(s)=A(k,\nu)$,  then $H_0(L(\widehat\nu_{s}))=L^W(\nu,\ell(s))\oplus \e(\s_R)\Pi L^W(\nu,\ell(s))$, where $\Pi$ is the reversal of parity functor.
\end{proposition}
\begin{proof}
By Conjecture \ref{Arakawa} a), the functor $H$ is equal to $H_0$ and it is exact. It follows that 
$H(L(\widehat \nu_s))$ is a quotient of $H(M(\widehat\nu_s))$, hence, by Proposition \ref{Hhw}, $H(L(\widehat \nu_s))$ is a quotient of $M^W(\nu,\ell(s))$. 
Since, by hypothesis, $\widehat\nu_s$ is non-degenerate, $H(L(\widehat \nu_s))$ is non-zero. By Conjecture \ref{Arakawa} b) any direct summand of  $H(L(\widehat \nu_s))$ is  a highest weight module of highest weight $(\nu,\ell(s))$. 
\end{proof}

Now we can prove character formulas in the Ramond sector.
\begin{theorem}\label{characters}Assume Conjecture \ref{Arakawa}.
Let $k$ be in the unitary range
and let $\nu\in P^+_k$ be a weight which is not Ramond extremal. 
Fix $\ell\ge A(k,\nu)$ and, if $\theta/2$ is not a root of $\g$, assume $\ell\ne A(k,\nu)$. Choose $s$ such that $\ell=\ell(s)$ and, if  $\ell<B(k,\nu,\rho_R)$, assume that $s$ satisfies \eqref{hyp}.
Then,  if $\ell>A(k,\nu)$, we have
\begin{align}
&\widehat F^{R}ch\, L^W(\nu,\ell)\notag\\
&=q^{\tfrac{(\widehat\nu_s|\widehat\nu_s+2\rhat^{\,\tw})}{2(k+h^\vee)}+a(k)}e^{-\rho_R}\sum_{w\in\Wa^\natural}det(w)q^{-(w(\widehat\nu_s+\rhat^{\,\tw})- \rhat^{\,\tw})(x+D)}e^{(w(\widehat\nu_s+\rhat^{\,\tw})- \rhat^{\,\tw})_{|\h^\natural}},\label{ff1}
\end{align}
while, if $\theta/2$ is a root of $\g$ and $\ell=A(k,\nu)$, we have
\begin{align}
&\widehat F^{R}ch\, L^W(\nu,\ell)\notag\\
&=\tfrac{1}{2}q^{\tfrac{(\widehat\nu_s|\widehat\nu_s+2\rhat^{\,\tw})}{2(k+h^\vee)}+a(k)}e^{-\rho_R}\sum_{w\in\Wa^\natural}det(w)q^{-(w(\widehat\nu_s+\rhat^{\,\tw})- \rhat^{\,\tw})(x+D)}e^{(w(\widehat\nu_s+\rhat^{\,\tw})- \rhat^{\,\tw})_{|\h^\natural}}.\label{ff1theta}
\end{align}
where $a(k)$ is given by \eqref{aofk}.
\end{theorem}
\begin{proof} By Lemma \ref{notextremal}, $L(\widehat\nu_s)$ is integrable. We want to apply Proposition \ref{typical}, so we need only to check that \eqref{pp} holds.

We start by checking that \eqref{pp} holds. If $\ell>B(k,\nu,\rho_R)$, then $\ell=\ell(s)$ with $s=\frac{k+1}{2} +\sqrt{-1}t_0$ with $t_0\ne 0$. It is then obvious that \eqref{pp} holds
 since the left hand side of \eqref{pp} is not real.
  If instead $\theta/2$ is not a root of $\g$ and $ A(k,\nu)< \ell\leq B(k,\nu,\rho_R)$, then \eqref{pp} holds in this case by Lemma 
\ref{hc}, and,
if 
$\theta/2$ is a root and $\ell=A(k,\nu)$, \eqref{pp} holds by Lemma \ref{hcthetahalf}. 

 By Proposition \ref{Lnust} and Conjecture  \ref{Arakawa}, we have that, if $\ell>A(k,\nu)$, 
 \begin{equation}\label{cu}\sum_j(-1)^jH_j(L(\widehat\nu_s))=H_0(L(\widehat\nu_s))=L^W(\nu,\ell(s)),\end{equation}
while, , if $\theta/2$ is a root of $\g$ and $\ell=A(k,\nu)$, 
 \begin{equation}\label{cutheta}\sum_j(-1)^jH_j(L(\widehat\nu_s))=H_0(L(\widehat\nu_s))=L^W(\nu,\ell(s))\oplus L^W(\nu,\ell(s)).
 \end{equation}
Applying  Proposition \ref{typical}, 
equalities  \eqref{ff1} and \eqref{ff1theta} now follow from \eqref{cu}, \eqref{cutheta}.  
\end{proof}

We finish this section explaining how Theorem \ref{characters} implies unitarity of $L^W(\nu,\ell)$ for $\nu$ that is not Ramond extremal and $\ell\ge A(k,\nu)$. For this we need a free version of the modules $N(\mu,\nu)$. This is constructed as follows.
 Let $y$ be an indeterminate. Define an action of the abelian Lie algebra $\mathcal H_0=\C a+\C K$ on $\C[y]$ by letting $K$ act as the identity and $a$ (cf. \eqref{norm}) act by multiplication by $y$. Extend this action to $\C[t]\otimes \C a$ by letting $at^j$ act trivially.  Let $M(y)$ be the corresponding induced module to the affine algebra  $\C[t,t^{-1}]\otimes \C a\oplus \C K$. This module can be regarded as a $V^1(\C a)$--module by means of the field $ Y(a,z)$ defined by setting, for $m\in M(y)$,
 $$
 Y(a,z)m=\sum_{j\in\ZZ}(t^{j}\otimes a)\cdot m\,z^{-j-1}.
 $$
Set $M:= M(y)\otimes L(\nu)\otimes F(\g_{1/2},\si_R)$ and define
\begin{equation}\label{emb}
\widetilde N(y,\nu)=\Psi(W^k_{\min}(\g))\cdot (1\otimes \C[y]\otimes v_\nu\otimes 1)\hookrightarrow M.
\end{equation}
Since $M(y)\otimes L(\nu)\otimes F(\g_{1/2},\si_R)$ is free as a $\C[y]$--module, $\widetilde N(y,\nu)$ is also free. 
If $\mu\in\C$, set also
\begin{equation}\label{nt}\widetilde N(\mu,\nu)=\left(\C[y]/(y-\mu)\right)\otimes_{\C[y]} \widetilde N(y,\nu). 
\end{equation}

By construction $\widetilde N(\mu,\nu)$ is clearly a $\si_R$-twisted highest weight module for $W^k_{\min}(\g)$. 
Recall from \eqref{nmunu}, \eqref{ell0}  the definitions of $N(\mu,\nu)$ and $ \ell_0(\mu,\nu)$.  The embedding \eqref{emb} defines a map
$\psi:\widetilde N(\mu,\nu)\to \left(\C[y]/(y-\mu)\right)\otimes_{\C[y]}M=M(1,\mu)\otimes L(\nu)\otimes F(\g_{1/2},\si_R)$ whose image is $N(\mu,\nu)$.
\begin{prop}\label{forse} There is a countable set $\mathcal M\subset \C$ such that, for  $\mu\in \C\setminus\mathcal M$, the map $\psi$ defines an isomorphism between $\widetilde N(\mu,\nu)$ and $N(\mu,\nu)$.
\end{prop}
\begin{proof} Given $\l(y)\in\C[y]$, set   $M_{\l(y)}=\{m\in M\mid L_0 m= \l(y)m\}$ and $M_n=M_{\ell_0(y,\nu)+n}$. We observe  that 
\begin{enumerate}
\item $M=\bigoplus_{n\in\half \ZZ_+}M_n;$
\item $M_n=M_{\ell_0(y,\nu)+n}$ is a $\C[y]$-module of finite rank.
\end{enumerate}
Let $\{m^{(n)}_{1},\ldots,m^{(n)}_{i_n}\}$ be a basis of $M_n$ over $\C[y]$. Consider $\widetilde N(y,\mu)\cap M_n$: it is a free module over $\C[y]$.  Fix a  basis $\mathcal V=\{v_1,\ldots,v_r\}$, so that $v_i=\sum\limits_{j=1}^{i_n}p_{ij}(y) m^{(n)}_{j}$. Then the $r\times i_n$ matrix $(p_{ij}(\mu))$ has rank less than $r$ for a finite  number of values of $\mu$: otherwise the matrix with polynomial entries $(p_{ij}(y))$ has rank less than $r$, against the fact that $\mathcal V$ is a basis.  Let $E_n$ be the set of these values, and define $\mathcal M=\cup_nE_n$. We now show that  $\psi$ is an isomorphism outside $\mathcal M$. 
 Set $\widetilde N(\mu,\nu)_n=\{v\in N(\mu,\nu)\mid L_0v=(\ell_0(\mu,\nu)+n)v\}$.
Then  $\widetilde N(\mu,\nu)=\bigoplus_{n\in\half \ZZ_+}\widetilde N(\mu,\nu)_n$, and the previous argument shows that if $\mu\notin \mathcal M$, then $\psi_{| \widetilde N(\mu,\nu)_n}$ is injective, hence an isomorphism onto its image.
\end{proof}

\begin{lemma} \label{muofs}Set $\mu(s)=\sqrt{-1}\frac{\sqrt{2}s}{\sqrt{|k+h^\vee}|}$. 
Then
$$
\ell_0(\mu(s),\nu-\rho_R)=\ell(s),
$$
where $\ell(s)$ is defined by \eqref{ls}.

In particular, the module  $L^W(\nu,\ell(s))$ is the irreducible quotient of $N(\mu(s),\nu-\rho_R)$.
\end{lemma}
\begin{proof} It follows  from \eqref{ell0} that 
$$\ell_0(\mu,\nu-\rho_R)-B(k,\nu,\rho_R)=\frac{\mu^2}{2}-s_k\mu+\frac{(k+1)^2}{4(k+h^\vee)}=\half (\mu-s_k)^2=\half (\mu-\sqrt{-1}\tfrac{(k+1)}{\sqrt{2|k+h^\vee|}})^2.$$
On the other hand 
$$\ell(s)-B(k,\nu,\rho_R) = \frac{(s-\tfrac{k+1}{2})^2}{k+h^\vee},$$
so that $$ \tfrac{1}{\sqrt{2}}\mu=\sqrt{-1}\tfrac{(k+1)}{2\sqrt{|k+h^\vee|}}\pm(\sqrt{-1}\frac{s-\tfrac{k+1}{2}}{\sqrt{|k+h^\vee}|}).$$
Choosing the plus sign in the previous formula we have the claim.\end{proof}

We note that 
\begin{equation}\label{iop}\widetilde N(\mu,\nu)=\bigoplus_{\l\in\h^\natural,	\ n\in   \half\ZZ_+ } \widetilde N(\mu,\nu)_{(\l,\ell_0(\mu,\nu)+n)},\end{equation}
where $\widetilde N(\mu,\nu)_{(\l,\ell)}=\{m\in \widetilde N(\mu,\nu)\mid h_{(0)}m=\l(h)m,\,h\in\h^\natural, L_0m=\ell m\}.$
From \eqref{iop} we deduce that 
\begin{align*}ch\,\widetilde N(\mu,\nu)&=\sum_{\l\in\h^\natural, n\in \half\ZZ_+   }\dim \widetilde N(\mu,\nu)_{(\l,\ell_0(\mu,\nu)+n)} q^{\ell_0(\mu,\nu)+n}e^{\l}.\end{align*}
We have already noticed that   $\widetilde N(\mu,\nu)$ is free in the variable $\mu$, i.e.
$\dim\widetilde N(\mu,\nu)_{(\l,\ell_0(\mu,\nu)+n)}
$
 does not depend on $\mu$. Then 
\begin{align}\label{chnm}q^{-\ell_0(\mu,\nu)}ch\,\widetilde N(\mu,\nu)&=\sum_{\l\in\h^\natural, n\in \half\ZZ_+   } (\dim\widetilde N(\mu,\nu)_{(\l,\ell_0(\mu,\nu)+n)})q^ne^{\l}\end{align}
does not depend on $\mu$.

\begin{proposition}\label{basisfree}Assume Conjecture \ref{Arakawa}. Given $\ell>A(k,\nu)$ choose $s\in\C$ such that $\ell=\ell(s)$ and, if $A(k,\nu)<\ell\le B(k,\nu,\rho_R)$, choose $s$ in the interval \eqref{hyp}.

 Then 
$$\widetilde N(\mu(s),\nu-\rho_R)=L^W(\nu,\ell(s)).
$$
\end{proposition}
\begin{proof} 
Put $s_0=\tfrac{k+1}{2}+\sqrt{-1}t_0,\,t_0\in\R$. By Proposition \ref{forse}, there exists $t_0\ne0$ such that  $N(\mu(s_0),\nu-\rho_R)=\widetilde N(\mu(s_0),\nu-\rho_R)$.  Since $N(\mu(s_0),\nu-\rho_R)$ is    unitary, it is  irreducible. 
Hence   $\widetilde N(\mu(s_0),\nu-\rho_R)=L^W(\nu,\ell(s_0))$. In particular, by Theorem \ref{characters} (which uses Conjecture \ref{Arakawa}),$$\widehat F^{R} ch\,\widetilde N(\mu(s_0),\nu-\rho_R)=
q^{\tfrac{(\widehat\nu_{s_0}|\widehat\nu_{s_0}+2\rhat^{\,\tw})}{2(k+h^\vee)}+a(k)}e^{-\rho_R}\sum_{w\in\Wa^\natural}det(w)e^{ev(w(\widehat\nu_{s_0}+\rhat^{\,\tw})- \rhat^{\,\tw})}.$$
We now prove that 
\begin{equation}\label{finale} \widehat F^{R} ch\,\widetilde N(\mu(s),\nu-\rho_R)=
q^{\tfrac{(\widehat\nu_s|\widehat\nu_s+2\rhat^{\,\tw})}{2(k+h^\vee)}+a(k)}e^{-\rho_R}\sum_{w\in\Wa^\natural}det(w)e^{ev(w(\widehat\nu_s+\rhat^{\,\tw})- \rhat^{\,\tw})}\end{equation}
for any $s\in \C$. 

Since $w(x)=x$ for all $w\in\Wa^\natural$, 
$$
e^{ev(w(\widehat\nu_s+\rhat^{\,\tw})- \rhat^{\,\tw})}=q^{s_0-s}e^{ev(w(\widehat\nu_{s_0}+\rhat^{\,\tw})- \rhat^{\,\tw})},
$$
so the RHS of \eqref{finale} is
$$
q^{s_0-s}q^{\tfrac{(\widehat\nu_s|\widehat\nu_s+2\rhat^{\,\tw})}{2(k+h^\vee)}+a(k)}e^{-\rho_R}\sum_{w\in\Wa^\natural}det(w)\frac{e^{ev(w(\widehat\nu_{s_0}+\rhat^{\,\tw})- \rhat^{\,\tw})}}{\widehat F^{R}}=q^{\ell(s)-\ell(s_0)}ch\,\widetilde N(\mu(s_0),\nu-\rho_R).
$$
Since, by \eqref{chnm}, $q^{-\ell_0(\mu,\nu-\rho_R)}ch\,\widetilde N(\mu,\nu-\rho_R)$ does not depend on $\mu$, we obtain
that 
\begin{align*}
q^{\ell(s)-\ell(s_0)}ch\,\widetilde N(\mu(s_0),\nu-\rho_R)&=q^{\ell_0(\mu(s),\nu-\rho_R)}q^{-\ell_0(\mu(s_0),\nu-\rho_R)}ch\,\widetilde N(\mu(s_0),\nu-\rho_R)\\
&=ch\,\widetilde N(\mu(s),\nu-\rho_R),
\end{align*}
as claimed.
Now, if $s$ lies in \eqref{hyp}, by Theorem \ref{characters}  (here we use Conjecture \ref{Arakawa} again), we have 
\begin{equation}\label{ech} ch\,\widetilde N(\mu(s),\nu-\rho_R)=ch\,L^W(\nu,\ell(s)).\end{equation}
Since $L^W(\nu,\ell(s))$ is a quotient of $ \widetilde N(\mu(s),\nu-\rho_R)$, \eqref{ech} implies that they are isomorphic.
\end{proof}

The same proof of Theorem 11.1 of \cite{KMP1} provides the following extension of Corollary \ref{thetahalfisaroot}:
\begin{theorem}Assume Conjecture \ref{Arakawa}. If $\ell\ge A(k,\nu),$ $k$ is in the unitary range, and $\nu\in P^+_k$  is  not Ramond extremal, then $L^W(\nu,\ell)$  is a unitary $\si_R$-twisted $W^k_{\min}(\g)$--module.
\end{theorem}
\begin{proof}We can assume that $\theta/2$ is not a root of $\g$.
For each weight $(\l,m)$, fix a basis of $\widetilde N(\mu,\nu-\rho_R)_{(\l,m)}$ independent from $\mu$. Define $det_{(\l,m)}(\mu)$ to be the determinant of the matrix of the Hermitian invariant form $H$ in this basis. We have seen in the proof of Proposition \ref{basisfree} that there is $\mu_0$ with $\ell_0(\mu_0)>A(k,\nu)$ such that $\widetilde N(\mu_0,\nu-\rho_R)$ is unitary. By  Proposition \ref{basisfree}, which uses Conjecture \ref{Arakawa}, $\widetilde N(\mu,\nu-\rho_R)=L^W(\ell_0(\mu),\nu)$ if $\ell_0(\mu)>A(k,\nu)$, hence $det_{(\l,m)}(\mu)\ne0$ if $\ell_0(\mu)>A(k,\nu)$, thus the Hermitian invariant form $H$ remains positive definite for $\ell_0(\mu)>A(k,\nu)$ and it is positive semidefinite if $\ell_0(\mu)=A(k,\nu)$.
\end{proof}

\section{Explicit conditions for unitarity for Ramond twisted modules $L^W(\nu,\ell)$}\label{Explicit}
In Section \ref{necessary} we found necessary conditions of unitarity of $\s_R$-twisted $\Wu$-modules $L^W(\nu,\ell)$, where $k$ is in the unitarity range, see  Theorem \ref{32t}.

We conjecture that all these modules are unitary. There are three types  of these modules satisfying the necessary conditions of unitarity:
\begin{enumerate}
\item the modules $L^W(\nu,\ell)$ with $\nu\in P^+_k$ not Ramond extremal and $\ell\ge B(k,\nu,\rho_R)$; we proved in Section \ref{sufficient} that these modules are indeed unitary;
\item  if $\theta/2$ is not a root of $\g$,  the modules $L^W(\nu,\ell)$ with $\nu\in P^+_k$ not Ramond extremal and $A(k,\nu)\le\ell< B(k,\nu,\rho_R)$;  we proved in Section \ref{UbetweenAB} that these modules are unitary assuming Conjecture \ref{Arakawa};
\item the modules $L^W(\nu,\ell)$ with the weight $\nu\in P^+_k$  Ramond extremal
in which case $\ell=A(k,\nu)$ (by Lemma \ref{extremal}); we don't know how to establish unitarity in this case.
\end{enumerate}

Below, for each $\g$ from Table \ref{numerical}, we make explicit the necessary conditions of Theorem \ref{32t} and the sufficient conditions of Section \ref{UbetweenAB}.

\subsection{$\g=psl(2|2)$}\label{1111} In this case 
$$\g^\natural=sl(2),\ M_1(k)=-k-1,\ P^+_k=\{\nu=r\theta_1/2\mid r\in\Z_+,\ 0\le r\le M_1(k)\}.
$$ 
\begin{itemize}
\item The weight $\nu=0$  is  the only Ramond extremal weight. 
The necessary  conditions for unitarity for $L^W(0,\ell)$ are $M_1(k)\in\Z_+$ and $\ell= -\frac{k+1}{4}$.
\item 
If $\nu$ is not Ramond extremal then $\nu=\frac{r}{2}\theta_1$, $1\le r\le M_1(k)$ and
the necessary conditions for unitarity are $M_1(k)\in\Z_+$ and 
\begin{equation}\label{condsuffpsl22}
\ell\ge -\frac{k+1}{4}.
\end{equation}
The sufficient conditions for unitarity are $M_1(k)\in\Z_+$ and 
$$
\ell\ge -\frac{k^2+k+r^2}{4 k}.
$$
Condition \eqref{condsuffpsl22}  is also sufficient assuming Conjecture \ref{Arakawa}. 
\end{itemize}
The inequality $
\ell\ge \frac{M_1(k)}{4}$
is precisely the bound stated in \cite{ET3}.
\subsection{$\g=spo(2|3)$}\label{spo23} In this case
$$\g^\natural=sl(2),\ M_1(k)=-4k-2,\ P^+_k=\{\nu=r\theta_1/2\mid r\in\Z_+,\ 0\le r\le M_1(k)\}.
$$
The Ramond extremal weights are $\nu=0$ and $\nu=\tfrac{M_1(k)}{2}\theta_1$.
\begin{itemize}\item 
The necessary  conditions for unitarity for $L^W(\nu,\ell)$ with $\nu=r\theta_1/2$ Ramond extremal weight are $M_1(k)\in\Z_+$ and
$$
\ell=-\frac{8 k^2+10 k+2 r ^2+3}{32 k+16}.
$$
\item 
If $\nu$ is not Ramond extremal, then $\nu=\frac{r}{2}\theta_1$, $1\le r< M_1(k)$ and the necessary and sufficient conditions for unitarity are $M_1(k)\in\nat$ and 
$$
\ell\ge -\frac{8 k^2+10 k+2 r ^2+3}{32 k+16}.
$$
\end{itemize}
 In terms of $M_1(k)$, the inequality reads
\begin{equation}\label{forspo}
\ell\ge\frac{(M_1(k)-1)}{16}+\frac{r^2}{4 M_1(k)},
\end{equation}
which is precisely the bound stated in \cite[2.3.11]{M}.
\subsection{$\g=spo(2|2r)$, $r>2$} In this case
$$\g^\natural=so(2r),\ M_1(k)=-2k-1,
$$
$$
P^+_k=\{\nu=\sum_i m_i\e_i,\,m_i\in\half +\ZZ\text{ or }m_i\in\ZZ,\ m_1\geq\ldots\geq m_{r-1}\geq |m_r|,\ m_1+m_2\le M_1(k)\},
$$
and
\begin{align*}
A(k,\nu)&=\frac{-4\left( \sum_{i=1}^{r-1}(2 (r-i)-1)m_i +m_i^2\right)-4 k^2+2(r-4) k  +r-3}{16 (k+2-r)}\\
&=-\frac{\left( \sum_{i=1}^{r-1}(2 (r-i)-1)m_i +m_i^2\right)+p(k)}{4 (k+h^\vee)}.
\end{align*}

 If $\eta_{\min}=\e_r$, the Ramond extremal weights are the weights $\nu$ such that $m_r=-m_{r-1}$
 \begin{itemize}
 \item If $\nu$ is Ramond extremal, the necessary conditions for unitarity are $M_1(k)\in\Z_+$ and $\ell=A(k,\nu)$.
\item If $\nu$ is not Ramond extremal, the necessary condition for unitarity are  $M_1(k)\in\Z_+$ and
\begin{equation}\label{1ab}
\ell\ge A(k,\nu).
\end{equation}
The sufficient conditions are $M_1(k)\in\Z_+$ and
$$
\ell\ge B(k,\nu,\rho_R)=-\frac{\left( \sum_{i=1}^{r}(2 (r-i)-1)m_i +m_i^2\right)+p(k)}{4 (k+h^\vee)}.
$$
The conditions \eqref{1ab} are also sufficient assuming Conjecture \ref{Arakawa}. 
\end{itemize}

If $\eta_{\min}=-\e_r$,
the Ramond extremal weights are the weights $\nu$ such that $m_r=m_{r-1}$.
 \begin{itemize}
 \item If $\nu$ is Ramond extremal, the necessary conditions for unitarity are $M_1(k)\in\Z_+$ and $
\ell=A(k,\nu)$.
\item If $\nu$ is not Ramond extremal, the necessary conditions for unitarity are $M_1(k)\in\Z_+$ and
\begin{equation}\label{1abetamin}
\ell\ge A(k,\nu).
\end{equation}
The sufficient conditions are
$$
\ell\ge B(k,\nu,\rho_R)=-\frac{\left( \sum_{i=1}^{r}(|2 (r-i)-1|)m_i +m_i^2\right)+p(k)}{4 (k+h^\vee)}.
$$
The conditions \eqref{1abetamin} are also sufficient assuming Conjecture \ref{Arakawa}.
\end{itemize}
\subsection{$\g=spo(2|2r+1), r>1$}
In this case
$$\g^\natural=so(2r+1),\ M_1(k)=-2k-1,
$$
$$
P^+_k=\{\nu=\sum_i m_i\e_i,\,m_i\in\half +\ZZ\text{ or }m_i\in\ZZ,\ m_1\geq\ldots\geq m_{r-1}\geq m_r\ge0,\ m_1+m_2\le M_1(k)\}.
$$
and
\begin{align*}
A(k,\nu)&=\frac{-8\left( \sum_{i=1}^{r}(2 (r-i+1)m_i +m_i^2\right)-8 k^2+(4r-14) k +2 r-5}{32 (k+3/2-r)}\\
&=-\frac{\left( \sum_{i=1}^{r}(2 (r-i)+1)m_i +m_i^2\right)+p(k)}{4(k+h^\vee)}.
\end{align*}
The Ramond extremal weights are the weights $\nu$ such that $m_r=0$.
 \begin{itemize}
 \item If $\nu$ is Ramond extremal, the necessary conditions for unitarity are $M_1(k)\in\Z_+$ and $\ell=A(k,\nu)$.
 \item If $\nu$ is not Ramond extremal, the necessary and sufficient conditions for unitarity are $M_1(k)\in\Z_+$ and $
\ell\ge A(k,\nu)$.
\end{itemize}
\subsection{$\g=D(2,1;\frac{m}{n})$, $m,n\in\mathbb N$,  $m,n$ {\rm coprime}}
In this case
$$\g^\natural=\g^\natural_1\oplus \g^\natural_2,\,\g^\natural_i\simeq sl(2),\  
 M_1(k)=-\tfrac{m+n}{n}k-1,\  M_2(k)=-\tfrac{m+n}{m}k-1,
 $$
 $$
 P^+_k=\{\nu=\tfrac{r_1}{2}\theta_1+\tfrac{r_2}{2}
 \theta_2\mid r_i\in\ZZ_+, r_i\le M_i(k)\},
$$
and 
\begin{align*}
A(k,\nu)&=-\frac{(\tfrac{m}{n}+1)^2 k (k+1)+\tfrac{m}{n} \left((r_1+r_2+1)^2\right)}{4 (\tfrac{m}{n}+1)^2
   k}\\
      &=-\frac{\tfrac{m n}{(m+n)^2}((r_1+r_2)^2+2(r_1+r_2))+p(k)}{4
  ( k+h^\vee)}.
\end{align*}
If $\eta_{\min}=\e_2-\e_3$,
the Ramond extremal weights are the weights $\nu$ such that $r_1=0$ or $r_2=M_2(k)$.
\begin{itemize}
 \item If $\nu$ is Ramond extremal, the necessary conditions for unitarity are $$(M_1(k),M_2(k))\in\ZZ_+\times \ZZ_+\text{ and }\ell=A(k,\nu).$$
\item  If $\nu$ is not Ramond extremal, the necessary conditions for unitarity are $$(M_1(k),M_2(k))\in\ZZ_+\times \ZZ_+$$ and
\begin{align}\label{2D21a}
\ell\ge A(k,\nu).
\end{align}
The sufficient conditions are
$$
\ell\ge B(k,\nu,\rho_R)= -\frac{k+1}{4}+\frac{m(r_2+1)^2+nr_1^2}{4 (m+n) k}.
$$
The conditions \eqref{2D21a} are also sufficient assuming Conjecture \ref{Arakawa}.
\end{itemize}

If $\eta_{\min}=-\e_2+\e_3$, the Ramond extremal weights are the weights $\nu$ such that $r_1=M_1(k)$ or $r_2=0$.
\begin{itemize}
\item If $\nu$ is Ramond extremal, the necessary conditions for unitarity are $$(M_1(k),M_2(k))\in\ZZ_+\times \ZZ_+\text{ and }\ell=A(k,\nu).$$
\item If $\nu$ is not Ramond extremal, the necessary conditions for unitarity are $$(M_1(k),M_2(k))\in\ZZ_+\times \ZZ_+$$ and
\begin{align}\label{2D21b}
\ell\ge A(k,\nu).
\end{align}
The sufficient conditions are
$$
\ell\ge B(k,\nu;\rho_R)= -\frac{k+1}{4}+\frac{mr_2^2+n(r_1+1)^2}{4 (m+n) k}.
$$
The conditions \eqref{2D21b} are also sufficient assuming Conjecture \ref{Arakawa}.
\end{itemize}
\subsection{$\g=F(4)$} In this case
$$\g^\natural=so(7),\  
 M_1(k)=-\tfrac{3}{2}k-1,
 $$
 $$
 P^+_k=\{\nu=r_1\epsilon_1+r_2\epsilon_2+r_3\epsilon_3,\ r_1\ge r_2\ge r_3\ge0,\   r_i\in\half+\Z\text{ or  } r_i\in\Z,\ r_1+r_2\le M_1(k) \},
$$
and
\begin{align*}
A(k,\nu)&=-\frac{9 k^2+8 r_1^2+8 r_1 (r_2+r_3+5)+8 r_2^2-8 r_2
   r_3+32 r_2+8 r_3^2+8 r_3-4}{36 (k-2)}\\
      &=-\frac{\frac{8}{9} \left(r_1^2+r_1 (r_2+r_3+5)+r_2^2-r_2
   r_3+4 r_2+r_3^2+r_3\right)+p(k)}{4
  ( k+h^\vee)}.
\end{align*}
If $\eta_{\min}=\half(-\e_1+\e_2+\e_3)$, 
the Ramond extremal weights are the weights $\nu$ such that $r_3=0$.
\begin{itemize}
\item If $\nu$ is Ramond extremal, the necessary conditions for unitarity are  $M_1(k)\in\ZZ_+$ and $\ell=A(k,\nu)$.
\item If $\nu$ is not Ramond extremal, the necessary conditions for unitarity are  $M_1(k)\in\ZZ_+$ and
\begin{align}\label{1F4b}
\ell\ge A(k,\nu).
\end{align}
The sufficient conditions are
$$
\ell\ge B(k,\nu,\rho_R)= -\frac{3 k^2+4 \left(r_1^2+4 r_1+r_2^2+2
   r_2+r_3^2\right)}{12 (k-2)}.
$$
The conditions \eqref{1F4b} are also sufficient assuming Conjecture \ref{Arakawa}.
\end{itemize}

If $\eta_{\min}=\half(\e_1-\e_2-\e_3)$, 
the Ramond extremal weights are the weights $\nu$ such that $r_1=r_2$.
\begin{itemize}
\item If $\nu$ is Ramond extremal, the necessary conditions for unitarity are  $M_1(k)\in\ZZ_+$ and $\ell=A(k,\nu)$.
\item If $\nu$ is not Ramond extremal, the necessary conditions for unitarity are  $M_1(k)\in\ZZ_+$ and
\begin{align}\label{2F4b}
\ell\ge A(k,\nu).
\end{align}
The sufficient conditons are
$$
\ell\ge B(k,\nu,\rho_R)= -\frac{3 k^2+4 r_1^2+12 r_1+4 r_2^2+12 r_2+4 r_3^2+4
   r_3-1}{12 (k-2)}.
$$
The conditions \eqref{2F4b} are also sufficient assuming Conjecture \ref{Arakawa}.
\end{itemize}
\subsection{$\g=G(3)$}\label{fff}
 In this case
$$\g^\natural=G_2,\  
 M_1(k)=-\tfrac{4}{3}k-1,
 $$
 $$
 P^+_k=\{\nu=r_1\e_1+r_2\e_2,\, 2r_1\ge r_2\ge r_1,\,r_i\in\ZZ_+,\ r_1+r_2\le M_1(k) \},
$$
and 
\begin{align*}
A(k,\nu)&=\frac{8 k^2+2 k+8 r_1^2-8 r_1 r_2+8 r_2^2+24 r_2-3}{48-32
   k}\\
      &=-\frac{ r_1^2- r_1 r_2+ r_2^2+3 r_2+p(k)}{4
  ( k+h^\vee)}.
\end{align*}
The Ramond extremal are the weights $\nu$ such that $2r_1=r_2$.
\begin{itemize}
\item If $\nu$ is Ramond extremal, the necessary conditions for unitarity are  $M_1(k)\in\ZZ_+$ and $\ell=A(k,\nu)$.
\item If $\nu$ is not Ramond extremal, the necessary and sufficient conditions for unitarity are  $M_1(k)\in\ZZ_+$ and
$
\ell\ge A(k,\nu)$.

\end{itemize}

\section{Unitarity for Ramond extremal modules of the $N=3$ and $N=4$  superconformal algebras}

\subsection{$N=3$} Let $R$ be the Lie conformal superalgebra with basis 
 $$
 \{\tilde L, \tilde G^\pm,\tilde G^0, J^\pm, J^0,\Phi , K\}
 $$
 and commutation relations given in \cite[\S\! 8.5]{KW1}. The $N=3$ superconformal vertex algebra is 
 $$
 \mathcal W_{N=3}^k=V(R)/(K-(k+\tfrac{1}{2})\vac).
 $$   
 Recall that 
 there is a conformal vertex algebra isomorphism 
\begin{equation}\label{WN=3embeddingg}
\mathcal W_{N=3}^k\to W^k_{\min}(spo(2|3))\otimes F_\Phi, 
\end{equation} 
where $F_\Phi$ is the fermionic vertex algebra $F(\C\Phi)$ constructed in Example \ref{12}. 

By Lemma \ref{Fsigma}, $F_\Phi$ admits a Ramond twisted unitary  module $F_\Phi^{\,\tw}$ generated by $1$.

It follows that, if $M$ is Ramond twisted module for $W^k_{\min}(spo(2|3))$ then $M\otimes F_\Phi^{\,\tw}$ admits a $W^k_{\min}(spo(2|3))\otimes F_\Phi$-invariant form. Since the isomorphism in \eqref{WN=3embeddingg} is conformal, a Hermitian form that is invariant for $W^k_{\min}(spo(2|3))\otimes F_\Phi$ is also invariant for $\mathcal W_{N=3}^k$. 

As explained in Section \ref{Twist}, the Ramond twisted modules for $V(R)$ are the same as the restricted $Lie(R,\si_R)$-modules, hence a Ramond twisted $\mathcal W_{N=3}^k$-module $M$ is the same as a  restricted $Lie(R,\si_R)$-modules such that $K$ acts by $(k+\tfrac{1}{2})I_M$. In particular, if $M$  (resp. $M'$)  are Ramond twisted modules 
 for $\mathcal W_{N=3}^k$ (resp. $\mathcal W_{N=3}^{k'}$), then $M\otimes M'$ is a Ramond twisted $\mathcal W_{N=3}^{k+k'+\frac{1}{2}}$--module. Clearly, if both $M,M'$ are unitary, then 
 $M\otimes M'$ is unitary. 
 
 \begin{prop} Let $M_1=-4k-2\in \mathbb N$.  Then the Ramond extremal $W^k_{\min}(spo(2|3))$--modules $L^W(0,\tfrac{M_1-1}{16})$, $L^W(\tfrac{M_1}{2}\theta_1,\tfrac{M_1-1}{16}+\tfrac{M_1}{4})$ are both unitary.
 \end{prop}
 \begin{proof} To make the argument  more transparent we make explicit the dependence on $k$, so we write $L(k,\nu,\ell_0)$ for the $W^k_{\min}(spo(2|3))$--module $L^W(\nu,\ell_0)$. \par
 We proceed by induction on $M_1$. The base case $M_1=1$ corresponds to the collapsing level $k=-3/4$, when $W_{-3/4}^{\min}(spo(2|3))=V_1(sl_2)$. Recall that  $V_1(sl_2)$ has only two irreducible modules $N_1$ and $N_2$, which are both unitary and have highest weights $\nu=0$ and $\nu=\tfrac{\theta_1}{2}$.
  The necessary condition for unitarity (given explicitly in \S\! \ref{spo23}) imply that $N_1=L(-3/4,0,0)$ and $N_2=L(-3/4,\tfrac{\theta_1}{2},1/4)$. 
  
  Assume now $M_1>1$, $k=-\tfrac{M_1+2}{4}$, and set $k_1=-\tfrac{M_1+1}{4}$. Assume by induction that 
 $L(k_1,0,\tfrac{M_1-2}{16})$ and  $L(k_1,\tfrac{M_1-1}{2}\theta_1,\tfrac{M_1-2}{16}+\tfrac{M_1-1}{4})$ are unitary. Then  $M=L(k_1,0,\tfrac{M_1-2}{16})\otimes F^{\,\tw}_\Phi$ is unitary for 
 $\mathcal W_{N=3}^{k_1}$ and $M'=L(-3/4,0,0)\otimes F^{\,\tw}_\Phi$ is unitary for $\mathcal W_{N=3}^{-3/4}$. Therefore  $M\otimes M'$ is unitary for 
 $\mathcal W_{N=3}^{k_2},\,k_2=k_1-\tfrac{3}{4}+\tfrac{1}{2}=-\tfrac{M_1+1}{4}-\tfrac{3}{4}+\tfrac{1}{2}=-\tfrac{M_1}{4}-\tfrac{1}{2}=k$, hence $M\otimes M'$ is a unitary Ramond twisted $W^k_{\min}(spo(2|3))$-module. In particular, the
 $W^{k}_{\min}(spo(2|3))$--module generated by $v_{0,\tfrac{M_1-2}{16}}\otimes 1\otimes v_{0,0}\otimes 1$ is a unitary highest weight module $L(k,\nu,\ell_0)$. Clearly $\nu=0$ and,  by the necessary conditions of \S\! \ref{spo23}, $\ell_0=\frac{M_1-1}{16}$, as required. 

 Repeating the same argument with  $M=L(k_1,\tfrac{M_1-1}{2}\theta_1,\tfrac{M_1-2}{16}+\tfrac{M_1-1}{4})\otimes F^{\,\tw}_\Phi$  and $M'=L(-3/4,\tfrac{\theta_1}{2},\tfrac{1}{4})\otimes F^{\,\tw}_\Phi$ we prove the unitarity of 
 $L(k,\tfrac{M_1}{2}\theta_1,\tfrac{M_1-1}{16}+\tfrac{M_1}{4})$.
 \end{proof}
 Our results match \cite[(2.3.ii)]{M}.
\subsection{$N=4$}  In this subsection we  recover   results of  Eguchi-Taormina (cf. \cite[(5),(6)]{ET3}) using their free field realization.
The $N=4$ superconformal algebra is $W^k_{\min}(psl(2|2))$. 
  We choose  strong generators $J^0,J^\pm,G^\pm,$ $\bar G^\pm,L$ for $W^{k}_{\min}(psl(2|2))$ as in \cite[\S\! 8.4]{KW1}. 
 The   $\l$--brackets among these generators are linear. 
 It is therefore enough to prove unitarity of the Ramond extremal module  $L^W(0,1/4)$ at level $k=-2$ (see \S\! \ref{1111}). Arguing as in the $N=3$ case, the  Ramond extremal modules 
 at level $k<-2$ are obtained by iterated tensor product of $L^W(0,1/4)$.
 
 The unitarity of $L^W(0,1/4)$ is proved by constructing this module as a submodule of a manifestly unitary module.
 This is achieved by using the free field realization $FFR:W^{-2}_{\min}(psl(2|2))\to\mathcal F$, given in  \cite{ET1} (see also \S\! 13.2 of \cite{KMP1}), where 
 $\mathcal F= V^1(\C^4)\otimes F(\C^{4}_{\bar1})$. Here $\C^4$ is viewed as the four-dimensional abelian Lie algebra and $\C^{4}_{\bar1}$ is the four-dimensional totally odd space.

 According to \cite[\S\! 5.2]{KMP} and Lemma \ref{Fsigma} above,  $\mathcal F^{\,\tw}=V^1(\C^4)\otimes F(\C^{4}_{\bar1},\si_R)$ is a unitary Ramond twisted $\mathcal F$-module.  Since $FFR$ is conformal and preserves the $\Z_2$-gradation, $\mathcal F^{\,\tw}$ is also a unitary Ramond twisted module for $W^k_{\min}(psl(2|2))$. It is clear that  the $W^k_{\min}(psl(2|2))$-submodule of $\mathcal F^{\,\tw}$ generated by $\vac\otimes1$ is a unitary highest weight representation $L^W(0,\ell_0)$ of $W^{-2}_{\min}(psl(2|2))$. By the necessary conditions for unitarity $\ell_0=\frac{1}{4}$ and we are done.
\section{The characters of massless Ramond twisted modules for minimal $W$-algebras}\label{massless}

\begin{definition}
We say that a $\si_R$-twisted irreducible highest weight $\Wu$-module  $L^W(\nu,\ell)$ is {\it massless} if there exists $s\in \C$ such that $\ell=\ell(s)$  with 
$L(\widehat\nu_s)$ an atypical representation of $\ga^{\tw}$, where $\widehat\nu_s$ is defined by \eqref{nuhats}.
\end{definition}

It can be easily proved that this definition yields the representations called  massless  in \cite{ET3} and \cite{M} for $psl(2|2)$ and $spo(2|3)$, respectively.

\begin{remark}\label{sindependent}
Recall that, by Lemma \ref{differ}, $\ell(s)=\ell(s')$ if and only if $s'=k+1+\e(\s_R)-s$. Using this relation, it is easy to check that, if $\eta\in \overline\D_{1/2}$,  we have
$$
(\widehat\nu_s+\rhat^{\,\tw}|\frac{p}{2}\d+\frac{\theta}{2}+\eta)=(\widehat\nu_{s'}+\rhat^{\,\tw}|-\frac{p+2}{2}\d+\frac{\theta}{2}-\eta).
$$
In particular $\widehat \nu_s$ is atypical if and only if $\widehat \nu_{s'}$ is atypical.
\end{remark}

Define 
$$
\Pi^\nu_{\bar 1}=\begin{cases}\{-\half\d+\theta/2+\eta_{\min}\}\quad&\text{if $\g\ne spo(2|3)$, $psl(2|2)$,}\\
\{-\frac{\d}{2}+\d_1+\e_1\}\quad&\text{if $\g=spo(2|3)$ and $\nu=0$,}\\
\{\frac{\d}{2}
   +\d_1-\e_1\}\quad&\text{if $\g=spo(2|3)$ and $\nu=\tfrac{M_1(k)}{2}\e_1$,}
   \\
\{-\frac{\d}{2}+\d_1-\e_2,-\frac{\d}{2}+\e_1-\d_2\}\quad&\text{if $\g=psl(2|2)$.}
\end{cases}
$$
Tote that if $\g\ne spo(2|3)$, then $\Pi^\nu_{\bar 1}$ is  the set of odd isotropic simple roots in $\Da^{\tw}_+$.

If $\theta/2$ is not a root of $\g$, set $s_0=\tfrac{k+1}{2}-(\nu-\rho_R+\rho^\natural|\eta_{\min})$.   If $\theta/2$ is a root of $\g$, set $s_0=\tfrac{2k+1}{4}$. In all cases $\ell(s_0)=A(k,\nu)$.

\begin{prop}\label{82} 1)  If $L^W(\nu,\ell)$ is massless, then $\ell=A(k,\nu)$. If $\theta/2$ is not a root of $\g$, then  the converse holds.

2) If $\theta/2$ is  a root of $\g$, then $L^W(\nu,\ell)$  is massless if and only if $\ell=A(k,\nu)$ and $\nu$ is Ramond extremal.

3) In all cases $\Pi_{\bar 1}^\nu$ is the set of simple isotropic roots orthogonal to $\widehat\nu_{s_0}+\widehat\rho^{\,\tw}$.
 \end{prop}
\begin{proof} If $\ell>A(k,\nu)$, then $\ell=\ell(s)$ with either $s=\tfrac{k+1}{2}+\tfrac{\e(\si_R)}{4}+\sqrt{-1}t,\,t\in\R$, or $s\in\R$. In the former case, the claim is obvious if  $t\ne 0$. If $t=0$ and $\e(\si_R)=1$ or $\e(\s_R)=0$ and $(\nu-\rho_{R}+\rho^\natural|\eta_{\min})=0$, then $\ell(s)=A(k,\nu)$. It remains only to check the case where $s\in\R$, $\e(\si_R)=0$, $(\nu-\rho_{R}+\rho^\natural|\eta_{\min})\ne0$, and 
$$
|s-\frac{k+1}{2}|<|(\nu-\rho_{R}+\rho^\natural|\eta_{\min})|.
$$
By Remark \ref{sindependent}, we can assume that $s$ belongs to \eqref{hyp}. With this assumption, we have shown already  in Lemma \ref{hc} that $(\widehat \nu_s+\widehat\rho^{\tw}|\a)\ne 0$ for all odd isotropic roots $\a$. 
\par
Assume now that $\theta/2$ is not a root and $\ell=A(k,\nu)=\ell(s_0)$.
In this case
\begin{align}&(\widehat \nu_{s_0}+\widehat\rho^{\tw}|-\half \d+\theta/2+\eta_{\min})\\&=\notag
((k+h^\vee)\L_0+(\tfrac{k+1}{2} - (\nu-\rho_{R}+\rho^\natural|\eta_{\min}))\theta+\nu+\rho-\rho_R|-\half \d+\theta/2+\eta_{\min})
\\&=-\half (k+h^\vee)+\tfrac{k+1}{2}-(\nu-\rho_{R}+\rho^\natural|\eta_{\min})+\half(h^\vee-1)+(\nu-\rho_{R}+\rho^\natural|\eta_{\min})=0.\notag
\end{align}
Hence $L^W(\nu,\ell)$ is massless. This proves 1) and the fact that the  simple isotropic root orthogonal to $\widehat\nu_{s_0}+\rho^{\,\tw}$ is precisely the root in $\Pi^{\nu}_{\bar 1}$. Statement 2) is proved via a case-wise analysis. Consider the case $\g=G(3)$.  The set of  positive odd isotropic roots 
of $\ga^{\tw}$ is $$\{\gamma_i(p)\mid 1\leq i\leq 9,\,p \text{ odd, $p\geq -1$ for $i=1,2,3,\, p\geq 1$ otherwise}\},$$ where $\gamma_i(p)$ are displayed in Table \ref{tableG(3)}. 
\par\noindent
We prove that if  there exist $p$ and $i$ such that  $(\widehat \nu_{\tfrac{2k+1}{4}}+\widehat\rho^{\tw}|\gamma_i(p))=0$, then $\nu-\rho_R\notin P^+$. Recall that $\nu=a\e_1+b\e_2$ with $a,b\in \ZZ_+, 2a\ge b\ge a, b\leq m:=M_1(k)$.  The condition on $a,b$ implied by $(\widehat \nu_{s_0}+\widehat\rho^{\tw}|\gamma_i(p))=0$ is listed in Table \ref{tableG(3)}. 
\begin{center}

\begin{tabular}{   c | c| c}
 & $\gamma_i(p)$ & $(\widehat \nu_{s_0}+\widehat\rho^{\tw}|\gamma_i(p))=0$\\\hline
1&$\tfrac{p}{2}\d + \d_1 + \e_1$ &$ -4 a+2 b=3 (3+m) (1+p)$\\\hline
2&$\tfrac{p}{2}\d + \d_1 + \e_2$ &$2 a-4 b=9+3 (3+m) (1+p)$\\\hline
3&$\tfrac{p}{2}\d + \d_1 + \e_1 + \e_2$ &$2a+2b=9+3 (3+m) (1+p)$\\\hline
4&$\tfrac{p}{2}\d - \d_1 + \e_1$ &$4a-2b=3 (3+m) (-1+p)$\\\hline
5&$\tfrac{p}{2}\d - \d_1 + \e_2$ &$2 a-4 b=6 +3(3+m)(p-1)$\\\hline
6&$\tfrac{p}{2}\d - \d_1 + \e_1 + \e_2$ &$2a+2b=-6 - 3 (p - 1) (m + 3)$\\\hline
7&$\tfrac{p}{2}\d + \d_1 - \e_1$ &$4 a-2 b=3 (3+m) (1+p)$\\\hline
8&$ \tfrac{p}{2}\d + \d_1 - \e_2$ &$-2 a+4 b=-6 + 3 (3 + m) (p + 1)$\\\hline
9&$\tfrac{p}{2}\d + \d_1 - \e_1 - \e_2$ &$2 a+2 b=-6 + 3 (3 + m) (p + 1)$\\\hline
\end{tabular}\vskip8pt
 \captionof{table}{G(3)\label{tableG(3)}}
\end{center}

Case 1. Since $-4 a+2 b\leq 0$, and the r.h.s is non-negative, the only possibility is $p=-1$, which forces $b=2a$ and the only simple isotropic root orthogonal to $\widehat\nu_s+\widehat\rho^{\,\tw}$ is precisely the root in  
 $\Pi^{\nu}_{\bar 1}$. In this case $\nu-\rho_R= (a-1)\e_1+(2a-1)\e_2\notin P^+$.

Cases 2, 5. Since $2a-4b\le 0$, equality cannot hold.

Case 3. If $p\geq 1$, then equality cannot occur since $2a+2b\leq 4m$. If $p=-1$, the equality becomes $2a+2b=9$, which is impossible.

Case 4. If $p=1$ we are back to Case 1. If $p\geq 3$, equality cannot occur since $4a-2b\leq 4m$ wheres the r.h.s is greater than $6m$.

Case 6. The left hand side is non-negative and  right hand side is negative.

Case 7. Similar to Case 4.

Cases 8,\,9. In both cases the left hand side is less or equal than $4m$, whereas $-6 + 3 (3 + m) (p + 1)\geq 6m+3$, hence equality cannot hold.

Let now $\g=spo(2|2r+1)$.  The set of  odd isotropic roots 
of $\ga^{\tw}$ is $\{p\d/2\pm\e_1\pm\d_1\mid 1\leq i\leq r,\,p\text{ odd integer}\}$. Set $m:=M_1(k)$. Recall that $\nu=\sum_{j=1}^ra_i\e_i$ with $a_1\geq a_2\geq \ldots \ge a_r\ge 0$ and $a_i\in \ZZ_+$ for all $i$ or $a_i\in \half+\ZZ_+$ for all $i$, and finally $a_1+a_2\leq m$.  Relation
$(\widehat \nu_{s_0}+\widehat\rho^{\tw}|\gamma_i(p))=0$ implies
\begin{equation*}\frac{p\pm 1}{2}(-m+2-2r)=\pm (a_i+r-i),\end{equation*}
which in turn  implies
\begin{equation}\label{qwert}\left|\frac{p\pm 1}{2}\right|(m+r+r-2)= (a_i+r-i)\leq m+r-1.\end{equation}
If $r>1$, \eqref{qwert} implies $p=\pm 1$, $i=r$ and $a_r=0$ and the only simple isotropic root orthogonal to $\widehat\nu_{s_0}+\widehat\rho^{\,\tw}$ is precisely the root in  
 $\Pi^{\nu}_{\bar 1}$.\par
Finally consider the case $r=1$, i.e.   $\g=spo(2|3)$. Then, if $\nu=\tfrac{a}{2}\e_1$, we have
\begin{equation}\label{equationn}(\widehat\nu_{s_0}+\widehat\rho^{\tw}|\pm \e_1\pm \d_1 + \tfrac{p}{2}\d)=
\tfrac{1}{8}(\mp2a+m(\mp 1-p)).\end{equation}
If \eqref{equationn} vanishes, then $m$ divides $a$, and since $0\leq a \leq m$, we have either $a=0$ or $a=m$. In the former case the $\nu-\rho_R\notin P^+$, in the latter 
$\nu-\rho_R$ is extremal. In both cases the only simple isotropic root orthogonal to $\widehat\nu_{s_0}+\widehat\rho^{\,\tw}$ is precisely the root in  
 $\Pi^{\nu}_{\bar 1}$. 
\end{proof}


\begin{theorem}\label{charactersmassless}Assume Conjecture \ref{Arakawa}.
Let $k$ be in the unitary range, $\nu\in P^+_k$, and assume that $\widehat\nu_{s_0}$ is non-degenerate.
Then
\begin{align}\label{ff2}
&(1+\e(\si_R))\widehat F^{R}ch\, L^W(\nu,A(k,\nu))=\\
&\notag q^{\tfrac{(\widehat\nu_{s_0}|\widehat\nu_{s_0}+2\rhat^{\,\tw})}{2(k+h^\vee)}+a(k)}e^{-\rho_R}\sum_{w\in\Wa^\natural}det(w)\frac{q^{(w(\widehat\nu_{s_0}+\rhat^{\,\tw})- \rhat^{\,\tw})(x+D)}e^{(w(\widehat\nu_{s_0}+\rhat^{\,\tw})- \rhat^{\,\tw})_{|\h^\natural}}}{\prod_{\be\in\Pi_{\bar 1}^\nu}(1+q^{w(\be)(x+D)}e^{-w(\be)_{|\h^\natural}})},
\end{align}
where $a(k)$  is given by \eqref{aofk}.
\end{theorem}
\begin{proof} Combining Lemmas \ref{notextremal} and  \ref{isextremal} with Proposition \ref{82}, we have proved  that  the hypothesis of Proposition \ref{atypicalnonzero} are satisfied.
  Formula \eqref{ff2}  follows  form Proposition \ref{atypicalnonzero} using Conjectures  \ref{Arakawa} in the same way as in the proof of Theorem \ref{characters}.
\end{proof}
\begin{remark}\label{125} In the NS sector a formula similar to \eqref{ff2} holds. More precisely
\begin{align}\label{ff2NS}
&\widehat F^{NS}ch\, L^W(\nu,A(k,\nu))= q^{\tfrac{(\widehat\nu_{t_0}|\widehat\nu_{t_0}+2\rhat)}{2(k+h^\vee)}}\sum_{w\in\Wa^\natural}det(w)\frac{q^{(w(\widehat\nu_{t_0}+\rhat)- \rhat)(x+D)}e^{(w(\widehat\nu_{t_0}+\rhat)- \rhat)_{|\h^\natural}}}{\prod_{\be\in\Pi_{\bar 1}}\left(1+q^{w(\be)(x+D)}e^{-w(\be)_{|\h^\natural}}\right)},
\end{align}
where $t_0$ is either $(\nu|\xi)$ or $k+1-(\nu|\xi)$ (we have shown in \cite{KMP1} that at least one of the two values yields a non-degenerate $\nu_{t_0}$). This is essentially formula (14.6) from \cite{KMP1}. Note that if $\nu=0$ then $(k\L_0+\rhat|\d-\theta)=k+1$, which is never a positive integer, hence we can choose $t_0=0$.
\end{remark}

\section{Denominator identities}\label{denominator}Let $k_0$ be non-critical and such that $W^{\min}_{k_0}(\g)=\C\vac$.
Since, as shown in  \cite{AKMPP}, this happens if and only if $M_i(k_0)=0$ for all $i$, it follows from Table \ref{numerical} that this happens in the following cases, where $v=u_i$ from Table \ref{numerical}.
\begin{table}[h]
\begin{tabular}{c | c| c |c | c }
$\g$&
$psl(2|2)$&
$spo(2|m),\ m\ne3$&
$F(4)$&
$G(3)$\\
\hline
$k_0$&$-1$&$-\half$&$-\tfrac{2}{3}$&$-\tfrac{3}{4}$\\
\hline
$u$&$-2$&$-1$&$-\tfrac{4}{3}$&$-\tfrac{3}{2}$\\
\hline
$h^\vee$&$0$&$2-\tfrac{m}{2}$&$-2$&$-\tfrac{3}{2}$\\
\hline
$\bar h^\vee$&$-2$&$1-\tfrac{m}{2}$&$-\tfrac{10}{3}$&$-3$\\
\hline
$b=\tfrac{h^\vee+\bar h^\vee}{u}$&$1$&$m-3$&$4$&$3$
\end{tabular}
\vspace{5pt}
	\captionof{table}{\label{Table2}}
\end{table}

 Recall from \cite[\S\,6]{VB} the decomposition \begin{equation}\label{WT}\widehat W^\natural = W^\natural\ltimes T^\natural,\end{equation} where $T^{\natural}=\{t_\a\mid \a \in M^\natural\}$,   $M^\natural$ is the $\ZZ$-span of the long roots of $\g^\natural$,  and 
\begin{equation}\label{trasl}t_\a(\l)=\l+\tfrac{2}{v}\l(K)\a-\tfrac{2}{v}((\l|\a)+\tfrac{1}{v}(\a|\a)\l(K))\d\end{equation}
(cf. \cite[(6.5.3)]{VB}, where this formula is given using the normalized invariant bilinear form).

Let $b=\frac{h^\vee+\bar h^\vee}{u}$ (see Table \ref{Table2} for its values).
\begin{theorem}\label{denidentites}  We have for all $\g$ listed in Table \ref{Table2}, except for $\g=spo(2|N)$, $0\le N\le 3$:
\begin{equation}\label{detNS} 
 \widehat F^{NS}=e^{-\rho^\natural}\sum\limits_{\bar w\in W^\natural}\sum_{\a\in M^\natural}\det(\bar w)\frac{e^{\bar w(\rho^{\natural}+b\a)}}{\prod_{\beta\in \Pi_{\bar 1}}(1+e^{-\bar w(\beta_{|\h^\natural})}q^{-\frac{2}{u}(\a|\beta)+\frac{1}{2}})}q^{\tfrac{b}{v}(\a|\a)+\tfrac{2}{v}(\rho^\natural|\a)}.\end{equation}

If Conjecture \ref{Arakawa}  holds, then 
\begin{equation}\label{detR} 
 \widehat F^{R}=\frac{e^{\rho_R-\rho^\natural}}{1+\e(\s_R)}\sum\limits_{\bar w\in W^\natural}\sum_{\a\in M^\natural}\det(\bar w)\frac{e^{\bar w(\rho^{\natural}-\rho_R+b\a)}}{\prod_{\beta\in \Pi^\nu_{\bar 1}}(1+e^{-\bar w(\beta_{|\h^\natural})}q^{-\frac{2}{u}(\a|\beta)})}q^{{\tfrac{b}{v}(\a|\a)+\tfrac{2}{v}(\rho^\natural-\rho_R|\a)}}.\end{equation}

\end{theorem}
\begin{proof}We apply Remark \ref{125} to the (untwisted) $\ga$-module $L(k_0\L_0)\simeq V_{k_0}(\g)$ and $t_0=0$. 
Since $H_0(L(k_0\L_0))=W^{\min}_{k_0}(\g)=\C\vac$, we obtain
\begin{equation}\label{detNSSS}
\widehat F^{NS}= \sum_{w\in\Wa^\natural}det(w)\frac{q^{(w(k_0\L_0+\rhat)- \rhat)(x+D)}e^{(w(k_0\L_0+\rhat)- \rhat)_{|\h^\natural}}}{\prod_{\be\in\Pi_{\bar 1}}\left(1+q^{w(\be)(x+D)}e^{-w(\be)_{|\h^\natural}}\right)}.
\end{equation} 
To compute the R.H.S. of \eqref{detNSSS}, write $w=\bar w t_\a$, given by  the decomposition \eqref{WT} and formula \eqref{trasl}, to obtain
\begin{align*}&w(k_0\L_0+\rhat)-\rhat=\bar wt_\a((k_0+h^\vee)\L_0+\rho)-\rhat\\
&=k_0\L_0+\bar w(\rho)-\rho+\tfrac{2}{v}(k_0+h^\vee)\bar w(\a)-\tfrac{2}{v}(\a|\rho+\tfrac{1}{v}(k_0+h^\vee)\a)\\
&=k_0\L_0+\bar w(\rho)-\rho+b\bar w(\a)-\tfrac{2}{v}(\a|\rho+\tfrac{b}{2}\a)\d,\end{align*}
so that, since $(\a|\rho)=(\a|\rho^\natural)$, we have
$$
q^{(w(k_0\L_0+\rhat)- \rhat)(x+D)}e^{(w(k_0\L_0+\rhat)- \rhat)_{|\h^\natural}}=e^{\bar w(\rho^{\natural}+b\a)-\rho^\natural}q^{\tfrac{2}{v}(\a|\rho^\natural+\tfrac{b}{2}\a)}.
$$
Similarly
$$
w(\be)=\bar wt_\a(\be)=\bar w(\be)-\tfrac{2}{v}(\be|\a)\d,
$$
so that, since $\bar w(\be)(D+x)=\bar w(\be)(x)=\be(x)=\half$,
$$
q^{w(\be)(x+D)}e^{-w(\be)_{|\h^\natural}}=e^{-\bar w(\be_{|\h^\natural})}q^{-\tfrac{2}{v}(\be|\a)+\tfrac{1}{2}}.
$$
Substituting, we find  \eqref{detNS}.

Now we prove \eqref{detR}. Recall that
$$s_0=\begin{cases}
\frac{k_0+1}{2}+(\rho_R-\rho^\natural|\eta_{\min})&\text{if $\theta/2$ is not a root of $\g$,}\\
\frac{2k_0+1}{4}&\text{if $\theta/2$ is a root of $\g$.}
\end{cases}
$$
By Lemma \ref{Aknusecond},  $\ell(s_0)=A(k_0,0)$. Recall the polynomial $p(k)$, mentioned at the beginning of Section \ref{SectionZhu}. Recall \cite[Theorem 3.3]{AKMPP}  that $p(k_0)=0$. Combining this observation with \eqref{Aknufirst} and \eqref{ev},  we find that $A(k_0,0)=0$, so $\ell(s_0)=0$.  
If $\nu=0$, then  (see \eqref{nuhats}) $\widehat\nu_{s_0}=k_0\L_0+s_0\theta+\rho_R$.
Since
$$
(\widehat\nu_{s_0}+\widehat\rho^{\tw}|\d-\theta)=\begin{cases}
2(\rho_R-\rho^\natural|\eta_{\min})&\text{if $\theta/2$ is not a root of $\g$,}\\
0&\text{if $\theta/2$ is a root of $\g$,}
\end{cases}
$$
we see by a case-wise inspection that $\widehat \nu_{s_0}$ is non-degenerate. 
Assuming Conjecture \ref{Arakawa},  by Proposition \ref{Lnust} and the fact that $\nu=0$ and $\ell(s_0)=0$, we have 
$$
H_0(L(\widehat\nu_{s_0}))=(1+\e(\s_R))L^W(0,0).
$$
Since $W^{\min}_{k_0}(\g)=\C\vac$ and the maximal proper ideal of $W^{k_0}_{\min}(\g)$ is $\s_R$-stable, $\C\vac$ is a one-dimensional $\s_R$-twisted representation of $W^{k_0}_{\min}(\g)$, hence $L^W(0,0)=\C\vac$. Apply now Theorem \ref{charactersmassless} to $L^W(0,0)$. 
As above, we compute:
\begin{align*}
&w(k_0\L_0+s_0\theta+\rho_R+\rhat^{\tw})-\rhat^{\tw}=\bar wt_\a((k_0+h^\vee)\L_0+s_0\theta-\rho_R+\rho)-\rhat^{\tw}=\\
&k_0\L_0+s_0\theta+\bar w(\rho-\rho_R)+2\rho_R-\rho+b\bar w(\a)+\tfrac{2}{v}(\a|\rho-\rho_R+\tfrac{b}{2}\a)\d,
\end{align*}
so
$$
e^{ev(w(k_0\L_0+s_0\theta+\rho_R+\rhat^{\tw})-\rhat^{\tw})}=e^{\bar w(\rho^\natural-\rho_R+b\a)}e^{2\rho_R-\rho^\natural}q^{-s_0}q^{\tfrac{2}{v}(\a|\rho^\natural-\rho_R+\tfrac{b}{2}\a)}.
$$
Similarly
$$
w(\be)=\bar wt_\a(\be)=\bar w(\be)-\tfrac{2}{v}(\be|\a)\d,
$$
so that, since $\bar w(\be)(D+x)=\be(D+x)=0$,
$$
e^{-ev(w(\be)}=e^{-\bar w(\be_{|\h^\natural})}q^{-\tfrac{2}{v}(\be|\a)}.
$$
Substituting in \eqref{ff2} we find \eqref{detR}, since $s_0=\widehat\nu_{s_0}(x+D)$
and
$$
\tfrac{(\widehat\nu_{s_0}|\widehat\nu_{s_0}+2\rhat^{\,\tw})}{2(k+h^\vee)}+a(k)-\widehat\nu_{s_0}(x+D)=\ell(s_0)=0.
$$
\end{proof}
 
 In the subsequent subsections we write down the denominator identities \eqref{detNS} and \eqref{detR} explicitly. To simplify notation we set 
\begin{align}
&\vartheta_{0}(x)=\prod_{j=1}^\infty (1-xq^{j-1})(1-x^{-1}q^{j}),\ \vartheta_{1}(x)=\vartheta_0(-xq^{\frac{1}{2}})=\prod_{j=1}^\infty(1+xq^{j-\frac{1}{2}})(1+x^{-1}q^{j-\frac{1}{2}}),\\
&\varphi(q)=\prod_{j=1}^\infty(1-q^j),\ \varphi_1(q)=\prod_{j=1}^\infty(1+q^{j-\frac{1}{2}}),\ \varphi_2(q)=\frac{\varphi(q^2)}{\varphi(q)}=\prod_{j=1}^\infty(1+q^j).
\end{align}
If $\g=spo(2|2r)$, $\g=D(2,1;a)$, or $\g=F(4)$, the denominator formulas in the Ramond sector depend on the choice of the set $\Pia$ of simple roots for $\ga^{\,\tw}$, which ultimately depends on the choice of $\eta_{\min}$. 
We now explain how to obtain one formula from the other one. Choose $\eta_{\min}$ and write $F^R(\eta_{\min})$ for the corresponding denominator and $\rho_R(\eta_{\min})$ for the corresponding $\rho_R$. We observe that $F^R(-\eta_{\min})=e^{-\eta_{\min}}F^R(\eta_{\min})$ and that $\rho_R(-\eta_{\min})=\rho_R(\eta_{\min})-\eta_{\min}$. It follows that the denominator formula for $-\eta_{\min}$ is obtained from the formula for $\eta_{\min}$ by multiplying both sides by $e^{-\eta_{\min}}$.
In these cases we make only one choice for $\eta_{\min}$ and write down only the corresponding formula.

\subsection{$\g=spo(2|N)$, $N=0,1,2$, $k_0=-\half$} These cases are not covered by \eqref{detNS} and \eqref{detR}.

If $N=0$, $\Wu$ is the universal Virasoro vertex  algebra of central charge $c(k)=\frac{3k}{k+2}-6k-2$, so, since $c(k_0)=0$, $\Ws=H_0(L(-\frac{1}{2}\L_0))=\C$. Next note that $-\frac{1}{2}$ is an admissible level so, by \cite[Theorem 1 and Example 1]{KW88},
\begin{equation}\label{charspo(2|1)}
\widehat R \ ch\, L(-\frac{1}{2}\L_0)=\sum_{w\in\widehat W_{int}}det(w)e^{w(-\frac{1}{2}\L_0+\widehat\rho)-\widehat\rho},
\end{equation}
where $\widehat W_{int}$ is the Weyl group of the set of roots corresponding to the set of simple roots $\widehat \Pi_{int}=\{\theta,2\d-\theta\}$.
Applying the functor $H$, it follows, by Arakawa theorem,  that
\begin{equation}\label{denvir}
\prod_{n\ge1}(1-q^n)=\sum_{w\in\widehat W_{int}}det(w)q^{-(w(-\frac{1}{2}\L_0+\widehat\rho)-\widehat\rho)(x+D)}.
\end{equation}
We argue as in Theorem \ref{denidentites}: write
$w\in\ \in\widehat W_{int}$ as $\bar w t_{n\theta}$ with $\bar w\in \{1,s_\theta\}$ and 
$$
t_{n\theta}(\widehat \L)=\widehat \L+2n\widehat \L(K)\theta-4n(n\widehat \L(K)+\widehat \L(x))\d.
$$
In our special case we obtain
$$
t_{n\theta}(-\half\L_0+\widehat \rho)-\widehat \rho=-\half\L_0+3n\theta-2n(1+3n)\d.
$$
while
$$
s_\theta t_{n\theta}(-\half\L_0+\widehat \rho)-\widehat \rho=-\half\L_0-(3n+1)\theta-2n(1+3n)\d
$$
so \eqref{denvir} becomes
$$
\prod_{n\ge1}(1-q^n)=\sum_{n\in \Z}(q^{6n^2-n}-q^{6n^2+5n+1})=\sum_{m\in \Z}(-1)^mq^{\tfrac{3m^2+m}{2}},
$$
which is the Euler identity for the classical partition function.

Next we discuss $N=1$. In the NS sector we have $\C\vac=W^{\min}_{-1/2}(spo(2|1))=H(L(-\frac{1}{2}\L_0))$.
It follows from \cite[Example 2 and Theorem 1]{KW88}, that $-\half\L_0$ is an admissible weight for $spo(2|1)^{\wedge}$, hence  \eqref{charspo(2|1)}
holds with $\Wa_{int}$ the Weyl group of the root subsystem generated by the set of simple roots 
$$
\Pi_{int}=\{\theta/2,\d-\theta/2\}.
$$
Applying the functor $H$ and using Arakawa theorem, we obtain 
\begin{equation}\label{denNSNS}
\prod_{n\ge1}\frac{1-q^n}{1+q^{n-\frac{1}{2}}}=\sum_{w\in\widehat W_{int}}det(w)q^{-(w(-\frac{1}{2}\L_0+\widehat\rho)-\widehat\rho)(x+D)}.
\end{equation}
Note that the group $\Wa_{int}$ is the same as in the $N=0$ case.  
In our special case we obtain
$$
t_{n\theta}(-\half\L_0+\widehat \rho)-\widehat \rho=-\half\L_0-2n\theta-n(1+4n)\d,
$$
while
$$
s_\theta t_{n\theta}(-\half\L_0+\widehat \rho)-\widehat \rho=-\half\L_0-(2n+\half)\theta-n(1+4n)\d,
$$
so \eqref{denNSNS} becomes
\begin{align*}
&\prod_{n\ge1}\frac{1-q^n}{1+q^{n-\frac{1}{2}}}=\sum_{n\in \Z}(q^{4n^2-n}-q^{4n^2+3n+\half})\\
&=\sum_{n=0}^\infty(q^{4n^2-n}-q^{4n^2+3n+\half})+\sum_{n=1}^\infty(q^{4n^2+n}-q^{4n^2-3n+\half})\\
&=\sum_{m\in4\Z_+-1} q^{\frac{1}{4}m(m+1)}-\sum_{m\in4\Z_++1} q^{\frac{1}{4}m(m+1)}+\sum_{m\in4\nat}q^{\frac{1}{4}m(m+1)}-\sum_{m\in4\nat-2} q^{\frac{1}{4}m(m+1)}\\
&=\sum_{m=0}^\infty (-q^{\tfrac{1}{2}})^{m(m+1)/2}.
\end{align*}
Replacing $q$ by $q^2$ and then changing the sign of $q$, we obtain the Gauss identity for the generating series of triangular numbers:
$$
\prod_{n\ge1}\frac{1-q^{2n}}{1-q^{2n+1}}=\sum_{n=0}^\infty q^{\tfrac{n(n+1)}{2}}.
$$

 In the Ramond sector, using Conjecture \ref{Arakawa}, we have 
\begin{equation}\label{evvv}
 H(L(-\half\L_0))=L^W(0,0)\oplus L^W(0,0)=\C^2.
\end{equation}
The character of the $\ga^{\,\tw}$-module $L(-\half\L_0)$ is given by
$$
ch\, L(-\frac{1}{2}\L_0)=\sum_{w\in\widehat W_{int}}det(w)\frac{e^{w(-\frac{1}{2}\L_0+\rhat^{\,\tw})-\rhat^{\,\tw}}}{\widehat R^{\,\tw}}
$$
  so, applying the twisted quantum Hamiltonian reduction functor, the identity becomes 
 \begin{equation}\label{denNSR}
2\prod_{n\ge 1} \frac{1-q^n}{
 1+q^{n-1}}=\prod_{n\ge 1} \frac{1-q^n}{
 1+q^n}=\sum_{w\in\widehat W_{int}}det(w)q^{-(w(-\frac{1}{2}\L_0+\rhat^{\,\tw})-\rhat^{\\,tw})(x+D)}.
\end{equation} 
The Weyl group $\Wa_{int}$ is again the same as in the $N=0$ case.
We obtain
$$
t_{n\theta}(-\half\L_0+\rhat^{\tw})-\rhat^{\tw}=-\half\L_0 +2n\theta-2n(1+2n)\d,
$$
while
$$
s_\theta t_{n\theta}(-\half\L_0+\rhat^{\tw})-\rhat^{\tw}=-\half\L_0-(2n+1)\theta-2n(1+2n)\d,
$$
so \eqref{denNSR} becomes
$$
\prod_{n\ge1}\frac{1-q^n}{1+q^{n}}=\sum_{n\in \Z}(q^{(2n)^2}-q^{(2n+1)^2})=\sum_{n\in \Z}(-1)^nq^{n^2},
$$
which is Gauss identity for the generating series of square numbers.

\begin{remark}\label{evv} This gives some evidence for Conjecture \ref{Arakawa}, which has been used in a crucial way in \eqref{evvv}.\end{remark}

In the $N=2$ case we have $W^{\min}_{-1/2}(spo(2|2))=\C$.  By \cite[Corollary 11.2.4]{GK2} and the remark thereafter, we can apply formula (35) of [loc. cit.]. Following \cite{GK2}, we choose the set of simple roots for $\g=spo(2|2)$ to be $\Pi=\{\a_1,\a_2\}$ with both simple roots isotropic so that  we can compute the character of $L(-\half\L_0)$ explicitly, using \cite[(14)]{GK2}. By applying the quantum Hamiltonian reduction functor as in Proposition \ref{atypicalnonzero}, we derive the character formula of $H(L(-\half\L_0))$:
$$
ch\,H(L(-\half\L_0))=\frac{1}{\widehat F^{NS}}\sum\limits_{w\in\widehat W_{int}}det(w)\frac{q^{-(w(-\tfrac{1}{2}\L_0+\rhat)-\rhat)(x+D)}e^{(w(-\tfrac{1}{2}\L_0+\rhat)-\rhat)_{|\h^\natural}} }{\prod_{\be\in\Pi_{\bar 1}}(1+q^{(w\be)(x+D)}e^{-w(\be)_{|\h^\natural}})}
$$  
In this case $\Pi_{\bar 1}=\{\a_1\}$ and $\Wa_{int}$ is, once again, the group as in the $N=0$ case.
 
 Arguing as in the proof of Theorem \ref{denidentites}, we find
\begin{align}
&\prod_{n\ge1}\frac{\left(1-q^n\right)^2}{\left(1+e^{-(\a_1)_{|\h^\natural}}q^{n-\frac{1}{2}}\right) \left(1+e^{-(\a_2)_{|\h^\natural}}
   q^{n-\frac{1}{2}}\right)}\notag\\
   &=\sum\limits_{w\in\widehat W_{int}}det(w)\frac{q^{-(w(-\tfrac{1}{2}\L_0+\rhat)-\rhat)(x+D)}e^{(w(-\tfrac{1}{2}\L_0+\rhat)-\rhat)_{|\h^\natural}} }{1+q^{(w\a_1)(x+D)}e^{-w(\a_1)_{|\h^\natural}}}\label{spo2|2NS}.
\end{align}

More explicitly, write $\a_1=\d_1+\e_1$, so that $\a_2=\d_1-\e_1$, and set $z=e^{\e_1}$. Since $\rhat=\L_0$ in this case, we have
$$
t_{n\theta}(-\half\L_0+\widehat \rho)-\widehat \rho=-\half\L_0+n\theta-2n^2\d,\ t_{n\theta}(\a_1)=-2n\d+\a_1,
$$
while
$$
s_\theta t_{n\theta}(-\half\L_0+\widehat \rho)-\widehat \rho=-\half\L_0-n\theta-2n^2\d,\ s_\theta t_{n\theta}(\a_1)=-2n\d-\a_2,
$$
so \eqref{spo2|2NS} becomes
\begin{equation}\label{e1NS}
\frac{\varphi(q)^2}{\vartheta_1(z^{-1})}=\sum\limits_{n\in\Z}\left(\frac{q^{2n^2-n} }{1+z^{-1}q^{-2n+\tfrac{1}{2}}}-\frac{q^{2n^2+n} }{1+z^{-1}q^{-2n-\tfrac{1}{2}}}\right).
\end{equation}
Here and further this is viewed as an identity of formal power series in $q$ with functions in $z$ as coefficients, using the $|q|<1$ expansion of the two series (see Remark \ref{remarkqexpansion}).

In the Ramond sector choose $\eta_{\min}={\e_1}$, so that $s_0=0$ and $\rho_R={\e_1}/2$.  Using Conjecture \ref{Arakawa} b), we have 
 $$
 H(L(-\half\L_0+\half{\e_1}))=L^W(0,0)=\C.
 $$
As in the NS sector, we apply formula (35) of \cite{GK2}, using (14) of [loc. cit.] to compute its  right hand side, and then apply the quantum Hamiltonian reduction functor. This gives the character formula:
\begin{align*}
&ch\, H(L(-\half\L_0+\half{\e_1}))\\
&=\frac{e^{-\rho_R}}{\widehat F^{R}}\sum\limits_{w\in\widehat W_{int}}det(w)\frac{q^{-(w(-\tfrac{1}{2}\L_0+\half{\e_1}+{\rhat^{\,\tw}})-{\rhat^{\,\tw}})(x+D)}e^{(w(-\tfrac{1}{2}\L_0+\half{\e_1}+{\rhat^{\,\tw}})-{\rhat^{\,\tw}})_{|\h^\natural}} }{\prod_{\be\in\Pi_{\bar 1}}(1+q^{(w\be)(x+D)}e^{-w(\be)_{|\h^\natural}})},
\end{align*} 
 where $\Pi_{\bar 1}=\{\beta\}$ with $\beta=-\half\d+\a_1$ and $\Wa_{int}$ is the group as in the $N=0$ case.
The denominator identity becomes
 \begin{align}
&\prod_{n\ge1}\frac{\left(1-q^n\right)^2}{\left(1+z^{-1}q^{n-1}\right) \left(1+z q^n\right)}\notag\\
&=z^{-\frac{1}{2}}\sum\limits_{w\in\widehat W_{int}}det(w)\frac{q^{-(w(-\half\L_0+\e_1/2+\rhat^{\tw})-\rhat^{\tw})(x+D)} e^{(w(-\half\L_0+\e_1/2+\rhat^{\tw})-\rhat^{\tw})_{|\h^\natural}}}{1+q^{(w\beta)(x+D)}e^{-(w\beta)_{|\h^\natural}}}.\label{N=2R}
\end{align} 
 In this case, we have
$$
t_{n\theta}(-\half\L_0+\e/2+\rhat^{\tw})-\rhat^{\tw}=-\half\L_0+\e/2+n\theta-2n^2\d,
\quad t_{n\theta}(\beta)=-(\half+2n)\d+\a_1,
 $$
while
$$
s_\theta t_{n\theta}(-\half\L_0+\widehat \rho)-\widehat \rho=-\half\L_0+\e/2-n\theta-2n^2\d,\quad
s_\theta t_{n\theta}(\beta)=-(\half+2n)\d-\a_2,
$$
so \eqref{N=2R} becomes
\begin{align}\label{N=2Rexplicit}
&\frac{\varphi(q)^2}{\vartheta_0(-z^{-1})}=\sum\limits_{n\in\Z}\left(\frac{q^{2n^2-n} }{1+z^{-1}q^{-2n}}-\frac{q^{2n^2+n} }{1+z^{-1}q^{-2n-1}}\right)
=\sum\limits_{r\in\Z}(-1)^r\frac{q^{\frac{r(r+1)}{2}} }{1+z^{-1}q^{r}}.
\end{align}
This last identity is proven in \cite{KWNT} by specializing a denominator identity for $sl(2|1)^{\wedge }$ (see \cite{KWNT}, formula (4.8)), which is the celebrated Ramanujan identity:
\begin{equation}\label{e1}\frac{\varphi(q)^2\vartheta_0(xy)}{\vartheta_0(-x)\vartheta_0(-y)}
=
\left(\sum_{m,n=0}^\infty-\sum_{m,n=-1}^{-\infty}\right)(-1)^{m+n}x^my^nq^{mn},
\end{equation}
therefore \eqref{N=2Rexplicit} can be seen as another piece of evidence for Conjecture \ref{Arakawa}. Actually \eqref{e1} follows from \eqref{e1NS} by replacing $z$ by $zq^{\frac{1}{2}}$. This is not surprising due to the spectral flow. 
  \subsection{$\g=psl(2|2), k_0=-1$}
  
  In the NS sector, \eqref{detNS}  gives, letting $e^{\d_1}=x$ and $e^{-\d_2}=y$, we obtain
\begin{align}\label{134}
 \frac{\varphi(q)^2\vartheta_0(x^{-1}y^{-1})}{
 \theta_1(x) \theta_1(y)}
 =\sum_{n\in\ZZ}\left(\frac{x^{n}y^{n}}{(1+x q^{n+\frac{1}{2}})(1+y q^{n+\frac{1}{2}})}-\frac{x^{-n-1}y^{-n-1}}{(1+x^{-1}q^{n+\frac{1}{2}})(1+y^{-1}q^{n+\frac{1}{2}})}\right)q^{n^2+n}.
\end{align}
In the Ramond sector, recalling that   
  $\Pi^{\nu}_{\bar1}=\{-\half\d+\d_1-\e_2,-\half\d+\e_1-\d_2\}$, \eqref{detR} gives, letting $e^{\d_1}=x$ and $e^{-\d_2}=y$,
\begin{align}\label{135}
& \frac{ \varphi(q)^2\vartheta_0(x^{-1}y^{-1})}{\vartheta_0(-x^{-1})\vartheta_0(-y^{-1})} =\sum_{n\in\ZZ}\left(\frac{x^{n}y^{n}}{(1+x^{-1}q^{-n})(1+y^{-1}q^{-n})}
-
\frac{x^{-n}y^{-n}}{(1+xq^{-n})(1+yq^{-n})}\right)
q^{n^2}.
\end{align}
Note that this identity follows from \eqref{e1} by replacing $x$ with $x^{-1}$ and $y$ by $y^{-1}$.

\subsection{$\g=spo(2|3)$, $k_0=-\half$}
In the $N=3$ case $k_0$ is critical for $\g$, hence our previous approach does not apply. Nevertheless we are able to prove a denominator formula by replacing  in \eqref{e1} $q$ by $q^2$ and then setting $x=qz,\,y=q$. We obtain
$$
\prod_{n\ge1}\frac{(1-q^{2n})^2(1-z^{-1}q^{2n-2})(1-zq^{2n})}{(1+q^{2n-1})^2(1+zq^{2n-1})(1+z^{-1}q^{2n-1})}=
\left(\sum_{m,n=0}^\infty-\sum_{m,n=-1}^{-\infty}\right)(-1)^{m+n}z^mq^{2mn+m+n},
$$
or, after replacing $q$ by $q^{\frac{1}{2}}$ and $z$ by $z^{-1}$,
\begin{align*}
&\frac{\varphi(q)^2\vartheta_0(z)}{\vartheta_1(z)}=\varphi_1(q)\left(\sum_{m,n=0}^\infty-\sum_{m,n=-1}^{-\infty}\right)(-1)^{m+n}z^{-m}q^{mn+\frac{1}{2}(m+n)}.
\end{align*}
which,  setting $z=e^{-\e_1}=e^{\a_1{}_{|\h^\natural}}$, is the denominator identity for $W^{k}_{\min}(spo(2|3))$ in NS sector.

In the Ramond sector we rewrite \eqref{e1} as
$$
\prod_{n\ge1}\frac{(1-q^n)^2(1-x^{-1}y^{-1}q^{n-1})(1-xyq^n)}{(1+x^{-1}q^{n-1})(1+xq^{n})(1+y^{-1}q^{n-1})(1+yq^{n})}
=\left(\sum_{n=1}^\infty(-1)^{n}y^{-n}+\sum_{m=0}^{\infty}(-1)^{m}x^{-m}\right)
$$
$$
+\left(\sum_{m,n=1}^\infty-\sum_{m,n=-1}^{-\infty}\right)(-1)^{m+n}x^{-m}y^{-n}q^{mn}
$$
which, in $\mathcal R(\Pi^R)$, is equivalent to
\begin{align}
&\prod_{n\ge1}\frac{(1-q^n)^2(1-x^{-1}y^{-1}q^{n-1})(1-xyq^n)}{(1+x^{-1}q^{n})(1+xq^{n})(1+y^{-1}q^{n-1})(1+yq^{n})}
=1+\left(\sum_{n=1}^\infty(-1)^{n}(1+x)y^{-n}\right)\notag
\\
&+(1+x^{-1})\left(\sum_{m,n=1}^\infty-\sum_{m,n=-1}^{-\infty}\right)(-1)^{m+n}x^{-m}y^{-n}q^{mn}.
\end{align}
We note that we can specialize $x=1$ in both sides so  we obtain the identity (setting $z=y^{-1}$)
\begin{align*}
&\frac{\varphi(q)^2\vartheta_0(z)}{\varphi_2(q)\vartheta_0(-z)}=\varphi_2(q)\left(1+2\sum_{n=1}^\infty(-1)^{n}z^{n}+2\left(\sum_{m,n=1}^\infty-\sum_{m,n=-1}^{-\infty}\right)(-1)^{m+n}z^nq^{mn}\right).
\end{align*}

\subsection{$\g=D(2,1;1)=spo(2|4), k_0=-\half$}
  In the NS sector, \eqref{detNS}  gives, letting  $x=e^{\theta_1/2}, y=e^{\theta_2/2}$: 
{\small
\begin{align*}&\frac{\varphi(q)^3
   \vartheta_0(x^{-2})\vartheta_0(y^{-2}) }
   {\vartheta_1(xy)\vartheta_1(xy^{-1})}=\\
   &\sum_{m,n\in\ZZ}\!\!\!\left(\frac{x^{2 m} y^{2
   n}}{1+x y
   q^{m+n+\frac{1}{2}}}-\frac{x^{-2 m-2} y^{2 n}}{1+x^{-1} y
   q^{m+n+\frac{1}{2}}}-\frac{x^{2 m} y^{-2 n-2}}{1+x y^{-1}
   q^{m+n+\frac{1}{2}}}+\frac{x^{-2 m-2} y^{-2 n-2}}{1+x^{-1}y^{-1} q^{m+n+\frac{1}{2}}}\right)q^{m^2+n^2+m+n}.
 \end{align*}}
 
 In the Ramond sector, we have two choices for $\eta_{\min}=\pm(\theta_1/2-\theta_2/2)$.  Choosing the $+$ sign,   \eqref{detR} gives, letting $e^{\theta_1/2}=x,\,e^{\theta_2/2}=y$:
  \begin{align*}&\frac{\varphi(q)^3
   \vartheta_0(x^{-2})\vartheta_0(y^{-2})}{
  \vartheta_0(-y^{-1}x^{-1})\vartheta_0(-x^{-1}y)}=\\
&\sum_{m,n\in\ZZ}\left(\frac{x^{-2m}y^{2n}}{1+x^{-1}yq^{n+m}}-\frac{x^{2m}y^{2n}}{1+xyq^{n+m}}
-\frac{x^{-2m} y^{-2n-2}}{1+x^{-1}y^{-1}q^{n+m}}+\frac{x^{2m}y^{-2n-2}}{1+xy^{-1}q^{n+m}}\right)q^{m^2+n^2+n}.\end{align*}
\subsection{$\g=spo(2|2r)$, $r>2$, $k_0=-\half$} Set $y_i=e^{\e_i}$, $i=1,\ldots,r$. 
Recall that in this case $W^{\natural}$ is the subgroup of $\{\pm 1\}^r\rtimes \mathfrak S_r$ consisting of elements $(i_1,\ldots,i_r)\sigma$ with an even number of $-1$ in $(i_1,\ldots,i_r)$. Moreover, 
$$M^\natural=\left\{\sum_{i=1}^r m_i\e_i\mid m_i\in\Z,\,1\le i\le r,\sum_{i=1}^r m_i\in 2\ZZ\right\}.$$
The denominator identity \eqref{detNS} is 
{
\begin{align*}
& \frac{\varphi(q)^{r+1}\prod_{1\le i<j\le r}
      \vartheta_0(y_i^{-1}y_j) \vartheta_0(y_i^{-1}y_j^{-1}) }{\prod_{1\le i\le r}
 \vartheta_1(y_i) }\\
&=
\prod_{i=1}^ry_i^{i-r}\sum\limits_{\bar w\in W^\natural}\sum_{m_1,\ldots,m_r}\det(\bar w)\frac{\bar w(\prod_{i=1}^ry_i^{(2r-3)m_i+r-i})}{1+\bar w(y_1) q^{m_1+\frac{1}{2}}}q^{(r-\tfrac{3}{2})\sum_im^2_i+\sum_i(r-i)m_i}.
\end{align*}}
Choosing $\eta_{\min}=\e_r$, the denominator identity \eqref{detR} is 
{
\begin{align*}
&\prod^{\infty}_{n=1} \frac{\varphi(q)^{r+1}\prod_{1\le i<j\le r}
      \vartheta_0(y_i^{-1}y_j) \vartheta_0(y_i^{-1}y_j^{-1}) }{\prod_{1\le i\le r}
\vartheta_0(-y_i^{-1})}\\&=
\half\prod_{i=1}^ry_i^{\tfrac{1}{2}+i-r}\sum\limits_{\bar w\in W^\natural}\sum_{m_1,\ldots,m_r}\det(\bar w)\frac{\bar w(\prod_{i=1}^ry_i^{(2r-3)m_i+r-i-\tfrac{1}{2}})}{1+\bar w(y_r^{-1}) q^{-m_r}}q^{(r-\tfrac{3}{2})\sum_im^2_i+\sum_i(r-i-\half)m_i}.
\end{align*}}
\subsection{$\g=spo(2|2r+1)$, $r\ge2$, $k_0=-\half$} Set $y_i=e^{\e_i}$, $i=1,\ldots,r$. 
Recall that in this case $W^{\natural}\cong \{\pm 1\}^r\rtimes \mathfrak S_r$; if the isomorphism is $\bar w \leftrightarrow (i_1,\ldots,i_r)\sigma,\,i_j\in \{\pm 1\}$, the action of $\sigma$ on the $y_i$ is just the permutation action, whereas $(i_1,\ldots,i_n)(y_j)=y_j^{i_j}$. The lattice $M^\natural$ is the same as in the even case.
The denominator identity \eqref{detNS} is 
{
\begin{align*}
&\frac{\varphi(q)^{r+1}\prod_{1\le i<j\le r}
       \vartheta_0(y_i^{-1}y_j)  \vartheta_0(y_i^{-1}y_j^{-1})\prod_{1\le i\le r}\vartheta_0(y_i^{-1})
    }{\varphi_1(q)\prod_{1\le i\le r}
 \vartheta_1(y_i) }\\&=
\prod_{i=1}^ry_i^{\frac{2i-2r-1}{2}}\sum\limits_{\bar w\in W^\natural}\sum_{m_1,\ldots,m_r}\det(\bar w)\frac{\bar w(\prod_{i=1}^ry_i^{2(r-1)m_i+\frac{2r+1-2i}{2}})}{1+\bar w(y_1) q^{m_1+\frac{1}{2}}}q^{(r-1)\sum_im^2_i+\sum_i(r-i)m_i}.
\end{align*}}
The denominator identity \eqref{detR} is 
{
\begin{align*}
&\frac{\varphi(q)^{r+1}\prod_{1\le i<j\le r}
       \vartheta_0(y_i^{-1}y_j)  \vartheta_0(y_i^{-1}y_j^{-1})\prod_{1\le i\le r}\vartheta_0(y_i^{-1}) }{\varphi_2(q)\prod_{1\le i\le r}
  \vartheta_0(-y_1^{-1})}\\&=
\prod_{i=1}^ry_i^{i-r}\sum\limits_{\bar w\in W^\natural}\sum_{m_1,\ldots,m_r}\det(\bar w)\frac{\bar w(\prod_{i=1}^ry_i^{2(r-1)m_i+r-i})}{1+\bar w(y_r^{-1}) q^{-m_r}}q^{(r-1)\sum_im^2_i+\sum_i (r-i+\half)m_i}.
\end{align*}}
\subsection{$\g=F(4), k_0=-\tfrac{2}{3}$}
Set $y_i=e^{\e_i/2}$, $i=1,2,3$. Since $\g^\natural=so(7)$ we identify $W^\natural$ with $\{\pm1\}^3\rtimes \mathfrak S_3$ as in the $spo(2|2r+1)$ case. 

  In the NS sector, \eqref{detNS} reads
\begin{align*}
& \frac{\varphi(q)^{4} \prod\limits_{1\le i\le 3}\vartheta_0 (y_i^{-2})\prod\limits_{1\le i\le j \le 3}\vartheta_0(y_i^{-2}y_j^{-2})\prod\limits_{1\le i\le j \le 3}\vartheta_0(y_i^{-2}y_j^2)}{\vartheta_1(y_1y_2y_3)\vartheta_1(y_1y_2y_3^{-1})\vartheta_1(y_1y_2^{-1}y_3)
\vartheta_1(y_1^{-1}y_2y_3)}\\
&= y_1^{-5}y_2^{-3}y_3^{-1}\!\!\!\!\!\!\!\!\!\!\!\!\!\!\sum_{\substack{m,r,t\in\ZZ \\m+t+r\equiv 0 \mod 2}}\!\!\!\!\!\!\!\!\!\!\!\!q^{2 m^2+2r^2+2 t^2+\frac{5m+3r+t}{2}}\left(\sum_{\bar w\in W^\natural}\det(\bar w)
\tfrac{\bar w(y_1^{8 m+5} y_2^{8 r+3} y_3^{8 t+1})}{1+q^{b_{m,r,t}} \bar w(y_1 y_2 y_3)}\right),
\end{align*}
where $b_{m,r,t}=(m+t+r+1)/2$.  In the Ramond sector, we choose $\eta_{\min}=\half(-\e_1+\e_2+\e_3)$. Then \eqref{detR} becomes
\begin{align*}
&\frac{\varphi(q)^{4} \prod\limits_{1\le i\le 3}\vartheta_0 (y_i^{-2})\prod\limits_{1\le i\le j \le 3}\vartheta_0(y_i^{-2}y_j^{-2})\prod\limits_{1\le i\le j \le 3}\vartheta_0(y_i^{-2}y_j^2)}{\vartheta_0(-y_1^{-1}y_2^{-1}y_3^{-1})\vartheta_0(-y_1^{-1}y_2^{-1}y_3)\vartheta_0(-y_1^{-1}y_2y_3^{-1})\vartheta_0(-y_1y_2^{-1}y_3^{-1})}\\
&=y_1^{-4}y_2^{-2}\!\!\!\!\!\!\!\!\!\!\!\!\!\!\! \sum_{\substack{m,r,t\in\ZZ \\m+t+r\equiv 0 \mod 2}}\!\!\!\!\!\!\!\!\!\!\!\!q^{2 m^2 +  2 r^2 + 2 t^2+2 m +r }\left(\sum_{\bar w\in W^\natural}\det(\bar w)
\tfrac{\bar w(y_1^{8 m+4} y_2^{8 r+2} y_3^{8 t})}{1+q^{\bar b_{m,r,t}} \bar w(y_1 y_2^{-1} y_3^{-1})}\right),
\end{align*}
where $\bar b_{m,r,t}= (m - r - t)/2$. 
\subsection{$\g=G(3), k_0=-\tfrac{3}{4}$} Set $y_i=e^{\e_i}$, $i=1,2$. Since $\g^\natural=G_2$, $W^\natural$ is the dihedral group of order $12$ with Coxeter generators $s_1$, $s_2$ acting as 
$$
s_1(y_1)=y_1^{-1},\ s_2(y_1)=y_2,\ s_1(y_2)=y_1y_2,\ s_2(y_2)=y_1.
$$ 

  In the NS sector, \eqref{detNS} becomes
{
\begin{align*}
& \frac{\varphi(q)^{3}  \vartheta_0(y_1^{-1})\vartheta_0(y_2^{-1})\vartheta_0(y_1y_2^{-1})\vartheta_0(y_1^{-1}y_2^{-1})\vartheta_0(y_1^{-2}y_2^{-1})\vartheta_0( y_1^{-1}y_2^{-2}) }
  {\varphi_1(q)\vartheta_1(y_1)\vartheta_1(y_2)\vartheta_1(y_1y_2)}\\
 &=
y_1^{-2}y_2^{-3}\sum_{\substack{m,n\in\ZZ\\ m+n	\equiv 0\mod 3}}
q^{ m^2+n^2 +\frac{m-3mn+4n}{3}}\left(\sum_{\bar w\in W^\natural}\det(\bar w)
\frac{\bar w(y_1^{3m+2}y_2^{3n+3})}{1+\bar w(y_1y_2)q^{a_{m,n}}}\right),
\end{align*}
 where $a_{m,n}=\tfrac{m+n}{3}+\frac{1}{2}$. In the Ramond sector, \eqref{detR} becomes
\begin{align*}
&\frac{\varphi(q)^{3}  \vartheta_0(y_1^{-1})\vartheta_0(y_2^{-1})\vartheta_0(y_1y_2^{-1})\vartheta_0(y_1^{-1}y_2^{-1})\vartheta_0(y_1^{-2}y_2^{-1})\vartheta_0( y_1^{-1}y_2^{-2})}
  {\varphi_2(q)
 \vartheta_0(-y_1^{-1})   \vartheta_0(-y_2^{-1}) \vartheta_0(-y_1^{-1}y_2^{-1})}=
 \\
&y_1^{-1}y_2^{-2}\!\!\!\!\!\!\!\!\!\!\sum_{\substack{m,n\in\ZZ\\ m+n	\equiv 0 \mod 3}}
\left(\sum_{\bar w\in W^\natural}\det(\bar w)
 \tfrac{ \bar w(y_1^{ 3 m+1 } y_2^{ 3 m+2})}{ 1 + \bar w((y_1 y_2)^{-1}) q^{\bar a_{m,n}}}
\right)q^{ m^2+n^2 +m-mn},\end{align*}
where $\bar a_{m,n}=\frac{n-2 m}{3}$.

\section{Appendix. Denominator identity for minimal $W$-algebras of Deligne series}
Let $\g$ be the simple Lie algebra $D_4$, $E_6$, $E_7$, or $E_8$, and let $a=\tfrac{h^\vee}{6}+1$. Then, for any integer $j$ such that $0<j<a$, there exist  unique simple roots $\a_1,\ldots,\a_j$ such that $\theta-\sum_{i=1}^j\a_i$ is a root.
Set 
$\a=\theta-\sum_{i=1}^{a-1} \a_i$. Then $(\rho|\a)=h^\vee-a$. Let $k_0=-a$, then $(k_0\L_0+\widehat \rho|\a)=k_0+h^\vee=b=\tfrac{h^\vee+\bar h^\vee}{2}$ (in our cases $b=4,9,14,24$ respectively).
By \cite[Theorem 7.2]{AM} or \cite[Proposition 3.4]{AKMPP}, $\dim W^{\min}_{k_0}(\g)=1$.  A character formula, for certain $\g$-modules $L(\L)$ of negative integer level $k\ge k_0$, has been conjectured in \cite[(3.1)]{KW18} and proved in \cite{BKK}, formulas (5) and (6). A special case of this formula is 
\begin{equation}\label{chBKK} \widehat R\,ch\,L(k_0\L_0)=\half \sum_{w\in W}\det(w)\sum_{\gamma \in Q} ((\gamma|\a)+1)e^{wt_\gamma (k_0\L_0+\widehat\rho)-\widehat\rho},\end{equation}
where $Q$ is the root lattice of $\g$ and $W$ is  its Weyl group.  Let
$$
\widehat F^{NS}= \prod^{\infty}_{n=1}(1-q^n)^{\dim \h}\prod_{\alpha \in \Delta^\natural_{+}}
        (1-q^{n-1} e^{-\alpha}) (1-q^{n} e^{\alpha})\prod_{\beta\in\D_{\frac{1}{2}}}
        (1-q^{\frac{1}{2}+n} e^{\beta_{|\h^\natural}}).
$$
be the denominator for the $W$-algebra $\Wu$.

  Recall that any element $w\in W$ can be uniquely written as  $w=w^\natural \bar w $, where $w^\natural\in W^\natural$ and $\bar w$ is a right coset representative of $W^\natural$ in $W$ of minimal length,  and that $\rho-\bar w(\rho)=\sum_{\eta \in N(\bar w)}\eta$, where $N(\bar w)=\{\eta\in\Dp\mid -\bar w^{-1}( \eta)\in\Dp\}$. 
If $\bar w^{-1}(\theta)\in\D^+$ then $(\eta|\theta)=1$ for each $\eta\in N(\bar w)$, so $(\rho-\bar w(\rho)|\theta)=\ell(\bar w)$.
Observe that the map $\bar w\mapsto \bar w^{-1}(\theta)$ is a bijection between $W^\natural\backslash W$ and $\D$. Let $\eta\mapsto \bar w_{\eta}$ be its inverse.

\begin{theorem} 
\begin{align}\label{finalissimo}
&\widehat F^{NS}=\\&\tfrac{e^{-\rho^\natural}}{2}\sum_{\substack{\eta\in\Dp\\\gamma \in Q\\w^\natural\in W^\natural }}\det(\bar w_{\eta} w^\natural)
 (\gamma|\a) e^{w^\natural \bar w_\eta(b\gamma+\rho)_{|\h^\natural}}q^{(\rho|\gamma)+\tfrac{b(\gamma|\gamma)+h^\vee-1}{2}}
 \left(q^{\tfrac{c_{\gamma, \eta}}{2}}+q^{-\tfrac{c_{\gamma, \eta}}{2}}\right),\notag
\end{align}
where $c_{\gamma, \eta}=b(\gamma|\eta)-\ell(\bar w_\eta)+h^\vee-1$. \end{theorem}
\begin{proof} Since $H_0(L(k_0\L_0))=W^{\min}_{k_0}(\g)=\C\vac$, from \eqref{chBKK} we obtain
\begin{equation}\label{detNSS}
\widehat F^{NS}=\half\sum_{w\in W}\det(w)\sum_{\gamma \in Q} ((\gamma|\a)+1) e^{ev(wt_{\gamma}(k_0\L_0+\rhat)- \rhat)}
\end{equation}
Note that the coefficient $(\gamma|\alpha)+1$ can be replaced  by $(\gamma|\alpha)$ since
the  term corresponding to 1 in the R.H.S. of  \eqref{chBKK}, multiplied by
$2e^{\rhat}$, is
$\sum_{w\in \Wa} \det(w) e^{w(k_0\Lambda_0+\rhat)}$, which is $0$ since
$k_0\Lambda_0+\rhat$ is orthogonal to $\delta-\alpha$.\par
By the definition \eqref{evaluation}  of  $ev$ we have
\begin{equation}\label{143}ev(wt_{\gamma}(k_0\L_0+\rhat)- \rhat))=((wt_{\gamma}(k_0\L_0+\rhat)- \rhat)(-x-D),(wt_{\gamma}(k_0\L_0+\rhat)- \rhat)_{|\h^\natural}).\end{equation}
 To compute the R.H.S. of \eqref{143},  use formula \eqref{trasl}, noting that $u_i=2$ for simply laced Lie algebras:
$$wt_\gamma(k_0\L_0+\rhat)-\rhat=k_0\L_0+w(b\gamma+\rho)-\rho-((\rho|\gamma)+b\tfrac{(\gamma|\gamma)}{2})\d,$$
so \eqref{detNSS} becomes
$$
\widehat F^{NS}=\half\sum_{w\in W}\det(w)\sum_{\gamma \in Q} (\gamma|\a) e^{(w(b\gamma+\rho)-\rho)_{|\h^\natural}}q^{(\rho|\gamma)+b\tfrac{(\gamma|\gamma)}{2}}q^{(w(b\gamma+\rho)-\rho)(-x-D)}
$$
To compute 
$$(w(b\gamma+\rho)-\rho)(-x-D)=(w(b\gamma+\rho)-\rho)(-x),$$ observe  that $\bar w_{-\eta}$ and $s_\theta\bar w_{\eta}$ are in the same right coset mod $W^\natural$, so the set $\{\bar w_{\eta}\mid \eta\in\Dp\}\cup\{s_\theta\bar w_{\eta}\mid \eta\in\Dp\}$ is a set of right coset representatives for $W^\natural\backslash W$. If $   w=w^\natural \bar w_\eta $, $\eta\in\Dp$, then
\begin{align}\label{funz1}&(w(b\gamma+\rho)-\rho)(-x)=(w(b\gamma+\rho)-\rho|-\theta/2)=(\bar w_\eta(b\gamma+\rho)-\rho)|-\theta/2)\\
&=-b(\bar w_\eta\gamma|\theta/2)+(\rho-\bar w_\eta(\rho)|\theta/2)=-\tfrac{b}{2}(\gamma|\bar w_\eta^{-1}(\theta))+\half(\rho-\bar w_\eta(\rho)|\theta).\notag\end{align}
On the other hand, if $w=w^\natural s_\theta\bar w_\eta$, then
\begin{align}\label{funz2}&(w(b\gamma+\rho)-\rho)(-x)=(w(b\gamma+\rho)-\rho|-\theta/2)=(\bar w_\eta(b\gamma+\rho)+\rho)|\theta/2)\\
&=b(\bar w_\eta\gamma|\theta/2)+(\rho|\theta)+(\bar w_\eta(\rho)-\rho|\theta/2)=\tfrac{b}{2}(\gamma|\bar w_\eta^{-1}(\theta))+(\rho|\theta)-\half(\rho-\bar w_\eta(\rho)|\theta).\notag\end{align}
Since  $w^\natural s_\theta\bar w_{\eta}(b\gamma+\rho)_{|\h^\natural}=s_\theta w^\natural \bar w_{\eta}(b\gamma+\rho)_{|\h^\natural}=w^\natural \bar w_{\eta}(b\gamma+\rho)_{|\h^\natural}$, 
plugging \eqref{funz1} and \eqref{funz2} into \eqref{detNSS}, by the discussion preceding the statement of the theorem, we obtain 
\begin{align*}
\widehat F^{NS}&=\half\sum_{\eta\in\Dp}\det(\bar w_{\eta})\sum_{\gamma \in Q} (\gamma|\a) \sum_{w^\natural\in W^\natural}det(w^\natural) e^{w^\natural \bar w_\eta(b\gamma+\rho)_{|\h^\natural}- \rho_{|\h^\natural}}q^{(\rho|\gamma)+\tfrac{b(\gamma|\gamma-\eta)+\ell(\bar w_\eta)}{2}}\\
&
+\half\sum_{\eta\in\Dp}\det(\bar w_{\eta})\sum_{\gamma \in Q} (\gamma|\a) \sum_{w^\natural\in W^\natural}det(w^\natural) e^{w^\natural \bar w_{\eta}(b\gamma+\rho)_{|\h^\natural}- \rho_{|\h^\natural}}q^{(\rho|\gamma+\theta)+\tfrac{b(\gamma|\gamma+\eta)-\ell(\bar w_{\eta})}{2}}.\end{align*}
which, observing that $\rho_{|\h^\natural}=\rho^\natural$, is \eqref{finalissimo}.
\end{proof}

 \subsection*{Acknowledgements}
The authors are grateful to T. Arakawa for correspondence and to D. Adamovi\'c for useful discussions. P.P. wishes to thank C. Krattenthaler for useful discussions. P. M-F. wishes to thank Indam for its hospitality during the workshop held in Rome from 11 to 15 December 2023.
P.M-F. and P.P. are partially supported by the PRIN project 2022S8SSW2 - Algebraic and geometric aspects of Lie theory - CUP B53D2300942 0006, a project cofinanced
by European Union - Next Generation EU fund.
V.K. is partially supported by the Simons Collaboration grant.


  \end{document}